\documentclass{amsart}

\usepackage{amsmath,amssymb,latexsym,color}
\usepackage[bookmarks=true,colorlinks=true, linkcolor=blue, citecolor=cyan]{hyperref}
\usepackage{tikz}
\usetikzlibrary{arrows}
\usetikzlibrary{fit}
\usepackage[margin=1in,marginparwidth=0.8in, marginparsep=0.1in]{geometry}
\usepackage{extarrows}
\usepackage{verbatim}
\input xy
\xyoption{all}

\newtheorem{theorem}{Theorem}[section]

\newtheorem{conjecture}[theorem]{Conjecture}
\newtheorem{corollary}[theorem]{Corollary}
\newtheorem{definition}[theorem]{Definition}
\newtheorem{lemma}[theorem]{Lemma}

\newtheorem{proposition}[theorem]{Proposition}
\newtheorem{remark}[theorem]{Remark}
\newtheorem{example}[theorem]{Example}
\newtheorem{thm}{Theorem}
\numberwithin{equation}{section}

\renewcommand{\AA}{\mathbb{A}}
\newcommand{\C}{\mathbb{C}}
\newcommand{\CC}{\mathbb{C}}

\newcommand{\NN}{\mathbb{N}}
\newcommand{\ZZ}{\mathbb{Z}}

\newcommand{\bfe}{\mathbf{e}}
\newcommand{\bff}{\mathbf{f}}
\newcommand{\bfg}{\mathbf{g}}

\newcommand{\bfi}{\mathbf{i}}
\newcommand{\bfI}{\mathbf{I}}
\newcommand{\bfJ}{\mathbf{J}}
\newcommand{\bfs}{\mathbf{s}}
\newcommand{\bft}{\mathbf{t}}

\newcommand{\tbfe}{{\tilde\bfe}}
\newcommand{\tbff}{{\tilde\bff}}
\newcommand{\tbfg}{{\tilde\bfg}}

\newcommand{\tbfs}{{\tilde\bfs}}
\newcommand{\tbft}{{\tilde\bft}}

\newcommand{\cA}{\mathcal{A}}
\newcommand{\cB}{\mathcal{B}}
\newcommand{\cC}{\mathcal{C}}
\newcommand{\cG}{\mathcal{G}}
\newcommand{\cH}{\mathcal{H}}
\newcommand{\cP}{\mathcal{P}}
\newcommand{\cQ}{\mathcal{Q}}

\newcommand{\ui}{{\underline i}}

\newcommand\udim{{\underline{\dim}\, }}

\newcommand{\into}{\hookrightarrow}
\newcommand{\onto}{\to\!\!\!\!\!\to}

\newcommand{\Att}{\operatorname{Att}}
\newcommand{\codim}{\operatorname{codim}}
\newcommand{\Ext}{\operatorname{Ext}}
\newcommand{\Gr}{\mathrm{Gr}}
\newcommand{\GL}{\mathrm{GL}}
\newcommand{\Hom}{\operatorname{Hom}}
\renewcommand{\Im}{\operatorname{Im}}
\newcommand{\Ind}{\mathrm{Ind}}
\newcommand{\pt}{\mathrm{pt}}

\newcommand{\rep}{\operatorname{rep}}
\newcommand{\rk}{\mathrm{rk}}

\newcommand{\ses}[3]{\xymatrix@C15pt{0\ar[r] & #1\ar[r] & #2\ar[r] & #3 \ar[r] & 0}}
\newcommand{\sesm}[4]{\xymatrix{0\ar[r] & #1\ar[r] & #2\ar^{#4}[r] & #3 \ar[r] & 0}}

\newcommand{\supp}{\operatorname{supp}}
\newcommand{\vs}{\vspace{0.2cm}}

\title{Cell Decompositions for Rank Two Quiver Grassmannians}

\author{Dylan Rupel}
\address[Dylan Rupel]{University of Notre Dame, Department of Mathematics, Notre Dame, IN 46556, USA}
\email{drupel@nd.edu}
\author{Thorsten Weist}
\address[Thorsten Weist]{Bergische Universit\"at Wuppertal, Gau\ss str.\ 20, 42097 Wuppertal, Germany}
\email{weist@uni-wuppertal.de}

\begin{document}
\begin{abstract}
  We prove that all quiver Grassmannians for exceptional representations of a generalized Kronecker quiver admit a cell decomposition.  
  In the process, we introduce a class of regular representations which arise as quotients of consecutive preprojective representations.
  Cell decompositions for quiver Grassmannians of these ``truncated preprojectives'' are also established. 
  We also provide two natural combinatorial labelings for these cells.
  On the one hand, they are labeled by certain subsets of a so-called $2$-quiver attached to a (truncated) preprojective representation.
  On the other hand, the cells are in bijection with compatible pairs in a maximal Dyck path as predicted by the theory of cluster algebras.
  The natural bijection between these two labelings gives a geometric explanation for the appearance of Dyck path combinatorics in the theory of quiver Grassmannians. 
\end{abstract}

\setcounter{tocdepth}{2}

\maketitle

\tableofcontents
\section{Introduction}
\noindent A quiver Grassmannian is a projective variety attached to a fixed quiver representation which parametrizes subrepresentations of a fixed dimension vector.
In recent years, interest in quiver Grassmannians has grown considerably.
On the one hand, this is due to the fact that generating functions for the Euler characteristics of quiver Grassmannians of exceptional representations can be found as cluster variables \cite{ck}.
On the other hand, they are clearly interesting on their own as they reveal many properties of the representation and its geometry.

Although it follows from the results of Hille, Huisgen-Zimmermann and Reineke that every projective variety can be realized as a quiver Grassmannian, 
it turns out that very interesting phenomena arise when restricting to certain quivers or to representations with certain properties.
For instance, quiver Grassmannians attached to exceptional representations are smooth~\cite{cr}.
For Dynkin quivers and tame quivers of types~$\tilde A$ or~$\tilde D$, it is known that every quiver Grassmannian attached to an indecomposable representation admits a cell decomposition, see \cite{ce,lw} and references therein.
It has been conjectured that this is also true for exceptional representations of any quiver, in particular for preprojective and preinjective representations. 

There are basically two possible ways to find cell decompositions of quiver Grassmannians if they exist.
One is to find a non-trivial $\CC^*$-action on the quiver Grassmannian under consideration.
If the quiver Grassmannian is smooth, one can apply a result of Bia\l{}ynicki-Birula~\cite{bb} which shows that the quiver Grassmannian decomposes into affine bundles over the fixed point components.
In particular, this shows that the quiver Grassmannian has a cell decomposition if the fixed point components have such a decomposition.

Another method uses short exact sequences of quiver representations to induce maps between quiver Grassmannians.
More precisely, the quiver Grassmannian of the middle term maps to the product of the quiver Grassmannians for the two outer terms via the ``Caldero-Chapoton map'' which first appeared in \cite{cc}. 
If the short exact sequence has certain properties -- e.g. (almost) split sequences and certain generalizations -- then cell decompositions of quiver Grassmannians attached to the outer terms transfer to cell decompositions for the quiver Grassmannians of the middle term.

In this paper, we combine these two methods in order to show that every quiver Grassmannian attached to an exceptional representation of a generalized Kronecker quiver admits a cell decomposition.
The proof also shows that this is true for so-called truncated preprojective representations which appear as certain quotients of preprojective representations.
It turns out that these are precisely those representations which can be obtained from indecomposable representations with dimension vector $(d,1)$ when applying reflection functors.
Actually, we prove that quiver Grassmannians of truncated preprojective representations only depend on the dimension vector of the representation itself and on the fixed dimension vector of the subrepresentations.

As a first step, we show that a $\CC^*$-action with proper fixed point set can be defined on any quiver Grassmannian attached to a liftable representation of any acyclic quiver containing parallel arrows or non-oriented cycles, that is for representations which can be lifted to the universal (abelian) covering quiver.
These are precisely those cases where the universal covering quiver differs from the original quiver.
This lifting property holds in particular for so-called tree modules, a class of representations which includes all exceptional representations.
The fixed point set of this $\CC^*$-action consists precisely of those subrepresentations which can also be lifted to the universal abelian covering quiver.
Actually, it turns out that each fixed point component is itself a quiver Grassmannian attached to the lifted representation and thus, iterating this procedure, it suffices to understand the quiver Grassmannians for the universal covering quiver.

The next step is to investigate conditions under which the iterated fixed point components admit a cell decomposition.
Here the Caldero-Chapoton map comes into play.
In the case of the generalized Kronecker quiver, it turns out that a natural filtration of a fixed preprojective representation by preprojectives of smaller dimension transfers to the universal covering quiver.
These filtrations can be successively described by short exact sequences.
The main advantage when passing to the universal covering is that the preprojective representations covering the same dimension vector below become orthogonal, a property which rigidifies the situation in a sense.
In the end, this machinery can be used to recursively build cell decompositions of all quiver Grassmannians of lifted (truncated) preprojective representations.
As all the quiver Grassmannians of the (non-lifted) representation are smooth, this combines with the iterated torus actions on fixed point components to give a cell decomposition of these quiver Grassmannians.

As a benefit of this construction, we obtain a graph theoretic description of the non-empty cells.
More precisely, with every (truncated) preprojective representation we can associate a so-called $2$-quiver.
Essentially, such a quiver is obtained from a usual quiver by adding a collection of ``$2$-arrows'' between pairs of subquivers.
Now with every subset of the vertices we can associate a dimension vector.
If this subset is also strong successor closed, a condition which is easily verified in practice, it corresponds to a cell and vice versa.

As mentioned above, the Laurent polynomial expressions for cluster variables have been described using the representation theory of quivers \cite{cc,ck}: the cluster variables are generating functions for Euler characteristics of quiver Grassmannians.
For rank two cluster algebras, the Laurent expressions of cluster variables can also be computed using a certain Dyck path combinatorics \cite{llz}.
The confluence of these results gives rise to a combinatorial construction for the Euler characteristics and counting polynomials of certain quiver Grassmannians \cite{rupel}.
A consequence of our main result is a geometric explanation for these computations: we provide a one-to-one correspondence between the strong successor closed subsets and compatible pairs for an appropriate Dyck path which leads to a geometric explanation for the appearance of Dyck path combinatorics in the theory of quiver Grassmannians. 

The paper is organized as follows.
In Section~\ref{sec:covering}, we collect several results concerning quiver covering theory.
In Section~\ref{sec:RepK(n)}, we recall basic facts concerning the representation theory of generalized Kronecker quivers $K(n)$ which are needed later to investigate the quiver Grassmannians attached to preprojective representations.
We first focus on preprojective and preinjective representations, written as $P_m$ and $I_m$, which enables us to investigate a special class of indecomposable representations in Section~\ref{sec:truncated preprojectives} -- we call them truncated preprojective representations.
We prove that every preprojective representation admits a filtration by preprojectives of smaller dimensions such that all quotients appearing are actually truncated preprojective representations.
In Section~\ref{Lifting}, we use this together with the fact that every preprojective representation can be lifted to the universal covering in order to construct lifted filtrations.
Throughout this section, we collect many results which will turn out to completely reveal the structure of quiver Grassmannians attached to (truncated) preprojective representations.

The aim of Section~\ref{QG} is to study these quiver Grassmannians and to show that they each admit a cell decomposition.
This is obtained in Section~\ref{sec:fibrations} by combining iterated $\CC^*$-actions on quiver Grassmannians, which are introduced in Section~\ref{torusaction}, with the Caldero-Chapoton map for short exact sequences of representations.
Our first main result is Theorem~\ref{thm:torusfixedpoints} which may be formulated as follows.
\begin{thm}
  Let $X$ be a representation of a quiver $Q=(Q_0,Q_1)$ which can be lifted to a representation $\hat X$ of the universal abelian covering quiver $\hat Q=(Q_0\times A_Q,Q_1\times A_Q)$, where $A_Q$ is the free abelian group generated by $Q_1$.
  Then there exists a map $d:\supp(\hat X)\to\ZZ$ -- with a corresponding $\CC^*$-action on every $X_i=\bigoplus_{\chi\in A_{Q}} X_{(i,\chi)}$ defined by $t.x_{(i,\chi)}=t^{d(i,\chi)}x_{(i,\chi)}$ for $x_{(i,\chi)}\in X_{(i,\chi)}$ -- which induces a $\CC^*$-action on $\Gr_\bfe^Q(X)$ such that
  \[\Gr^Q_\bfe(X)^{\CC^*}\cong\bigsqcup_{\hat\bfe} \Gr^{\hat Q}_{\hat\bfe}(\hat X),\]
  where $\hat\bfe$ runs through all dimension vectors compatible with $\bfe$.
\end{thm}
This $\CC^*$-action can be iterated in such a way that the remaining $\CC^*$-fixed points are precisely the subrepresentations which can be lifted to the universal covering quiver.
As far as generalized Kronecker quivers are concerned, we can show in Theorem~\ref{thm:truncpp} that all quiver Grassmannians attached to truncated preprojective representations are smooth -- actually, they only depend on appropriate dimension vectors.
In view of results of Bia\l{}ynicki-Birula~\cite{bb} -- which roughly speaking yields that cell decompositions are preserved when passing from the fixed point components to the original variety -- we can use this result to lift the investigation of the geometry of quiver Grassmannians to the universal covering quiver.
This is important insofar as results such as Corollary~\ref{cor:perpendicular} are available which do not hold on the original quiver.
Analyzing the Caldero-Chapoton map applied to short exact sequences induced by lifts of the mentioned filtrations in greater detail, and combining it with the torus method, we obtain the main result of this paper, see Theorems~\ref{cellscover} and~\ref{celldec}.
\begin{thm} 
  For every $m\geq 1$ and for every point $V\neq \C^n$ of the total Grassmannian $\Gr(\C^n)$, there exists a (truncated) preprojective representation $P_{m+1}^V$ such that every quiver Grassmannian $\Gr_\bfe(P_{m+1}^V)$ admits a cell decomposition.
\end{thm}
\noindent Note that, for $V=0$, we obtain the preprojective representations $P_{m+1}$.

In Section \ref{sec:combinatorics}, we reveal the combinatorics behind the obtained cell decompositions by introducing the notion of $2$-quivers which are a slight generalization of the usual notion of quivers.
Theorem \ref{thm:2quivercells} can be formulated as follows.
\begin{thm}
  With every truncated preprojective representation $P_{m+1}^V$, say with $\dim V=r$, we can associated a $2$-quiver $\mathcal Q_{m+1}^{[r]}$ such that the affine cells of the cell decomposition attached to $\Gr_\bfe(P_{m+1}^V)$ are labeled by strong successor closed subsets $\beta\subset (\mathcal Q_{m+1}^{[r]})_0$.
  In particular, the Euler characteristic $\chi(\Gr_\bfe(P_{m+1}^V))$ is given by the number of these subsets.
\end{thm}

The results of \cite{rupel} give a combinatorial construction of counting polynomials for quiver Grassmannians of preprojective/preinjective representations of generalized Kronecker quivers $K(n)$.
This suggests that the dimensions of cells can be directly computed using this combinatorics (or the equivalent combinatorics of compatible pairs).
This is made precise in Conjecture~\ref{conj:cell dimensions}.

\subsubsection*{Acknowledgements}
We would like to thank Giovanni Cerulli Irelli, Hans Franzen, Oliver Lorscheid and Markus Reineke for very fruitful discussions related to this project.

\section{Quiver Covering Theory}
\label{sec:covering}
\noindent
We refer to \cite{gab} for an introduction to covering theory.
Let $Q$ be an acyclic quiver with vertices $Q_0$ and arrows $Q_1$ which we denote by $\alpha:s(\alpha)\to t(\alpha)$.
A $\CC$-representation $X$ of $Q$ consists of a collection of $\CC$-vector spaces $X_i$ for $i\in Q_0$ and a collection of $\CC$-linear maps $X_\alpha:X_{s(\alpha)}\to X_{t(\alpha)}$ for $\alpha\in Q_1$.
Given $\CC$-representations $X$ and $Y$ of $Q$, a morphism $f:X\to Y$ is a collection of $\CC$-linear maps $f_i:X_i\to Y_i$ for $i\in Q_0$ satisfying $f_{t(\alpha)}\circ X_\alpha=Y_\alpha\circ f_{s(\alpha)}$ for each $\alpha\in Q_1$.
We write $\rep Q$ for the hereditary abelian category of finite-dimensional $\C$-representations of $Q$ and we assume in the following that all representations are finite-dimensional.

Recall that, given $\CC$-representations $X$ and $Y$ of $Q$, any tuple of linear maps $(g_\alpha:X_{s(\alpha)}\to Y_{t(\alpha)})_{\alpha\in Q_1}$ defines a short exact sequence $\ses{Y}{Z}{X}$ with middle term given by the vector spaces $Z_i=X_i\oplus Y_i$ for $i\in Q_0$ and the linear maps $Z_\alpha=\begin{pmatrix} X_\alpha & 0 \\ g_\alpha & Y_\alpha \end{pmatrix}$ for $\alpha\in Q_1$. 
In general, considering the linear map  
\begin{align}
  \label{eq:maphomext}
  d_{X,Y}:\bigoplus_{i\in Q_0}\Hom_\CC(X_i,Y_i)\to\bigoplus_{\alpha\in Q_1}\Hom_\CC(X_{s(\alpha)},Y_{t(\alpha)}),
  \quad (f_i)_{i\in Q_0}\mapsto(f_{t(\alpha)}\circ X_\alpha-Y_\alpha\circ f_{s(\alpha)})_{\alpha\in Q_1},
\end{align}
we have $\mathrm{ker}(d_{X,Y})=\Hom_Q(X,Y)$ and $\mathrm{coker}(d_{X,Y})=\Ext_Q(X,Y)$. In the following, we write $\Hom$ (resp. $\Ext$) instead of $\Hom_Q$ (resp. $\Ext_Q$). As usual, we call a representation $X$ \emph{rigid} if $\Ext(X,X)=0$ and \emph{exceptional} if it is also indecomposable.

Let $W_Q$ be the free (non-abelian) group generated by $Q_1$.  
Write $A_Q\cong \ZZ^{Q_1}$ for the free abelian group generated by $Q_1$ and denote by $e_\alpha\in A_Q$ the generator corresponding to $\alpha\in Q_1$. 
\begin{definition}
  \label{def:covering quivers}
  The \emph{universal abelian covering quiver} $\hat Q$ of $Q$ has vertices $\hat Q_0=Q_0\times A_Q$ and arrows $\hat Q_1=Q_1\times A_Q$, where $(\alpha,\chi):\big(s(\alpha),\chi\big)\to\big(t(\alpha),\chi+e_\alpha\big)$ for $\alpha\in Q_1$ and $\chi\in A_Q$.
  Write $F_Q:\rep \hat Q\to\rep Q$ for the natural functor. 

  The \emph{universal covering quiver} $\widetilde Q$ of $Q$ has vertices $\widetilde Q_0=Q_0\times W_Q$ and arrows $\widetilde Q_1=Q_1\times W_Q$, where $(\alpha,w):\big(s(\alpha),w\big)\to\big(t(\alpha),w\alpha\big)$ for $\alpha\in Q_1$ and $w\in W_Q$.
  Write $G_Q:\rep\widetilde Q\to\rep Q$ for the natural functor. 

  We say that a representation $X\in\rep Q$ can be \emph{lifted} to $\hat Q$ (resp. $\widetilde Q$) if there exists a representation $\hat X\in\rep \hat Q$ (resp. $\widetilde X\in\rep \widetilde Q$) such that $F_Q\hat X=X$ (resp. $G_Q \widetilde X=X$).
\end{definition}

Note that in our definition every covering quiver has infinitely many connected components, but each of its connected components is a covering in the sense of \cite{gab}.
As indecomposable representations live on one of these components, this distinction will be irrelevant.
Note that the natural surjection $W_Q\onto A_Q$ induces a functor $H_Q:\rep \widetilde Q\to\rep \hat Q$.
In addition, observe that every connected component of the universal covering quiver of the universal abelian covering quiver is isomorphic to a connected component of the universal covering quiver of the original quiver.

\begin{lemma}
  \label{le:lifts of transjectives}
  Every preprojective and preinjective representation of $Q$ can be lifted to $\hat Q$ and to $\widetilde Q$. Any lift of a preprojective (or preinjective) representation is preprojective (preinjective). 
\end{lemma}
\begin{proof}
  This statement is clear for the simple representations $S_i$, $i\in Q_0$.
  Now every preprojective or preinjective representation of $Q$ can be obtained by applying a sequence of BGP-reflections \cite{bgp} to a simple representation $S'_j$ of a quiver $Q'$ whose underlying graph is the same as the one for $Q$.
  Applying a BGP-reflection to a fixed vertex $i$ of $Q$ corresponds to applying BGP-reflections to all vertices $(i,\chi)\in\hat Q_0$, where $\chi$ runs through all $\chi\in A_Q$ (resp. to all $(i,w)\in \tilde Q_0$ with $w\in W_Q$).
  This gives both claims.
\end{proof}

The functor $F_Q$ induces a map $F_Q:\ZZ^{\hat Q_1}\to \ZZ^{Q_1}$.
We say that a dimension vector $\hat \bfe$ of $\hat Q$ is \emph{compatible} with $\bfe$ if $F_Q(\hat\bfe)=\bfe$.
The group $A_Q$ acts on $\hat Q$ via translation, this induces actions of $A_Q$ on $\rep\hat Q$ and on $\ZZ^{\hat Q_1}$.
The analogous observation can also be made for $\widetilde Q$.
If $X$ is a representation of $\hat Q$ (resp. $\widetilde Q$), we denote by $X_\chi$ (resp. $X_w$) the representation obtained via translation by $\chi\in A_Q$ (resp. $w\in W_Q$).  

As $F_Q$ and $G_Q$ are covering functors when restricting to one connected component, we obtain the following result from \cite{gab}.
\begin{theorem}
  \label{covering}
  The functors $F_Q$ and $G_Q$ preserve indecomposability.
  Moreover, for all representations $\hat X,\hat Y \in\rep(\hat Q)$, we have 
  \[\Hom_Q(F_Q\hat X, F_Q\hat Y)\cong \bigoplus_{\chi\in A_Q}\Hom_{\hat Q}(\hat X_\chi,\hat Y)\cong\bigoplus_{\chi\in A_Q}\Hom_{\hat Q}(\hat X,\hat Y_\chi).\]
  Analogous isomorphisms exist when replacing $\Hom$ by $\Ext$ and/or $\rep \hat Q$ by $\rep\widetilde Q$.
\end{theorem}

\section{Representation Theory of Generalized Kronecker Quivers}
\label{sec:RepK(n)}

\noindent Fix $n\ge3$. Denote by $K(n)$ the \emph{$n$-Kronecker quiver} $1\stackrel{n}{\longleftarrow}2$ with vertices $K_0(n)=\{1,2\}$ and $n$ arrows from vertex $2$ to vertex $1$. 
The category $\rep K(n)$ of finite-dimensional representations of $K(n)$ is equivalent to the category of modules over the path algebra $A(n)$ of $K(n)$.
As a $\CC$-vector space, the path algebra $A(n)$ can be written as $A_0\oplus A_1$, where 
\begin{itemize}
  \item $A_0=\CC e_1\oplus \CC e_2$ is a two-dimensional semisimple algebra with orthogonal idempotents $e_1$ and $e_2$;
  \item $A_1=\bigoplus_{i=1}^n \CC\alpha_i$ is the $A_0$-bimodule spanned by the arrows of $K(n)$, that is $e_k\alpha_ie_\ell=\delta_{k1}\delta_{\ell2}\alpha_i$ for $1\le i\le n$ and $k,\ell\in\{1,2\}$.
\end{itemize}

Write $\Sigma_1$ and $\Sigma_2$ for the BGP-reflection functors of $K(n)$ \cite{bgp}. 
We use the same symbols $\Sigma_1$, $\Sigma_2$ for the BGP-reflection functors of $K(n)^{op}$, this should not lead to any confusion. 
Then each endofunctor $\Sigma_i^2$ is naturally isomorphic to the identity map on the full subcategory $\rep_{\langle i\rangle} K(n)\subset \rep K(n)$ whose objects are those representations of $K(n)$ which do not contain the simple $S_i$ as a direct summand.
In particular, $\Sigma_i$ gives an exact equivalence of categories $\Sigma_i:\rep_{\langle i\rangle} K(n)\to\rep_{\langle i\rangle} K(n)^{op}$.
Also, following \cite{brenner-butler}, the Auslander-Reiten translation $\tau:\rep K(n)\to\rep K(n)$ may be identified with the functor $\Sigma_2\Sigma_1$.

Define Chebyshev polynomials $u_k$ for $k\in\ZZ$ by the recursion $u_0=0$, $u_1=1$, $u_{k+1}=nu_k-u_{k-1}$.
The following is well-known.
\begin{theorem}
  \label{th:rigids}
  For each $m\ge1$, there exist unique (up to isomorphism) exceptional representations $P_m$ and $I_m$ of $K(n)$ with dimension vectors $(u_m,u_{m-1})$ and $(u_{m-1},u_m)$ respectively satisfying $\Hom(P_m,P_r)=0$ (resp. $\Hom(I_r,I_m)=0$) and $\Ext(P_r,P_m)=0$ (resp. $\Ext(I_m,I_r)=0$) for $1\leq r\leq m$.
  Moreover, any rigid representation of $K(n)$ is isomorphic to one of the form $P_m^{a_1}\oplus P_{m+1}^{a_2}$ or $I_m^{a_1}\oplus I_{m+1}^{a_2}$ for some $m\ge1$ and some $a_1,a_2\ge0$.
	
\end{theorem}
The representations $P_m$ are called the \emph{preprojective} representations of $K(n)$ and the representations $I_m$ are called \emph{preinjective}.
\begin{remark}
  \label{rem:reflection recursion}
  We may identify the quiver $K(n)$ with $K(n)^{op}$ by interchanging the vertex labels.
  This induces an isomorphism of categories $\rep K(n)\cong\rep K(n)^{op}$ which we write as $M\mapsto M^\sigma$.
  Note that $\Sigma_1(M^\sigma)=(\Sigma_2 M)^\sigma$ and $\Sigma_2(M^\sigma)=(\Sigma_1 M)^\sigma$.
  \begin{enumerate}
    \item The preprojective and preinjective representations satisfy the following recursions using the reflection functors:
      \[P_1=S_1,\quad P_m^\sigma=\Sigma_2 P_{m-1},\quad I_1=S_2,\quad I_m^\sigma=\Sigma_1 I_{m-1}\]
      for $m\ge2$.
      In particular, we have $P_{m-1}=\tau P_{m+1}$ and $I_{m+1}=\tau I_{m-1}$ for $m\ge2$.
    \item If $\ses{M}{B}{N}$ is a short exact sequence such that no direct summand of $M$, $B$, nor $N$ is preinjective, then the sequences $\ses{(\Sigma_1\Sigma_2)^nM}{(\Sigma_1\Sigma_2)^nB}{(\Sigma_1\Sigma_2)^nN}$ and $\ses{\Sigma_2(\Sigma_1\Sigma_2)^nM}{\Sigma_2(\Sigma_1\Sigma_2)^nB}{\Sigma_2(\Sigma_1\Sigma_2)^nN}$ are exact for any $n\ge0$ and none of these representations contain preinjective direct summands.
  \end{enumerate}
\end{remark}

Set $\cH_m:=\Hom(P_m,P_{m+1})$ for $m\ge1$.
Write $\Gr(\cH_m)$ for the \emph{total Grassmannian} of $\cH_m$ whose elements are non-trivial proper subspaces $V\subset \cH_m$.
Some results below remain true if we allow $\cH_m$ or $0$ as elements of $\Gr(\cH_m)$, but not all, so for uniformity of exposition we omit these possibilities.

For each $m\ge2$, there is an Auslander-Reiten sequence (cf.\ \cite[Section V]{ars})
\begin{equation}
  \label{eq:AR sequence}
  0\longrightarrow P_{m-1}\stackrel{\iota_{m-1}}{\longrightarrow} P_m\otimes \cH_m\stackrel{ev}{\longrightarrow} P_{m+1}\longrightarrow 0,
\end{equation}
where the right-hand morphism is the natural evaluation map.
\begin{lemma}
  \label{le:injective evaluation maps}
  For any $V\in \Gr(\cH_m)$, $m\geq 1$, the natural evaluation map $ev_V:P_m\otimes V\to P_{m+1}$ is injective.
\end{lemma}
\begin{proof}
  As $\Ext(P_m\otimes V,P_{m-1})=0$ and $\Hom(P_1,I_1)=0$, we obtain a commutative diagram 
  \[\xymatrix{0 \ar[r] & P_{m-1}\ar@{=}[d]\ar[r] &P_{m-1}\oplus P_m\otimes V \ar@{}[dr]|(.7){\lrcorner} \ar[r]\ar[d]_{(\iota_{m-1},\mathrm{id}_{P_m}\otimes \iota_V)} &P_{m}\otimes V \ar[r]\ar_{ev_V}[d] & 0 \ar[d]\\ 0 \ar[r] & P_{m-1}\ar[r]^{\iota_{m-1}} & P_m\otimes \cH_m \ar[r]^{ev} &P_{m+1} \ar[r] & I_{2-m}\ar[r]&0}\]
  in which we set $P_0=0$ and $I_l=0$ for $l\leq 0$.
  Since $\mathrm{id}_{P_1}\otimes \iota_V$ is injective, the snake lemma shows that $\ker(ev_V)=0$ for $m=1$ and thus $ev_V$ is injective in this case.
  
  In view of Remark~\ref{rem:reflection recursion}, it is enough for the case $m>1$ to show that the cokernel of $ev_V$ does not have a preinjective direct summand when $m=1$.
  Clearly the cokernel of $\mathrm{id}_{P_1}\otimes \iota_V$ is isomorphic to $P_1\otimes \cH_1/V$.
  Thus we obtain the following commutative diagram induced by the cokernels of the above vertical maps, note that the vertical maps below are surjective:
  \[\xymatrix{0 \ar[r]  & P_1\otimes \cH_1 \ar[d]\ar[r]^{ev} &P_{2}\ar[d] \ar[r] & I_{1}\ar@{=}[d]\ar[r]&0\\ 0\ar[r] & P_1\otimes \cH_1/V \ar[r] &K\ar[r]& I_{1}\ar[r]&0.}\]

  We need to show that $K$ has no preinjective direct summand.
  As $\Hom(P_2,P_1)=0$ and as the vertical maps are surjective, the representation $K$ has no direct summand which is isomorphic to $P_1=S_1$.
  But this already shows that $K$ is indecomposable as $\udim K=(\dim \cH_1/V,1)$.
  Since $V$ is a proper subspace of $\cH_1$, $\udim K$ is not the dimension vector of a preinjective representation and the claim follows.
\end{proof}

In what follows we will not distinguish between $P_m\otimes V$ and its image under $ev_V$.

\subsection{Truncated Preprojectives}
\label{sec:truncated preprojectives}

Motivated by Lemma~\ref{le:injective evaluation maps}, we define the following.
\begin{definition}
  For $V\in \Gr(\cH_m)$, define the \emph{truncated preprojective} $P_{m+1}^V$ to be the cokernel of the map $ev_V:P_m\otimes V\to P_{m+1}$, i.e.\ we have a short exact sequence
  \begin{equation}
    \label{eq:truncated preprojectives}
    0\longrightarrow P_m\otimes V\stackrel{ev_V}{\longrightarrow} P_{m+1}\stackrel{\pi_V}{\longrightarrow} P_{m+1}^V\longrightarrow 0.
  \end{equation}
\end{definition}
\begin{remark}
  It will be convenient to also set $P_{m+1}^0=P_{m+1}$, observe that this notation is consistent with taking $V=0$ in the sequence \eqref{eq:truncated preprojectives}.
\end{remark}

We collect below several basic homological results related to preprojective representations.
\begin{lemma}
  \label{le:truncated homomorphisms}
  For $V\in \Gr(\cH_m)$, $m\ge1$, we have $\Hom(P_m,P_{m+1}^V)\cong \cH_m/V$ and $\Ext(P_m,P_{m+1}^V)=0$.
\end{lemma}
\begin{proof}
  As $P_m$ is exceptional, applying the functor $\Hom(P_m,-)$ to the sequence \eqref{eq:truncated preprojectives}, gives an exact sequence
  \[\xymatrix{0 \ar[r] & \Hom(P_m,P_m\otimes V) \ar[r] & \Hom(P_m,P_{m+1}) \ar[r] & \Hom(P_m,P_{m+1}^V) \ar[r] & 0}\]
  and an isomorphism
  \[\Ext(P_m,P_{m+1})\cong\Ext(P_m,P_{m+1}^V).\]
  But there is a natural isomorphism $\Hom(P_m,P_m\otimes V)\cong V$ and the first claim follows.
  The final claim follows from Theorem ~\ref{th:rigids} which implies $\Ext(P_m,P_{m+1})=0$.
\end{proof}

\begin{lemma}
  \label{le:unique preprojective morphism}
  For $V\in \Gr(\cH_m)$, $m\ge1$, the space $\Hom(P_{m+1},P_{m+1}^V)$ is one-dimensional spanned by the natural projection $\pi_V:P_{m+1}\to P_{m+1}^V$.
  Moreover, $\Ext(P_{m+1},P_{m+1}^V)=0$.
\end{lemma}
\begin{proof}
 As $\Hom(P_{m+1},P_m)=\Ext(P_{m+1},P_m)=0$, applying the functor $\Hom(P_{m+1},-)$ to the sequence \eqref{eq:truncated preprojectives} gives isomorphisms 
  \[\Hom(P_{m+1},P_{m+1})\cong\Hom(P_{m+1},P_{m+1}^V)\]
  and 
  \[\Ext(P_{m+1},P_{m+1})\cong\Ext(P_{m+1},P_{m+1}^V).\]
  Under the first isomorphism, the identity map on $P_{m+1}$ is taken to the projection $\pi_V:P_{m+1}\to P_{m+1}^V$.
  The second isomorphism together with the rigidity of $P_{m+1}$ gives the final claim.
\end{proof}

\begin{lemma}
  \label{le:unique morphisms}
  Consider $V,W\in \Gr(\cH_m)$, $m\ge1$.
  \begin{enumerate}
    \item There exists a morphism $P_{m+1}^W\to P_{m+1}^V$ if and only if $W\subset V$ and this morphism is unique (up to scalars) when it exists.
    \item For $W\subset V$, we have $\Ext(P_{m+1}^W,P_{m+1}^V)\cong W^*\otimes(\cH_m/V)$.
  \end{enumerate}
\end{lemma}
\begin{proof}
  We apply the functor $\Hom(-,P_{m+1}^V)$ to the sequence \eqref{eq:truncated preprojectives} for $W$ to get an exact sequence
  \[0\longrightarrow \Hom(P_{m+1}^W,P_{m+1}^V)\longrightarrow \Hom(P_{m+1},P_{m+1}^V)\stackrel{-\circ ev_W}{\longrightarrow} \Hom(P_m\otimes W,P_{m+1}^V)\longrightarrow \Ext(P_{m+1}^W,P_{m+1}^V)\longrightarrow 0.\]
  But the space $\Hom(P_{m+1},P_{m+1}^V)$ is one-dimensional and thus $\Hom(P_{m+1}^W,P_{m+1}^V)$ is nonzero if and only if the morphism $-\circ ev_W$ of the above sequence is zero.
  But this occurs exactly when the image of the map $ev_W:P_m\otimes W\to P_{m+1}$ is contained in the kernel of $\pi_V$, i.e.\ in the image of $ev_V:P_m\otimes V\to P_{m+1}$, and this occurs if and only if $W\subset V$. 
  In this case, there are isomorphisms
  \[\Ext(P_{m+1}^W,P_{m+1}^V)\cong\Hom(P_m\otimes W,P_{m+1}^V)\cong W^*\otimes \cH_m/V,\]
  where the last isomorphism is immediate from Lemma~\ref{le:truncated homomorphisms}.
\end{proof}
\begin{remark}
  The total Grassmannian $\Gr(\cH_m)$ is naturally a poset under inclusion.
  This structure gives rise to a $\CC$-linear category $\CC \Gr(\cH_m)$ with objects the elements of $\Gr(\cH_m)$ and at most one morphism (up to scalars) between any two objects.
  Write $\cP_{m+1}$ for the full subcategory of $\rep K(n)$ with objects the truncated preprojectives $P_{m+1}^V$ for $V\in \Gr(\cH_m)$.
  By Lemma~\ref{le:unique morphisms}(1), the functor $V\mapsto P_{m+1}^V$ gives an isomorphism of categories $\CC \Gr(\cH_m)\cong\cP_{m+1}$.
\end{remark}

For the truncated preprojective representations $P_{m+1}^V$, we have  $\udim P_{m+1}^V=d(m,\dim V):=\udim P_{m+1}-\dim V\cdot\udim P_m$.
These will play an important role when describing quiver Grassmannians of preprojective representations recursively.
For a dimension vector $d=(d_1,d_2)\in\NN^{K(n)_0}$, write 
\[R_d(K(n))=\bigoplus_{i=1}^n\Hom_{\CC}(\CC^{d_2},\CC^{d_1})\]
for the affine space of representations of $K(n)$ with dimension vector $d$. 
\begin{proposition}
  \label{pro:indecomposables}
  Let $m\geq 1$ and $0\le r\le n-1$.
  The following hold:
  \begin{enumerate}
    \item The isomorphism classes of indecomposable representations of $K(n)$ with dimension vector $d(m,r)$ are in one-to-one correspondence with points of $\Gr_{n-r}(\CC^n)$.
    \item The indecomposable representations of $K(n)$ with dimension vector $d(m,r)$ are precisely the truncated preprojective representations $P_{m+1}^V$ for $V\in \Gr(\cH_m)$ with $\dim V=r$. 
    \item The set of indecomposable representations with dimension vector $d(m,r)$ is given by a non-empty open subset of $R_{d(m,r)}(K(n))$. 
  \end{enumerate}
\end{proposition}
\begin{proof}
  As the reflection functors $\Sigma_1,\,\Sigma_2$ preserve indecomposability and $\Sigma_2(d(m,r))^\sigma=d(m+1,r)$, it suffices to prove the first statement for $m=1$.
  Then we have $d(m,r)=(l,1)$ for $l:=n-r$.
  This means that $X\in R_{d(m,r)}(K(n))$ can be represented by a matrix $M_X\in\CC^{l\times n}$, where the $i^{\mathrm{th}}$ column stands for $X_{\alpha_i}$.
  Now $X$ is indecomposable if and only if $\rk(M_X)=l$.
  Indeed, $X$ admits a summand isomorphic to $S_1^k$ exactly when $\rk(M_X)=l-k$.

  This shows that the indecomposable representations in $R_{(l,1)}(K(n))$ are in one-to-one correspondence with $l\times n$ matrices of maximal rank.
  Thus we may associate to each such representation $X$ a subspace of $\CC^n$ of dimension~$l$ spanned by the row vectors of the corresponding matrix $M_X$.
  Now it is straightforward to check that the $\GL_{d(m,r)}=\GL_l(\CC)\times\CC^\ast$-action on $R_{d(m,r)}(K(n))$ corresponds to the base change action of $\GL_l(\CC)$ on the set of these subspaces.
  This shows the first statement.

  By Lemma~\ref{le:unique morphisms}, the endomorphism ring of $P_{m+1}^V$ is one-dimensional and so $P_{m+1}^V$ must be indecomposable.
  Since both the isomorphism classes of indecomposables and the isomorphism classes of truncated projectives with dimension vector $d(m,r)$ are parametrized by the same Grassmannian, this gives the second claim.

  As there exist representations with trivial endomorphism ring, the dimension vectors $d(m,r)$ are Schur roots.
  It follows that the set of indecomposable representations with trivial endomorphism ring forms a dense open subset of $R_{d(m,r)}(K(n))$, see for example \cite[Theorem 2.2]{sch}.
  This shows the last claim.
\end{proof}
\begin{remark}
  There is a more elegant way to prove the first part of Proposition~\ref{pro:indecomposables} using the notion of stability and moduli spaces.
  Actually, fixing the standard stability induced by the linear form $\Theta:\ZZ^{Q_0}\to\ZZ$ defined by $\Theta(d)=d_2$, it can be shown that all indecomposables are stable and that the moduli space of stable representations is in fact $\Gr_l(\CC^n)$.
  We opted for the proof above because the notion of stability would only be used at this point and we wanted to keep the exposition as simple as possible. 
\end{remark}

\begin{lemma}
  \label{le:basic homological properties}
  For $V\in\Gr(\cH_m)$, $m\ge1$, we have $\Hom(P_{m+1}^V,P_\ell)=0=\Ext(P_\ell,P_{m+1}^V)$ for all $\ell\ge1$.
\end{lemma}
\begin{proof}
  Recall that $\Hom(X,P_\ell)\ne0$ (resp. $\Ext(P_\ell,X)\ne0$) for some indecomposable representation $X$ and some $\ell\ge1$ implies that $X$ is preprojective of the form $P_r$ with $1\le r\le\ell$ (resp. $1\le r\le\ell-2$).
  However, $P_{m+1}^V$ is indecomposable by Proposition~\ref{pro:indecomposables} and it cannot be preprojective as it is not rigid by Lemma~\ref{le:unique morphisms}.
\end{proof}

\begin{lemma}
  \label{le:reflected truncated preprojectives}
  For $V\in\Gr(\cH_m)$, $m\ge2$, the representation $\Sigma_1(P_{m+1}^V)^\sigma$ is also truncated preprojective.
\end{lemma}
\begin{proof}
  By Proposition~\ref{pro:indecomposables} and Lemma~\ref{le:unique morphisms}, $P_{m+1}^V$ is indecomposable but not rigid.
  In particular, $P_{m+1}^V$ does not have a summand isomorphic to $S_1$.
  Thus, following Remark~\ref{rem:reflection recursion}, we may apply the functor $\Sigma_1(-)^\sigma$ to the sequence \eqref{eq:truncated preprojectives} to get the exact sequence
  \[0\longrightarrow P_{m-1}\otimes V\xrightarrow{\Sigma_1(ev_V)^\sigma} P_m\longrightarrow \Sigma_1(P_{m+1}^V)^\sigma\longrightarrow 0\]
  which gives the claim.
\end{proof}

\begin{lemma}
  \label{le:projective subrepresentations}
  For $V\in \Gr(\cH_m)$, $m\ge1$, any proper subrepresentation $X\subsetneq P_{m+1}^V$ can be written as a direct sum of preprojective representations $P_r$ with $1\le r\le m$.
\end{lemma}
\begin{proof}
  We proceed by induction on $m$.
  When $m=1$, the dimension vector of $P_{m+1}^V$ is $(\codim_{\cH_m} V,1)$.
  In particular, it is immediate that each proper subrepresentation of $P_{m+1}^V$ is isomorphic to $P_1^k$ for some $0\le k\le\codim_{\cH_m} V$.

  For $m\ge2$, we observe by induction that Lemma~\ref{le:reflected truncated preprojectives} implies that any subrepresentation of $P_{m+1}^V$ which has no summand isomorphic to $P_1$ must be a direct sum of preprojective representations $P_r$ with $2\le r\le m$.
  Indeed, each of these is obtained from a subrepresentation of $\Sigma_1(P_{m+1}^V)^\sigma$ by applying the functor $\Sigma_2(-)^\sigma$ and the claim follows from the recursions in Remark~\ref{rem:reflection recursion} for preprojective representations.
\end{proof}

For $V\in \Gr(\cH_m)$, $m\ge1$, any subspace $W\subset V$ gives rise to an exact sequence
\[\xymatrix{0 \ar[r] & P_m\otimes(V/W) \ar[r]^(.6){ev} & P_{m+1}^W \ar[r] & P_{m+1}^V \ar[r] & 0},\]
where the left hand morphism above is the natural evaluation morphism coming from Lemma~\ref{le:truncated homomorphisms}.
Each such sequence has the following almost-split property for proper subrepresentations of $P_{m+1}^V$.
\begin{corollary}
  \label{cor:base fibers}
  Consider $V,W\in \Gr(\cH_m)$, $m\ge1$, with $W\subset V$.
  Given any proper subrepresentation $X\subsetneq P_{m+1}^V$ and any subrepresentation $Z\subset P_m\otimes(V/W)$, there is a subrepresentation of $P_{m+1}^W$ isomorphic to $Z\oplus X$ which fits into a commutative diagram
  \[\xymatrix{0 \ar[r] & Z \ar[d]\ar[r] & Z\oplus X \ar[r]\ar[d] & X \ar[d]\ar[r] & 0 \\
    0 \ar[r] & P_m\otimes(V/W) \ar[r] & P_{m+1}^W \ar[r] & P_{m+1}^V \ar[r] & 0}\]
\end{corollary}
\begin{proof}
  Observe that $\Ext(P_r,P_m)=0$ for $1\le r\le m$ and, since $X$ is a direct sum of preprojectives $P_r$ with $1\leq r\leq m$, the upper pullback sequence
  \[\xymatrix{0 \ar[r] & P_m\otimes(V/W) \ar@{=}[d]\ar[r] & Y \ar[r]\ar[d]\ar@{}[dr]|(.7){\lrcorner} & X \ar[d]\ar[r] & 0 \\
    0 \ar[r] & P_m\otimes(V/W) \ar[r] & P_{m+1}^W \ar[r] & P_{m+1}^V \ar[r] & 0}\]
  must split.
  The claim for arbitrary subrepresentations $Z\subset P_m\otimes(V/W)$ is an immediate consequence of this splitting.
\end{proof}

\begin{lemma}
  \label{le:unique truncated extension}
  For $V\in \Gr(\cH_m)$, $m\ge2$, the space $\Ext(P_{m+1}^V,P_{m-1})$ is one-dimensional and spanned by the extension
  \begin{equation}
    \label{eq:unique extension}
    \xymatrix{0 \ar[r] & P_{m-1} \ar[r]^(.375){\kappa_V} & P_m\otimes(\cH_m/V) \ar[r]^(.625){ev} & P_{m+1}^V \ar[r] & 0}.
  \end{equation}
\end{lemma}
\begin{proof}
  Applying the functor $\Hom(-,P_{m-1})$ to the sequence \eqref{eq:truncated preprojectives} gives an isomorphism $\Ext(P_{m+1}^V,P_{m-1})\cong\Ext(P_{m+1},P_{m-1})$ with a one-dimensional space.
  Writing $X$ for the unique extension of $P_{m+1}^V$ by $P_{m-1}$, this isomorphism gives rise to the following pullback diagram:
  \[\xymatrix{ & & 0 \ar[d] & 0\ar[d] & \\
    & & P_m\otimes V \ar@{=}[r]\ar[d] & P_m\otimes V\ar[d] & \\
    0 \ar[r] & P_{m-1} \ar@{=}[d]\ar[r]^(.425){\iota_{m-1}} & P_m\otimes \cH_m \ar[r]\ar[d]\ar@{}[dr]|(.7){\lrcorner} & P_{m+1} \ar[r]\ar[d] & 0\\
    0 \ar[r] & P_{m-1} \ar[r]^{\kappa_V} & X \ar[r]\ar[d] & P_{m+1}^V \ar[r]\ar[d] & 0\\
    & & 0 & 0 & }\]
  from which we immediately obtain the isomorphism $X\cong P_m\otimes(\cH_m/V)$.
\end{proof}

\begin{lemma}
  \label{le:preprojective homomorphism duality}
  The sequence \eqref{eq:AR sequence} gives rise to an isomorphism $\cH_m^*\cong \cH_{m-1}$.
\end{lemma}
\begin{proof}
  Define a map $\cH_m^*\to \cH_{m-1}$ by $\varphi\mapsto\bar{\varphi}:=(id\overline{\otimes}\varphi)\circ\iota_{m-1}$, in words $\bar{\varphi}$ acts on $x\in P_{m-1}$ by contracting with the second factor in $\iota_{m-1}(x)$ to give an element of $P_m$.
  Suppose $\varphi\in \cH_m^*$ is a nonzero functional on $\cH_m$ and let $V\subsetneq \cH_m$ denote the kernel of $\varphi$.
  Then $\bar{\varphi}=0\in \cH_{m-1}$ if and only if the image of $\iota_{m-1}$ is contained in $P_m\otimes V\subsetneq P_m\otimes \cH_m$.
  But then Lemma~\ref{le:injective evaluation maps} implies $ev\circ\iota_{m-1}=ev_V\circ\iota_{m-1}\ne0$, a contradiction.
  Thus the map $\cH_m^*\to \cH_{m-1}$, $\varphi\mapsto\bar{\varphi}$ must be injective and hence an isomorphism.
\end{proof}
\begin{remark}
  The set $\Gr(\cH_m)$ is naturally a poset and Lemma~\ref{le:preprojective homomorphism duality} gives an identification of the opposite poset $\Gr(\cH_m)^{op}\cong \Gr(\cH_m^*)$ with $\Gr(\cH_{m-1})$.
  We write $\bar{V}\subset \cH_{m-1}$ for the subspace corresponding to $V\subset \cH_m$ under this identification.
  Under the isomorphism of $\cH_{m-1}$ with $\cH_m^*$, we have $\bar{V}=(\cH_m/V)^*$.
\end{remark}

\begin{corollary}
  \label{cor:truncated preprojective isomorphism}
  Suppose $V\in \Gr(\cH_m)$ has codimension-one in $\cH_m$.  Then $P_{m+1}^V\cong P_m^{\bar{V}}$.
\end{corollary}
\begin{proof}
  By Lemma~\ref{le:unique truncated extension}, we have the exact sequence
  \[\xymatrix{0 \ar[r] & P_{m-1} \ar[r]^(.35){\kappa_V} & P_m\otimes(\cH_m/V) \ar[r]^(.625){ev} & P_{m+1}^V \ar[r] & 0}.\]
  But $\cH_m/V$ is a one-dimensional vector space and so $P_m\otimes(\cH_m/V)\cong P_m$.
  Under this identification, the left hand morphism $\kappa_V$ in the sequence above identifies with a generator of $\bar{V}$ and thus $P_{m+1}^V\cong P_m^{\bar{V}}$.
\end{proof}

\subsection{Lifting to \texorpdfstring{$\mathbf{\widetilde{K(n)}}$}{K(n)}}
\label{Lifting}

Fix a natural number $n\geq 3$.
Write $W_n:=W_{K(n)}$ for the free group generated by the arrows $\alpha_i$, $1\le i\le n$, of $K(n)$ and denote by $e\in W_n$ its identity element.
In this section, we fix compatible bases for each $\cH_m:=\Hom(P_m,P_{m+1})$ and use these to lift the (truncated) preprojective representations of the quiver $K(n)$ to the universal cover~$\widetilde{K(n)}$.
This lifting will rigidify the situation, allowing more precise control over these representations and their subrepresentation structure.
Of particular importance is Corollary~\ref{cor:perpendicular} which has no reasonable analogue for $K(n)$.
One main advantage of the lifting is that those truncated preprojectives which can be lifted are exceptional representations on the universal covering quiver $\widetilde{K(n)}$.

We will mainly be interested in particular lifts $\tilde P_m$ of the preprojective representations $P_m$ of $K(n)$ to the universal cover $\widetilde{K(n)}$.
In the notation of Definition~\ref{def:covering quivers}, this means $G(\tilde P_m)=P_m$, where $G$ is the covering functor 
\[G:=G_{K(n)}:\rep\widetilde{K(n)}\to\rep K(n).\]

To construct the lifts, first recall that applying the BGP-reflection functor $\Sigma_i$ on $K(n)$ corresponds to applying the iterated reflection $\tilde\Sigma_i:=\prod_{w\in W_n}\Sigma_{(i,w)}$ on $\widetilde{K(n)}$.
Moreover, under this operation all sinks of~$\widetilde{K(n)}$ become sources and vice versa.
The preprojective lifts we use are defined by the following analogue of the recursions of Remark~\ref{rem:reflection recursion}.
\begin{itemize}
  \item We consider the lift $\tilde P_1$ satisfying $\udim(\tilde P_1)_{(1,e)}=1$ and $\udim(\tilde P_1)_{(i,w)}=0$ for $(i,w)\neq (1,e)$.
  \item We consider the lift $\tilde P_2$ satisfying $\udim(\tilde P_2)_{(2,e)}=\udim(\tilde P_2)_{(1,\alpha_i)}=1$ for $1\le i\le n$ and $\udim(\tilde P_2)_{(i,w)}=0$ for $(i,w)\notin\{(2,e),(1,\alpha_1),\ldots,(1,\alpha_n)\}$.
  \item For $m\geq 3$, we build the lifts $\tilde P_m$ recursively by applying reflection functors or as Auslander-Reiten translates.
    More precisely, we set
    \begin{align}
      \label{eq:recursive covers}
      \tilde P_m:=\tilde\Sigma_2(\tilde P_{m-1})^{\sigma}\qquad\text{or}\qquad\tilde P_{m+1}:=\tilde\Sigma_{1}\tilde\Sigma_2\tilde P_{m-1}=\tau^{-1} \tilde P_{m-1},
    \end{align}
    where $(-)^\sigma:\rep\widetilde{K(n)}^{op}\to\rep\widetilde{K(n)}$ is the lift of the corresponding functor for $K(n)$.
\end{itemize}

It will be rather important that our chosen lifts of $P_{2l}$ for $l\ge1$ and our chosen lifts of $P_{2l-1}$ for $l\ge1$ live on two different components of the universal covering quiver.
Indeed, recall that the group $W_n$ naturally acts on $\widetilde{K(n)}_0$ via translation, i.e.~$w.(i,w')=(i,ww')$, and this induces an action of $W_n$ on $\rep\widetilde{K(n)}$.
Following the notation of Section~\ref{sec:covering}, we write $\tilde P_{m,w}$ for the representation of $\widetilde{K(n)}$ obtained by translating the lifted preprojective representation~$\tilde P_m$ by the action of $w\in W_n$.
To simplify the notation, we abbreviate $\tilde P_{2l-1,j}:=\tilde P_{2l-1,\alpha_j}$ and $\tilde P_{2l,j}:=\tilde P_{2l,\alpha_j^{-1}}$.
\begin{lemma}
  \label{le:shifted subreps}
  For each $m\ge1$, the representation $\tilde P_{m+1}$ has precisely $n$ subrepresentations covering $P_m$.
  These are the representations $\tilde P_{m,j}$ corresponding to the $n$ different arrows of $K(n)$. 
\end{lemma}
\begin{proof}
  This is clear for $m=1$ and follows in general by applying the recursion \eqref{eq:recursive covers}.
\end{proof}
\begin{remark}
  Note that the dimension vectors of the lifted preprojectives $\tilde P_m$ are symmetric under permutations of the arrows of $K(n)$ (or rather the corresponding operation on $\widetilde{K(n)}$).
  In particular, they all have the same central vertex $(1,e)$ if $m$ is odd and central vertex $(2,e)$ if $m$ is even.
\end{remark}

\begin{corollary}
  \label{cor:AR lift}
  For $m\ge2$, the Auslander-Reiten sequence~\eqref{eq:AR sequence} on $K(n)$ lifts to the Auslander-Reiten sequence
  \begin{equation}
    \label{eq:lifted AR sequence}
    \xymatrix{0 \ar[r] & \tilde P_{m-1} \ar^(.4){\tilde\iota_{m-1}}[r] & \bigoplus_{j=1}^n \tilde P_{m,j} \ar[r] & \tilde P_{m+1} \ar[r] & 0}.
  \end{equation}
\end{corollary} 

\begin{example}
  \label{ex:lifted preprojectives}
  Here we explicitly describe the preprojective lifts $\tilde P_m$ for $m=2,3,4$ as well as their shifted preprojective subrepresentations as in Lemma~\ref{le:shifted subreps}.
  By Lemma~\ref{le:lifts of transjectives}, the lifts of all preprojectives are exceptional as representations of $\widetilde{K(n)}$ and thus they are uniquely determined by their dimension vectors.
  We make use of this fact below, stating only the support of the representation (also specifying those dimensions which are not one) and do not state the particular maps present in the lifts.
  We call this the \emph{support quiver} of the representation and let $W_n$ act on these quivers by translating all vertices and arrows.

  The representation $\tilde P_2$ is defined by the following quiver:
  \begin{equation}
    \label{dia:P2 quiver}
    \xymatrix@R40pt@C10pt{&(2,e)\ar_{\alpha_1}[ld]\ar^{\alpha_n}[rd]&\\ (1,\alpha_1)&\cdots&(1,\alpha_n).}
  \end{equation}
  Then $\tilde P_{1,j}\subset\tilde P_2$ corresponds to the one-dimensional space at vertex $(1,\alpha_j)$.
  Write $\cA_i$ for the quiver obtained from the one above by erasing the arrow $\alpha_i$ and the corresponding sink. 

  The representation $\tilde P_3$ is given by the quiver
  \begin{equation}
    \label{dia:P3 quiver}
    \xymatrix@R40pt@C10pt{&n-1&\\ \alpha_1^{-1}.\cA_1\ar^{\alpha_1}[ru]&\cdots&\alpha_n^{-1}.\cA_n,\ar_{\alpha_n}[lu]}
  \end{equation}
  where the top vertex of dimension $n-1$ is $(1,e)$ and the arrow from each $\alpha_j^{-1}.\cA_j$ emanates from its unique source~$(2,\alpha_j^{-1})$.
  Then $\tilde P_{2,j}\subset\tilde P_3$ has one-dimensional spaces at each vertex of the quiver
  
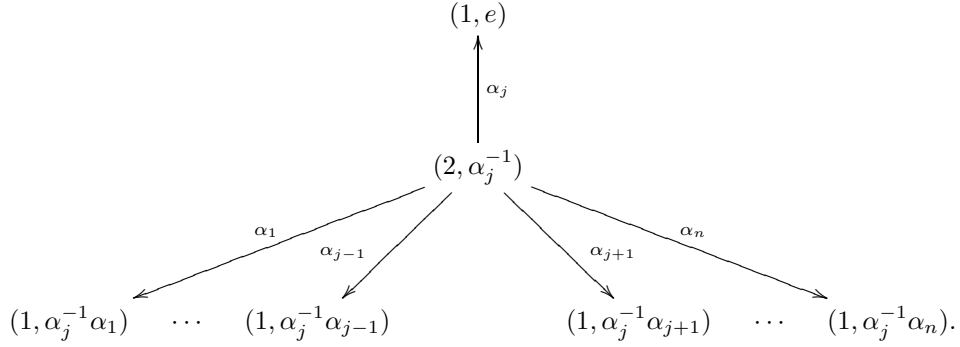
\[\xymatrix@R40pt@C10pt{&&&(1,e)&&&\\ &&&(2,\alpha_j^{-1})\ar_{\alpha_j}[u]\ar_{\alpha_1}[llld]\ar_(.65){\alpha_{j-1}}[ld]\ar^(.65){\alpha_{j+1}}[rd]\ar^{\alpha_n}[rrrd]&&&\\ (1,\alpha_j^{-1}\alpha_1)&\cdots&(1,\alpha_j^{-1}\alpha_{j-1})&&(1,\alpha_j^{-1}\alpha_{j+1})&\cdots&(1,\alpha_j^{-1}\alpha_n).}\]

  The representation $\tilde P_4$ has support quiver
  
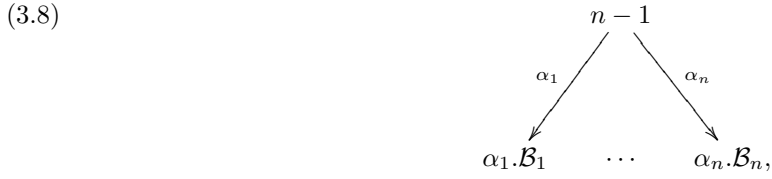
\begin{equation}
    \label{dia:P4 quiver}
    \xymatrix@R40pt@C10pt{&n-1\ar_{\alpha_1}[ld]\ar^{\alpha_n}[rd]&\\ \alpha_1.\cB_1&\cdots&\alpha_n.\cB_n,}
  \end{equation}
  where the top vertex of dimension $n-1$ is $(2,e)$ and $\cB_i$ is the following analogue of the support quiver for $\tilde P_3$:
  
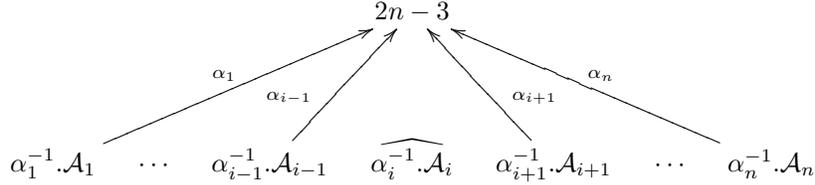
\[\xymatrix@R40pt@C10pt{&&&2n-3&&&\\
    \alpha_1^{-1}.\cA_1\ar^{\alpha_1}[rrru]&\cdots&\alpha_{i-1}^{-1}.\cA_{i-1}\ar^(.35){\alpha_{i-1}}[ru]&\widehat{\alpha_i^{-1}.\cA_i}&\alpha_{i+1}^{-1}.\cA_{i+1}\ar_(.35){\alpha_{i+1}}[lu]&\cdots&\alpha_n^{-1}.\cA_n\ar_{\alpha_n}[lllu]}\]
  with top vertex $(1,e)$ having dimension $2n-3$.
  Now $\tilde P_{3,j}$ can be found as the subrepresentation corresponding to the subquiver
  
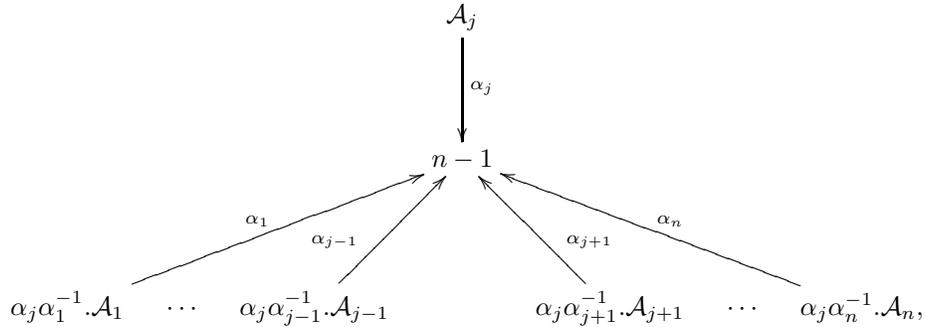
\[\xymatrix@R40pt@C10pt{
    &&&\cA_j\ar^{\alpha_j}[d]&&&\\
    &&&n-1&&&\\
    \alpha_j\alpha_1^{-1}.\cA_1\ar^{\alpha_1}[rrru]&\cdots&\alpha_j\alpha_{j-1}^{-1}.\cA_{j-1}\ar^(.35){\alpha_{j-1}}[ru]&&\alpha_j\alpha_{j+1}^{-1}.\cA_{j+1}\ar_(.35){\alpha_{j+1}}[lu]&\cdots&\alpha_j\alpha_n^{-1}.\cA_n,\ar_{\alpha_n}[lllu]}\]
  where the central vertex is $(1,\alpha_j)$.
  Note that taking the subquiver $\cA_j$ together with the image of the map $\alpha_j$ gives a subrepresentation of $\tilde P_{3,j}$ isomorphic to $\tilde P_2$ while taking the subquiver $\alpha_j\alpha_i^{-1}.\cA_i$ together with the image of the map $\alpha_i$ gives a subrepresentation of $\tilde P_{3,j}$ isomorphic to $\tilde P_{2,\alpha_j\alpha_i^{-1}}$.
\end{example}

The next result establishes some basic homological properties of the translated preprojective representations.
\begin{lemma}
  \label{le:homdecomposition}
  For $m\geq 1$, the following hold.
  \begin{enumerate}
    \item We have $\Hom(P_m,P_{m+1})\cong\bigoplus\limits_{i=1}^n \Hom(\tilde P_{m,i},\tilde P_{m+1})$, where each $\Hom(\tilde P_{m,i},\tilde P_{m+1})$ is one-dimensional.
    \item The representations $\tilde P_{m,i}$ are pairwise orthogonal, i.e.\ for $i\neq j$ we have 
      \[\Hom(\tilde P_{m,i},\tilde P_{m,j})=0=\Ext(\tilde P_{m,i},\tilde P_{m,j}).\]
    \item For $j\in\{1,\ldots,n\}$, we have
      \[\Hom(\tilde P_{m+1},\tilde P_{m,j})=0,\quad\Ext(\tilde P_{m+1},\tilde P_{m,j})=0,\quad\Ext(\tilde P_{m,j},\tilde P_{m+1})=0.\]
    \item For each proper subset $I\subsetneq\{1,\ldots,n\}$, there exists a truncated preprojective representation~$\tilde P_{m+1}^I$ fitting into an exact sequence 
      \begin{equation}
        \label{ses1}
        \sesm{\bigoplus_{i\in I} \tilde P_{m,i}}{\tilde P_{m+1}}{\tilde P_{m+1}^I}{\pi_{m+1}^I}.
      \end{equation}
      Moreover, $G(\tilde P_{m+1}^I)$ is a truncated preprojective of $K(n)$ for each $I\subsetneq\{1,\ldots,n\}$.
  \end{enumerate}
\end{lemma}
\begin{remark}
  Note that when $I=\varnothing$, it follows from the definition that $\tilde P_{m+1}^I=\tilde P_{m+1}$.
\end{remark}
\begin{proof}
  Part (1) is immediate from Theorem~\ref{covering}.
  Part (2) is also a consequence of Theorem~\ref{covering}.
  Indeed, for $1\leq j\leq n$, we have 
  $$\CC\cong\Hom(P_m,P_m)\cong\bigoplus_{i=1}^n \Hom(\tilde P_{m,i},\tilde P_{m,j}).$$
  But $\Hom(\tilde P_{m,j},\tilde P_{m,j})\cong\CC$ and so we must have $\Hom(\tilde P_{m,i},\tilde P_{m,j})=0$ for $i\neq j$.
  The vanishing of $\Ext(\tilde P_{m,i},\tilde P_{m,j})$ follows in the same manner using that $P_m$ is exceptional.

  Part (3) is clear for $m=1$ and follows for $m\ge2$ by applying the reflection recursion \eqref{eq:recursive covers}.

  For part (4), observe that under the isomorphism from part (1) the subset $I\subsetneq\{1,\ldots,n\}$ corresponds to the subspace $V\subset \cH_m$ spanned by the generators of the direct summands $\Hom(\tilde P_{m,i},\tilde P_{m+1})$ for $i\in I$.
  The map $\bigoplus_{i\in I} \tilde P_{m,i}\to\tilde P_{m+1}$ is then a lift of the evaluation morphism $ev_V:P_m\otimes V\to P_{m+1}$ and hence is injective by Lemma~\ref{le:injective evaluation maps}.
  Taking the cokernel defines the truncated preprojective $\tilde P_{m+1}^I$ and the preceding discussion shows that $G(\tilde P_{m+1}^I)\cong P_{m+1}^V$ is truncated preprojective as well.
\end{proof}

It will be important to understand the possible subrepresentations of the truncated preprojective representations $\tilde P_{m+1}^I$.
\begin{lemma}
  \label{le:subrep}
  For $m\ge1$ and $I\subsetneq\{1,\ldots,n\}$, all non-trivial proper subrepresentations of $\tilde P_{m+1}^I$ are direct sums of preprojective representations.	
\end{lemma}
\begin{proof}
  By Lemma~\ref{le:homdecomposition}, the projected representation $G(\tilde P_{m+1}^I)$ is a truncated preprojective of $K(n)$.
  Then Lemma~\ref{le:projective subrepresentations} shows that all subrepresentations of $G(\tilde P_{m+1}^I)$ are direct sums of preprojective representations of $K(n)$.
  But any subrepresentation of $\tilde P_{m+1}^I$ also projects to a subrepresentation of $G(\tilde P_{m+1}^I)$.
  Since, following Lemma~\ref{le:lifts of transjectives}, any lift of a preprojective of $K(n)$ will be a preprojective representation of $\widetilde{K(n)}$, this gives the result.
\end{proof}

In what follows, we will need to carefully understand the homological properties of the truncated preprojectives for $\widetilde{K(n)}$.
\begin{lemma}
  \label{le:properties}
  For $m\geq1$ and $I\subsetneq\{1,\ldots,n\}$, the following hold.
  \begin{enumerate}
    \item For $j\in\{1,\ldots,n\}$, we have $\Ext(\tilde P_{m,j},\tilde P_{m+1}^I)=0$.  Also, $\Hom(\tilde P_{m,j},\tilde P_{m+1}^I)\ne0$ if and only if $j\not\in I$, in which case 
      \[\Hom(\tilde P_{m,j},\tilde P_{m+1}^I)\cong\Hom(\tilde P_{m,j},\tilde P_{m+1})\cong\CC.\]
    \item We have $\Hom(\tilde P_{m+1},\tilde P_{m+1}^I)\cong\CC$ and $\Ext(\tilde P_{m+1},\tilde P_{m+1}^I)=0$. 
    \item For $j\in\{1,\ldots,n\}$, we have $\Hom(\tilde P_{m+1}^I,\tilde P_{m,j})=0$.
      Also, $\Ext(\tilde P_{m+1}^I,\tilde P_{m,j})\ne0$ if and only if $j\in I$, in which case  
      \[\Ext(\tilde P_{m+1}^I,\tilde P_{m,j})\cong\CC.\] 
    \item For any $J\subsetneq\{1,\ldots,n\}$, we have $\Hom(\tilde P_{m+1}^J,\tilde P_{m+1}^I)\ne0$ if and only if $J\subset I$, in which case
      \[\Hom(\tilde P_{m+1}^J,\tilde P_{m+1}^I)\cong\CC\quad\text{and}\quad\Ext(\tilde P_{m+1}^J,\tilde P_{m+1}^I)=0.\]
      In particular, $\tilde P_{m+1}^I$ is an exceptional representation of $\widetilde{K(n)}$.
  \end{enumerate}
\end{lemma} 
\begin{proof}
  Applying $\Hom(\tilde P_{m,j},-)$ to the sequence \eqref{ses1} gives the exact sequence
  \[\xymatrix{0 \ar[r] & \Hom(\tilde P_{m,j},\bigoplus_{i\in I} \tilde P_{m,i}) \ar[r] & \Hom(\tilde P_{m,j},\tilde P_{m+1}) \ar[r] & \Hom(\tilde P_{m,j},\tilde P_{m+1}^I) \ar[r] & 0},\]
  where the final zero follows from Lemma~\ref{le:homdecomposition}.(2).
  This also gives an isomorphism $\Ext(\tilde P_{m,j},\tilde P_{m+1}^I)\cong\Ext(\tilde P_{m,j},\tilde P_{m+1})=0$ by Lemma~\ref{le:homdecomposition}.(3).
  Now the middle space in the sequence above is one-dimensional while the left-hand space vanishes if and only if $j\notin I$, this proves part (1).

  Part (2) is an immediate consequence of Lemma~\ref{le:unique preprojective morphism} together with Theorem~\ref{covering} or can be obtained directly by applying $\Hom(\tilde P_{m+1},-)$ to the sequence \eqref{ses1}.
  The first part of (3) follows from Theorem~\ref{covering} together with Lemma~\ref{le:basic homological properties}.
  For the second part of (3), we apply $\Hom(-,\tilde P_{m,j})$ to the sequence \eqref{ses1}.
  Then taking into account Lemma~\ref{le:homdecomposition} part (3), we get the isomorphism
  \[\Ext(\tilde P_{m+1}^I,\tilde P_{m,j})\cong\Hom(\textstyle{\bigoplus_{i\in I}}\tilde P_{m,i},\tilde P_{m,j}).\] 
  Then Lemma~\ref{le:homdecomposition}.(2) gives the final claim of part (3).

  Part (4) follows by applying $\Hom(-,\tilde P^I_{m+1})$ to the sequence \eqref{ses1} for $J$ then using parts (1) and (2).
\end{proof}

\begin{example}
  \label{ex:truncated lifts}
  Fix $I\subsetneq\{1,\ldots,n\}$.
  Building on Example~\ref{ex:lifted preprojectives}, we describe here the truncated preprojectives $\tilde P_m^I$ for $m=2,3,4$.
  Following Lemma~\ref{le:properties}.(4), we can do this by simply specifying their dimension vectors as we did above.

  The support quiver of $\tilde P_2^I$ is obtained from the quiver \eqref{dia:P2 quiver} of $\tilde P_2$ by removing the sinks $(1,\alpha_i)$ for $i\in I$.
  The support quiver of $\tilde P_3^I$ is obtained from the quiver \eqref{dia:P3 quiver} of $\tilde P_3$ by removing the subquivers $\alpha_i^{-1}.\cA_i$ for $i\in I$ and decreasing the dimension of the space at vertex $(1,e)$ by $|I|$.

  The support quiver of $\tilde P_4^I$ is given by the following analogue of the quiver \eqref{dia:P4 quiver} of $\tilde P_4$:
  \[\xymatrix@R40pt@C10pt{&n-1-|I|\ar_{\alpha_1}[ld]\ar^{\alpha_n}[rd]&\\ \alpha_1.\cB_1^I&\cdots&\alpha_n.\cB_n^I,}\]
  where the top vertex of dimension $n-1-|I|$ is again $(2,e)$ and $\cB_i^I$ is simply a space of dimension~$n-1-|I|$ if $i\in I$ and otherwise is the quiver
  \[\xymatrix@R40pt@C10pt{&&&2n-3-|I|&&&\\
    \alpha_1^{-1}.\cA_1\ar^{\alpha_1}[rrru]&\cdots&\alpha_{i-1}^{-1}.\cA_{i-1}\ar^(.35){\alpha_{i-1}}[ru]&\widehat{\alpha_i^{-1}.\cA_i}&\alpha_{i+1}^{-1}.\cA_{i+1}\ar_(.35){\alpha_{i+1}}[lu]&\cdots&\alpha_n^{-1}.\cA_n.\ar_{\alpha_n}[lllu]}\]
\end{example} 

\begin{lemma}
  \label{le:truncated tau}
  For $m\ge3$ and $I\subsetneq\{1,\ldots,n\}$, we have $\tilde\Sigma_2(\tilde P_{m-1}^I)^\sigma=\tilde P_m^I$ and $\tau\tilde P_{m+1}^I=\tilde P_{m-1}^I$.
\end{lemma}
\begin{proof}
  By Lemma~\ref{le:homdecomposition}, we get the short exact sequence \eqref{ses1} defining the truncated preprojective $\tilde P_{m-1}^I$:
  \[\ses{\bigoplus_{i\in I} \tilde P_{m-2,i}}{\tilde P_{m-1}}{\tilde P_{m-1}^I}.\]
  Applying the functor $\tilde\Sigma_2(-)^\sigma$ to this sequence and recalling the reflection recursion \eqref{eq:recursive covers}, we obtain a short exact sequence
  \[\ses{\bigoplus_{i\in I} \tilde P_{m-1,i}}{\tilde P_m}{\tilde\Sigma_2(\tilde P_{m-1}^I)^\sigma}.\]
  The equality $\tilde\Sigma_2(\tilde P_{m-1}^I)^\sigma=\tilde P_m^I$ immediately follows.
  The equality $\tilde P_{m+1}^I=\tau^{-1}\tilde P_{m-1}^I$ is obtained in the same way using the functor $\tilde\Sigma_1\tilde\Sigma_2=\tau^{-1}$.
\end{proof}

We now introduce notation for locating specific lifted preprojectives as subrepresentations of our standard lifted preprojective representations.
Since we work only on two fixed components of $\widetilde{K(n)}$, the following notation will be useful in describing paths in these components.
For $k\ge1$, set
\[A_1^{(k)}:=\{(i_1,\ldots,i_k)\mid i_j\in\{1,\ldots,n\}\text{ for $1\le j\le k$}\}.\]
Depending on context (in particular, the parity of $m+1$), we will sometimes identify the word~$(i_1,\ldots,i_k)$ with the element $\alpha_{i_1}\alpha_{i_2}^{-1}\alpha_{i_3}\cdots\alpha_{i_k}^{(-1)^{k+1}}\in W_n$ and sometimes with the element $\alpha_{i_1}^{-1}\alpha_{i_2}\alpha_{i_3}^{-1}\cdots\alpha_{i_k}^{(-1)^k}\in W_n$.
In this way, to each word $\ui=(i_1,\ldots,i_k)\in A_1^{(k)}$ with $1\leq k\leq m$, we associate a preprojective subrepresentation~$\tilde P_{m+1-k,\ui}\subset\tilde P_{m+1}$ which lifts $P_{m+1-k}$.
More precisely, we obtain a sequence of preprojective subrepresentations which uniquely determines the desired inclusion:
\[\tilde P_{m+1-k,\ui}\subset\tilde P_{m+2-k,(i_1,\ldots,i_{k-1})}\subset\ldots\subset\tilde P_{m-1,(i_1,i_2)}\subset\tilde P_{m,i_1}\subset\tilde P_{m+1}.\]
Note that, although there is a translate of $\tilde P_{m+1-k}$ corresponding to each word $\ui\in A_1^{(l)}$ for $1\le l<k$, these will not be naturally equipped with a canonical inclusion to $\tilde P_{m+1}$.

To emphasize this point, consider a word $\ui\in A_1^{(k)}$ with $i_j=i_{j+1}$ for some $j$ and write $\ui'\in A_1^{(k-2)}$ for the word obtained from $\ui$ by removing the terms $i_j$ and $i_{j+1}$.
Then the representations $\tilde P_{m+1-k,\ui}$ and $\tilde P_{m+1-k,\ui'}$ are in fact equal, however $\tilde P_{m+1-k,\ui}$ is naturally identified as a subrepresentation of $\tilde P_{m+1}$ while $\tilde P_{m+1-k,\ui'}$ is not.
Indeed, by considering the support quiver of $\tilde P_{3,i}$ from Example~\ref{ex:lifted preprojectives}, we see that each $\tilde P_{2,(i,i)}$ is just a copy of $\tilde P_2$ when viewed as representations of $\widetilde{K(n)}$, however these provide distinct subrepresentations of $\tilde P_4$.
\begin{lemma}
  \label{le:special subrepresentations}
  Consider a non-empty subset $I\subsetneq\{1,\ldots,n\}$ and fix an element $j\in I$. 
  \begin{enumerate}
    \item For $m\ge2$, we have $\Hom(\tilde P_{m,j},\tau\tilde P_{m+1}^I)\cong\CC$.
      Moreover, the kernel of a nonzero morphism $\tilde P_{m,j}\to\tau\tilde P_{m+1}^I$ is the following representation 
      \[\tilde P_m(I,j):=
        \begin{cases}
          \bigoplus_{\substack{1\leq i\leq n\\i\neq j}}\tilde P_{m-1,(j,i)}\oplus \bigoplus_{i\in I, i\ne j}\tilde P_{m-2,(j,j,i)} & \text{ if $m\geq 3$;}\\
          \bigoplus_{\substack{1\leq i\leq n\\i\neq j}}\tilde P_{1,(j,i)} & \text{ if $m=2$.}
        \end{cases}\]
    \item For $m\ge3$, any nonzero morphism $\tilde P_{m,j}\to\tau\tilde P_{m+1}^I$ is surjective.
  \end{enumerate}
\end{lemma}
\begin{proof}
  The first claim of part (1) follows immediately by applying the Auslander-Reiten formulas \cite[Theorem IV.2.13]{ass} to Lemma~\ref{le:properties}.(3).
  Indeed, this gives
  \[\dim\Hom(\tilde P_{m,j},\tau\tilde P_{m+1}^I)=\dim\Ext(\tilde P_{m+1}^I,\tilde P_{m,j})=1.\]
  We establish the final claim of part (1) directly for $m=2,3$ and then deduce the general case by applying the reflection recursions \eqref{eq:recursive covers}.
  Using the description in Example~\ref{ex:truncated lifts}, it is not hard to see that $\tau\tilde P_3^I$ is indecomposable with one-dimensional spaces at only the vertices $(1,e)$ and $(2,\alpha_i^{-1})$ for $i\in I$.
  But then for~$j\in I$, the image of the unique homomorphism $\tilde P_{2,j}\to\tau\tilde P_3^I$ is the representation with support quiver
  \[(2,\alpha_j^{-1})\xrightarrow{\alpha_j}(1,e).\]
  From the support quiver of $\tilde P_{2,j}$ given in Example~\ref{ex:lifted preprojectives}, we see that the kernel of this map is precisely $\tilde P_2(I,j)$.

  For the $m=3$ case, we note that $\tau\tilde P_4^I=\tilde P_2^I$ by Lemma~\ref{le:truncated tau}.
  Then the claimed structure $\tilde P_3(I,j)$ of the kernel and the surjectivity of the map $\tilde P_{3,j}\to\tau\tilde P_4^I$ are immediate from the explicit descriptions of $\tilde P_{3,j}$ and $\tilde P_2^I$ in Example~\ref{ex:lifted preprojectives} and Example~\ref{ex:truncated lifts} respectively.
  The general cases for parts (1) and (2) both then follow using the reflection recursions \eqref{eq:recursive covers} taking into account Remark~\ref{rem:reflection recursion}.(2).
\end{proof}

\begin{remark}
  \label{rem:special case}
  We should point out that the case $m=2$ of Lemma~\ref{le:special subrepresentations} is rather special because it is the only one for which the unique morphism $\tilde P_{m,j}\to\tau\tilde P_{m+1}^I$ is not surjective.
  Indeed, recall from the proof of Lemma~\ref{le:special subrepresentations} that the image of the unique homomorphism $\tilde P_{2,j}\to\tau\tilde P_3^I$ is the representation with support quiver
  \begin{equation}
    \label{eq:special case}
    (2,\alpha_j^{-1})\xrightarrow{\alpha_j}(1,e).
  \end{equation}
  If we factor out the image of this morphism from $\tau\tilde P_3^I$, the remaining representation is a direct sum of the simple injective representations corresponding to the vertices $(2,\alpha_i^{-1})$ for $i\in I$, $i\ne j$.
  Note that these disappear after reflecting at all sources.
\end{remark}

The following orthogonality property is a primary reason we need to lift to the universal cover of $K(n)$.
\begin{corollary}
  \label{cor:perpendicular}
  Consider a non-empty subset $I\subsetneq\{1,\ldots,n\}$ and fix an element $j\in I$.
  For $m\geq 2$, we have $\tilde P_m(I,j)\in (\tilde P_{m+1}^I)^\perp$.
\end{corollary}
\begin{proof}
  If $m\geq 3$, we consider the long exact sequence obtained when applying $\Hom(\tilde P_{m+1}^I,-)$ to the sequence
  \[\ses{\tilde P_m(I,j)}{\tilde P_{m,j}}{\tau\tilde P_{m+1}^I}.\]
  Following Lemma~\ref{le:properties}.(3), we have $\Hom(\tilde P_{m+1}^I,\tilde P_{m,j})=0$ and using the long exact sequence this immediately implies $\Hom\big(\tilde P_{m+1}^I,\tilde P_m(I,j)\big)=0$.
  From Lemma~\ref{le:properties}.(3) again, we have $\Ext(\tilde P_{m+1}^I,\tilde P_{m,j})\cong\CC$.
  Using the Auslander-Reiten formulas \cite[Theorem IV.2.13]{ass} and Lemma~\ref{le:properties}.(4), we get
  \[\dim\Ext(\tilde P_{m+1}^I,\tau\tilde P_{m+1}^I)=\dim\Hom(\tilde P_{m+1}^I,\tilde P_{m+1}^I)=1.\]
  It follows that the surjective map $\Ext(\tilde P_{m+1}^I,\tilde P_{m,j})\to\Ext(\tilde P_{m+1}^I,\tau\tilde P_{m+1}^I)$ appearing in the long exact sequence is in fact an isomorphism.
  But the Auslander-Reiten formulas and Lemma~\ref{le:properties}.(4) again imply 
  \[\dim\Hom(\tilde P_{m+1}^I,\tau\tilde P_{m+1}^I)=\dim\Ext(\tilde P_{m+1}^I,\tilde P_{m+1}^I)=0.\]
  Combining with the preceding discussion, this gives $\Ext\!\big(\tilde P_{m+1}^I,\tilde P_m(I,j)\big)=0$.
  
  If $m=2$, we have $\Hom\!\big(\tilde P_3^I,\tilde P_2(I,j)\big)=0$ since $\tilde P_2(I,j)$ is a direct sum of simple projective representations.
  Using the explicit description of $\tilde P_3^I$ from Example~\ref{ex:truncated lifts}, we see that each $\tilde P_{1,(j,i)}$ for $i\ne j$ is supported at a vertex which is not a neighbor of the support of $\tilde P_3^I$ and this implies $\Ext\!\big(\tilde P_3^I,\tilde P_{1,(j,i)}\big)=0$ for $i\neq j$.
\end{proof}

The next step is to introduce notation in order to locate truncated preprojectives as quotients of other truncated preprojectives in the universal covering.
\begin{definition}
  For $I\subsetneq\{1,\ldots,n\}$, write $I^c=\{1,\ldots,n\}\setminus I$ for the complementary subset.
  A sequence of subsets $\bfI=(I_0,I_1,\ldots,I_k)$, $k\ge0$, in $\{1,\ldots,n\}$ is \emph{admissible} if the following hold:
  \begin{enumerate}
    \item if $k\ge1$, then $|I_0|=n-1$ and $|I_l|=n-2$ for $1\le l\le k-1$;
    \item the sets $I'_l$ defined recursively by $I'_0=I_0$ and $I'_{l+1}:=I_{l+1}\cup (I'_l)^c$ for $0\le l\le k$ satisfy $(I'_l)^c\cap I_{l+1}=\varnothing$ for $0\le l\le k$.
  \end{enumerate}
  Here we take $I_{k+1}=\varnothing$ so that there is no condition imposed on the set $(I'_k)^c$.
  Given an admissible sequence $\bfI=(I_0,\ldots,I_k)$ with $k\ge1$, define a new admissible sequence $\delta\bfI=(I'_1,I_2,\ldots,I_k)$.
\end{definition}

In the same way as for the preprojective lifts $\tilde P_{m+1}$, $m\ge1$, we can define truncated preprojectives $\tilde P_{m+1,w}^I$ of any translate $\tilde P_{m+1,w}$, where $w\in W_n$ and $I\subsetneq\{1,\ldots,n\}$.
That is, taking $\varepsilon=(-1)^{m+1}$ we have an exact sequence
\[\ses{\bigoplus_{i\in I} \tilde P_{m,w\alpha_i^\varepsilon}}{\tilde P_{m+1,w}}{\tilde P_{m+1,w}^I}.\]
These quotients are unique in the sense that $\Hom(\tilde P_{m+1,w},\tilde P_{m+1,w}^I)=\CC$ for all $w\in W_n$ and $I\subsetneq\{1,\ldots,n\}$.
We adopt similar notation for truncated preprojectives $\tilde P_{m+1-k,\ui}^I$ for $\ui\in A_1^{(k)}$ and $I\subsetneq\{1,\ldots,n\}$.

By Lemma~\ref{le:properties}.(1), we can quotient out the lifted preprojectives successively, i.e.\ for any proper subsets $J\subsetneq I\subsetneq\{1,\ldots,n\}$ we have a short exact sequence
\begin{equation}
  \label{eq:truncted sequence}
  \sesm{\bigoplus_{i\in I\setminus J}\tilde P_{m,i}}{\tilde P_{m+1}^J}{\tilde P_{m+1}^I}{\pi_{m+1}^{J,I}}.
\end{equation}
For any sequence of proper subsets $K\subsetneq J\subsetneq I\subsetneq\{1,\ldots,n\}$, the quotient maps satisfy $\pi_{m+1}^{K,I}=\pi_{m+1}^{J,I}\circ\pi_{m+1}^{K,J}$ and $\pi_{m+1}^I=\pi_{m+1}^{J,I}\circ\pi_{m+1}^J$.
For $w\in W_n$, we write $\pi_{m+1,w}^{I,J}:\tilde P_{m+1,w}^I\to\tilde P_{m+1,w}^J$ for the translated morphism with similar notation for truncated preprojectives $\tilde P_{m+1-k,\ui}^I$ for $\ui\in A_1^{(k)}$ and $I\subsetneq\{1,\ldots,n\}$.
\begin{lemma}
  \label{le:truncated quotients}
  For $m\ge3$, the following hold:
  \begin{enumerate}
    \item For $I\subsetneq\{1,\ldots,n\}$ with $|I|=n-1$ and $I^c=\{j\}$, we have an isomorphism 
      \[\tilde P_{m+1}^I\cong\tilde P_{m,j}^{I^c}=\tilde P_{m,j}/\tilde P_{m-1,(j,j)}.\]
    \item Consider an admissible sequence $\bfI=(I_0,\ldots,I_k)$ with $0<k<m$ and write $(I'_l)^c=\{i_l\}=:J_l$ for $0\le l\le k-1$.
      For $\ui=(i_0,\ldots,i_{k-1})$, there exists a commutative diagram:
      \begin{equation}
        \label{dia:truncated coherence}
        \xymatrix@R40pt@C70pt{\tilde P_{m+1} \ar^(.4){\pi_{m+1}^{I'_0}}[d] & \ar[l] \tilde P_{m,i_0} \ar_(.4){\pi_{m,i_0}^{J_0}}[dl]\ar^(.4){\pi_{m,i_0}^{I'_1}}[d] & \ar[l] \tilde P_{m-1,(i_0,i_1)} \ar_(.4){\pi_{m-1,(i_0,i_1)}^{J_1}\;}[dl]\ar^(.4){\pi_{m-1,(i_0,i_1)}^{I'_2}}[d] & \ar[l] \ar[dl] \cdots & \ar[l] \tilde P_{m+1-k,\ui} \ar_(.4){\pi_{m+1-k,\ui}^{J_{k-1}}\;}[dl]\ar^(.4){\pi_{m+1-k}^{I'_k}}[d]\\
        \tilde P_{m+1}^{I'_0}\ar^{\pi_{m,i_0}^{J_0,I'_1}}[r] & \tilde P_{m,i_0}^{I'_1}\ar^{\pi_{m-1,(i_0,i_1)}^{J_1,I'_2}}[r] & \tilde P_{m-1,(i_0,i_1)}^{I'_2}\ar^(.6){\pi_{m-2,(i_0,i_1,i_2)}^{J_2,I'_3}}[r] & \cdots\ar^{\pi_{m+1-k,\ui}^{J_{k-1},I'_k}}[r] & \tilde P_{m+1-k,\ui}^{I'_k}}
      \end{equation}
      where the composed map $\pi_{m+1}^\bfI:\tilde P_{m+1}\to\tilde P_{m+1-k,\ui}^{I'_k}$ is unique up to scaling and therefore we write
      \[\tilde P_{m+1}^\bfI:=\tilde P_{m+1-k,\ui}^{I'_k}.\]
      These truncated preprojective representations satisfy 
      \[\tilde P_{m+1}^\bfI=\tilde P_{m,i_0}^{\delta\bfI}=\cdots=\tilde P_{m+1-k,\ui}^{\delta^k\bfI}.\]
  \end{enumerate}
\end{lemma}
\begin{proof}
  The Auslander-Reiten sequence~\eqref{eq:lifted AR sequence} gives rise to the following commutative diagram:
  \[\xymatrix{& & 0\ar[d]& 0\ar[d]&\\
    & & \bigoplus_{i\in I} \tilde P_{m,i}\ar[d]\ar@{=}[r]& \bigoplus_{i\in I} \tilde P_{m,i}\ar[d]&\\
    0\ar[r]&\tilde P_{m-1} \ar[r]\ar@{=}[d] & \bigoplus_{i=1}^n \tilde P_{m,i}\ar[d]\ar[r]\ar@{}[dr]|(.7){\lrcorner} & \tilde P_{m+1}\ar[r]\ar[d]& 0\\
    0\ar[r] &\tilde P_{m-1}\ar[r] &  \tilde P_{m,j}\ar[r]\ar[d] & \tilde P_{m+1}^I\ar[r]\ar[d] & 0\\
  & & 0& 0&}\]
  The image of the inclusion $\tilde P_{m-1}\into \tilde P_{m,j}$ in the bottom row is the subrepresentation $\tilde P_{m-1,(j,j)}\subset\tilde P_{m,j}$ and part (1) follows.

  The first claim of part (2) is then immediate by repeatedly applying part (1) while the final claim of part (2) is a consequence of the definition of $\delta$.
\end{proof}

\section{Quiver Grassmannians}
\label{QG}

\noindent 
In this section, we aim to establish the existence of cell decompositions for quiver Grassmannians of (truncated) preprojective representations of $K(n)$ and its universal covering quiver.
By a \emph{cell decomposition} of an algebraic variety $X$, we mean a filtration $\varnothing=X_{k+1}\subset X_k\subset \cdots \subset X_2\subset X_1=X$ of $X$ by closed subsets $X_i\subset X$ so that each $X_i\setminus X_{i+1}$ is isomorphic to an affine space.
Alternatively, a cell decomposition of $X$ is a collection of disjoint locally closed subsets $U_1,\ldots,U_k\subset X$, each isomorphic to an affine space, such that each $X_i=U_i\sqcup U_{i+1}\sqcup\cdots\sqcup U_k$ is closed in $X$ with $X_1=X$.
We call the subsets $U_i\subset X$ the \emph{affine cells} for this cell decomposition.
Given varieties $X$ and $Y$ each with cell decompositions, we may choose an ordering on products of their affine cells (e.g.\ lexicographic) to get a cell decomposition of $X\times Y$. 
Given a variety $X$ with a cell decomposition, we call a subvariety $U\subset X$ \emph{compatible} with the cell decomposition if $U$ can be written as the union over a subset of the affine cells for $X$.
In this case, $U$ also has a cell decomposition given by taking exactly those affine cells for $X$ which are contained in $U$.

\subsection{Torus Actions and the Bia\l{}ynicki-Birula Decomposition}
\label{sec:bb}

The aim of Section~\ref{torusaction} is to define a $\CC^*$-action on quiver Grassmannians which can be used to simplify the calculation of homological invariants in general.
If the quiver Grassmannian is smooth, which is for instance the case for exceptional representations by \cite{cr}, it can also be used to stratify the quiver Grassmannians using the results of Bia\l{}ynicki-Birula.
More specifically, let $X$ be a smooth projective variety with a $\C^\ast$-action.
For a connected component of the fixed point set $C\subset X^{\C^\ast}$, we define its attracting set as
\[\Att(C):=\{y\in X\mid \lim_{t\to 0}t.y\in C\}.\]
The following result of Bia\l{}ynicki-Birula relates the geometry of $X$ to the geometry of its $\CC^*$-fixed points (see \cite[Section 4]{bb} or \cite[Section 4]{ca}).
\begin{theorem}
  \label{thm:bb}
  Let $X$ be a smooth projective complex variety with a $\C^\ast$-action.
  Then each attracting set $\Att(C)$ is a locally closed $\CC^*$-invariant subvariety of $X$ and the natural map $\Att(C)\to C$ is an affine bundle.
  Moreover, assuming $X^{\C^\ast}=\coprod_{i=1}^r C_i$ is a decomposition of the fixed point set of $X$ into finitely many connected components, we have $X=\coprod_{i=1}^r \Att(C_i)$, where we can choose an ordering such that $\coprod_{i=1}^s\Att(C_i)$ is closed for $1\leq s\leq r$.
  In particular, we have an equality of Euler characteristics $\chi(X)=\chi(X^{\C^\ast})$.
\end{theorem}

If each component $C_i$ admits a cell decomposition, Theorem~\ref{thm:bb} implies the same is true of $X$.
Indeed, we can trivialize each affine bundle $\Att(C_i)\to C_i$ over each affine piece of $C_i$ and then taking the natural ordering of the resulting affine spaces gives a cell decomposition of $X$.

\subsection{Torus Actions on Quiver Grassmannians}
\label{torusaction}

Fix a vector space $X$ of dimension $n$ and let $k\leq n$.
We first consider a natural $\CC^*$-action on the usual Grassmannian $\Gr_k(X)$ which is compatible with a given direct sum decomposition of the vector space $X$.
Then we generalize the concept to quiver Grassmannians and observe that the $\CC^*$-fixed point sets can be calculated in an analogous manner.

Given a basis $\mathcal B=\{v_1,\ldots,v_n\}$ of $X$ and a map $d:\{1,\ldots,n\}\to\ZZ$, we get a $\C^\ast$-action on $X$ when linearly extending the definition $t.v_r:=t^{d(r)}v_r$ for $r=1,\ldots,n$ to all of $X$.
This naturally induces an action of $\C^\ast$ on the Grassmannian $\Gr_k(X)$.
Our goal is to understand the fixed points of such an action.

For this recall that we can represent each subspace~$U\in\Gr_k(X)$ uniquely by a $k\times n$ matrix $M(U)$ whose rows provide a basis for $U$ when expanded as coefficient vectors in the basis $\mathcal B$.
The uniqueness of $M(U)$ comes from requiring that it be in row-echelon form, i.e.\ there exists a unique sequence $1\leq i_1<\ldots<i_k\leq n$ so that $M(U)$ is of the form
\[M(U):=
  \begin{pmatrix}
    \ast&\cdots &\ast &1&0 &\cdots& 0 & 0 & 0 &\cdots&0&0&0&\cdots&0\\
    \ast&\cdots &\ast&0&\ast&\cdots &\ast&1&0&\cdots&0&0&0&\cdots&0\\
    \ast &\cdots&\ast&0&\ast &\cdots&\ast&0&\ast&\cdots&0&0&0&\cdots&0\\[-0.4em]
    \vdots &\ddots&\vdots&\vdots&\vdots &\ddots&\vdots&\vdots&\vdots&\ddots&\vdots&\vdots&\vdots&\ddots&\vdots\\
    \ast&\cdots &\ast&0&\ast&\cdots &\ast&0&\ast&\cdots&0&0&0&\cdots&0\\
    \ast&\cdots &\ast&0&\ast&\cdots &\ast&0&\ast&\cdots&\ast&1&0&\cdots&0
  \end{pmatrix}\in M_{k,n}(\CC),\]
where the unit vectors are in the columns $\bfi=(i_1,\ldots,i_k)$.
The set of all $U\in\Gr_k(X)$ represented by matrices of this fixed form gives the \emph{Schubert cell} $X_\bfi$.

The $\C^\ast$-action on $U\in\Gr_k(X)$ can then be described in terms of the matrix representation $M(U)$, that is for $U\in X_\bfi$ we have
\[M(t.U)_{qr}=t^{d(r)-d(i_q)} M(U)_{qr}\]
for $q=1,\ldots,k$ and $r=1,\ldots,n$.
Observe that each Schubert cell $X_\bfi$ is invariant under this $\CC^*$-action.

Assume that $X=\bigoplus_{l=1}^m X_l$ is a direct sum decomposition of $X$ and fix a basis $\mathcal B=\{v_1,\ldots,v_n\}$ of~$X$ which is compatible with this decomposition, i.e.\ there exist indices $0=r_0<r_1<r_2<\cdots<r_{m-1}<r_m=n$ such that
\[v_{r_0+1},\ldots,v_{r_1}\in X_1,\quad v_{r_1+1},\ldots,v_{r_2}\in X_2,\quad\cdots\quad v_{r_{m-1}+1},\ldots,v_{r_m}\in X_m.\]
\begin{lemma}
  \label{le:usualGrass}
  Consider a map $d:\{1,\ldots,n\}\to\ZZ$ such that $d(r)=d(r')$ if $v_r,v_{r'}\in X_l$ for some $l$ and $d(r)\neq d(r')$ if $v_r\in X_l$ and $v_{r'}\in X_{l'}$ with $l\neq l'$.
  Then under the $\CC^*$-action determined by $d$, we have $U\in\Gr_k(X)^{\C^\ast}$ if and only if $U=\bigoplus_{l=1}^m U\cap X_l$.
\end{lemma}
\begin{proof}
  Assume $U=\bigoplus_{l=1}^m U\cap X_l$.
  Then any $u\in U$ can be written uniquely as $u=\sum_{l=1}^m u_l$ for some $u_l\in U\cap X_l$.
  It follows that $t.u=\sum_{l=1}^m t.u_l=\sum_{l=1}^m t^{d(r_l)}u_l\in\bigoplus_{l=1}^m U\cap X_l=U$ and thus $t.U=U$. 

  For the reverse direction, assume $U\in X_{\bfi}$ is a $\CC^*$-fixed point represented by the matrix $M(U)$.
  Then, if~$v_{i_q}\in X_l$, the assumptions on $d$ imply $M(U)_{qr}=0$ unless $r_{l-1}+1\le r\le r_l$.
  That is, $M(U)$ has the shape of a block matrix representing the decomposition $U=\bigoplus_{l=1}^m U\cap X_l$.
\end{proof}
The next step is to generalize this to quiver Grassmannians.
Let $Q$ be an acyclic quiver.
Choose a map $d:\hat Q_0\to\ZZ$ and fix a representation $X\in\rep Q$ which can be lifted to $\hat Q$.
We consider the decomposition $X_i=\bigoplus_{\chi\in A_Q} X_{(i,\chi)}$ and define a $\CC^*$-action on each $X_{(i,\chi)}$ via $t.x_{(i,\chi)}=t^{d(i,\chi)}x_{(i,\chi)}$ which is then extended linearly to each $X_i$.
Associated to each subspace $U_i$, there is a corresponding subspace $t.U_i$ for each $t\in\CC^*$.
In general, this does not induce a $\CC^*$-action on the quiver Grassmannians $\Gr_\bfe(X)$ since $t.U=(t.U_i)_{i\in Q_0}$ is not necessarily a subrepresentation of $X$ for every $U\in\Gr_\bfe(X)$.
Indeed, for this such an action must satisfy $X_\alpha(t.U_i)\subset t.U_j$ for every arrow $\alpha:i\to j$ of $Q$ and every $t\in\CC^*$. 
\begin{lemma}
  Let $X$ be a representation of $Q$ which can be lifted to $\hat Q$.
  Fix an integer $c_\alpha\in\ZZ$ for each~$\alpha\in Q_1$.
  Suppose $d:\hat Q_0\to\ZZ$ satisfies $d(j,\chi+e_\alpha)-d(i,\chi)=c_\alpha$ for each arrow $\alpha:i\to j$ of $Q$ and each $\chi\in A_Q$.
  Then the $\CC^*$-action on $X$ determined by $d$ induces a $\CC^*$-action on $\Gr_\bfe(X)$.
\end{lemma}
\begin{proof} 
  Fix $U\in\Gr_\bfe(X)$ and consider $u_i\in U_i$.
  Since $U$ is a subrepresentation, for an arrow $\alpha:i\to j$ of $Q$ we may write $X_\alpha(u_i)=u_j$ for some $u_j\in U_j$.

  As $X$ can be lifted to $\hat Q$, for any arrow $\alpha\in Q_1$ we can write $X_\alpha:X_i\to X_j$ as a block matrix consisting of linear maps $X_{(\alpha,\chi)}:X_{(i,\chi)}\to X_{(j,\chi+e_\alpha)}$ for $\chi\in A_Q$.
  Then writing $u_i=\sum_{\chi\in A_Q} u_{(i,\chi)}$ for some vectors $u_{(i,\chi)}\in X_{(i,\chi)}$, we have $X_\alpha(u_{(i,\chi)})=X_{(\alpha,\chi)}(u_{(i,\chi)})\in X_{(j,\chi+e_\alpha)}$, say $X_{(\alpha,\chi)}(u_{(i,\chi)})=u_{(j,\chi+e_\alpha)}$.
  It follows that $u_j=\sum_{\chi\in A_Q} u_{(j,\chi+e_\alpha)}$ and so
  \begin{align*}
    X_\alpha(t.u_i)
    &=\sum\limits_{\chi\in A_Q} t^{d(i,\chi)}X_{(\alpha,\chi)}(u_{(i,\chi)})=\sum\limits_{\chi\in A_Q} t^{d(i,\chi)}u_{(j,\chi+e_\alpha)}=t^{-c_\alpha}\sum\limits_{\chi\in A_Q} t.u_{(j,\chi+e_\alpha)}=t^{-c_\alpha} t.u_j.
  \end{align*}
  Therefore $X_\alpha(t.U_i)\subset t.U_j$ for every arrow $\alpha:i\to j$ of $Q$ and we obtain a $\CC^*$-action on $\Gr_\bfe(X)$.
\end{proof}

The next result provides the conditions on the map $d:\hat Q_0\to\ZZ$ needed to get an analogue of Lemma~\ref{le:usualGrass}
\begin{lemma}
  \label{le:degree condition} 
  Let $\hat X\in\rep\hat Q$ be an indecomposable representation of $\hat Q$.
  There exists $d:\supp(\hat X)\to\ZZ$ and $c_\alpha\in\mathbb N_+$ for each $\alpha\in Q_1$ such that
  \begin{enumerate}
    \item for $(i,\chi),\,(i,\chi')\in\supp(\hat X)$ with $\chi\neq\chi'$, we have $d(i,\chi)\ne d(i,\chi')$;
    \item for $(i,\chi),\,(j,\chi')\in\supp(\hat X)$, we have $d(j,\chi')-d(i,\chi)=c_\alpha$ if and only if $\chi'=\chi+e_\alpha$.
  \end{enumerate}
\end{lemma}
\begin{proof}
  For convenience we introduce the notation $Q_1=\{\alpha_1,\ldots,\alpha_n\}$.
  Since $\hat X$ is finite-dimensional and indecomposable, the support quiver $\supp(\hat X)$ is a connected and finite subquiver of $\hat Q$.
  In order to prove the statement, we may assume that $(i',0)\in\supp(\hat X)$ for some $i'\in Q_0$.
  Let $K$ be the maximal length of a path in $\supp(\hat X)$ starting or ending in $(i',0)$ such that the underlying graph of the path has no cycles.
  This implies that, for $(i,\chi)\in\supp(\hat X)$ with $\chi=\sum_{l=1}^n \kappa_l e_{\alpha_l}$, we have $|\kappa_l|\leq K$.

  Set $c_{\alpha_1}=1$ and choose $c_{\alpha_l}$ recursively in such way that 
  \[c_{\alpha_l}\geq 2(K+1)\sum_{k=1}^{l-1} c_{\alpha_k}\]
  for $l=2,\ldots,n$.  
  Then let $f:A_Q\to\ZZ$ be the group homomorphism defined by $f(e_\alpha)=c_\alpha$ for all $\alpha\in Q_1$ and define $d(i,\chi):=f(\chi)$ for $(i,\chi)\in\supp(\hat X)$.

  To check property (1), assume that $d(i,\chi)=d(i,\chi')$ for $\chi=\sum_{l=1}^n \kappa_l e_{\alpha_l}$ and $\chi'=\sum_{l=1}^n \kappa'_l e_{\alpha_l}$.
  This implies
  \[\sum_{l=1}^{n-1}(\kappa_l-\kappa'_l)c_{\alpha_l}=(\kappa'_n-\kappa_n)c_{\alpha_n}.\]
  But we have $|\kappa_l-\kappa'_l|\leq |\kappa_l|+|\kappa'_l|\le 2K$ and thus we obtain 
  \[|\kappa'_n-\kappa_n|c_{\alpha_n}=\left|\sum_{l=1}^{n-1}(\kappa_l-\kappa'_l)c_{\alpha_l}\right|\le 2K\sum_{l=1}^{n-1}c_{\alpha_l} < c_{\alpha_n}.\]
  This inductively yields $\kappa_l=\kappa'_l$ for $l=n,\ldots,1$ by the choice of the $c_{\alpha_l}$ and thus $\chi=\chi'$. 

  By definition, we have $d(j,\chi+e_\alpha)-d(i,\chi)=c_\alpha$ when $(j,\chi+e_\alpha)\in\supp(\hat X)$.
  Now assuming $d(j,\chi')-d(i,\chi)=c_\alpha$, an analogous argument to the one above shows that $\chi'=\chi+e_\alpha$.
\end{proof}

In the following, we say that $d:\supp(\hat X)\to\ZZ$ satisfies the \emph{degree condition} for $\hat X$ if it has the properties of Lemma~\ref{le:degree condition}.
\begin{theorem}
  \label{thm:torusfixedpoints}
  Let $X$ be a representation of $Q$ which can be lifted to a representation $\hat X$ of $\hat Q$ and choose $d:\supp(\hat X)\to\ZZ$ such that it satisfies the degree condition for $\hat X$.
  Then the $\CC^*$-action on $\bigoplus_{i\in Q_0} X_i$ determined by $t.x_{(i,\chi)}=t^{d(i,\chi)}x_{(i,\chi)}$ for $x_{(i,\chi)}\in X_{(i,\chi)}$ induces a $\CC^*$-action on $\Gr_\bfe^Q(X)$ such that
  \[\Gr^Q_\bfe(X)^{\CC^*}\cong\bigsqcup_{\hat\bfe} \Gr^{\hat Q}_{\hat\bfe}(\hat X),\]
  where $\hat\bfe$ runs through all dimension vectors compatible with $\bfe$.
\end{theorem}
\begin{proof}
  A representation $U\in\Gr_{\bfe}(X)$ is a $\CC^*$-fixed point if and only if $t.U=U$ for all $t\in\C^\ast$, i.e.\ $t.U_i=U_i$ for all $i\in Q_0$ and all $t\in\C^\ast$.
  Thus, apart from being a subrepresentation of $X$, each component $U_i$ is a fixed point of the induced $\CC^*$-actions on the usual Grassmannians of vector subspaces $\Gr_{\bfe_i}(X_i)$.
  By Lemma~\ref{le:usualGrass}, this holds precisely when we have a decomposition 
  \[U_i=\bigoplus_{\chi\in A_Q} U_i\cap X_{(i,\chi)}\]
  which is equivalent to $U$ being liftable to the universal abelian covering $\hat Q$.
\end{proof}
The next step is to iterate the $\CC^*$-actions, keeping in mind the following idea: every representation $X$ which lifts to the universal covering quiver also lifts to the universal abelian covering quiver and to the iterated universal abelian covering quivers, i.e.\ to each $\hat Q^{(k)}:=\widehat{\hat Q^{(k-1)}}$ with $\hat Q^{(1)}:=\hat Q$.
Now it is straightforward to check that there exist natural surjective morphisms $f_k:\widetilde Q\to \hat Q^{(k)}$ which become injective on finite subquivers if $k\gg 0$, see also \cite[Section 3.4]{wei}.
Since the support of $X$ is finite as a representation of $\tilde Q$, we can find $k\geq 0$ such that the full subquiver with vertices $\supp(X)\subseteq \hat Q^{(k+1)}_0$ is a tree.
Thus, writing $\hat X^{(\ell)}$ for the lift of $X$ to $\hat Q^{(\ell)}$, there exists a $\CC^*$-action on the vector spaces $\hat X^{(k)}_\beta$ for $\beta\in \hat Q^{(k-1)}_0\times A_{\hat Q^{(k-1)}}$ which induces $\CC^*$-actions on the quiver Grassmannians $\Gr_{\hat\bfe^{(k)}}^{\hat Q^{(k)}}\big(\hat X^{(k)}\big)$ such that the fixed point sets are precisely $\Gr_{\hat\bfe^{(k+1)}}^{\hat Q^{(k+1)}}\big(\hat X^{(k+1)}\big)$.
If we denote these iterated $\CC^*$-fixed points by $\Gr^Q_\bfe(X)^{(k+1)}$, we obtain the following result.
\begin{corollary}
  Let $X$ be a representation which can be lifted to $\tilde Q$.
  Then there exists an iterated torus action such that
  \[\Gr^Q_\bfe(X)^{(k+1)}\cong \bigsqcup_{\hat\bfe^{(k)}} \Gr_{\hat\bfe^{(k)}}^{\hat Q^{(k)}}\big(\hat X^{(k)}\big)^{\CC^*}\cong \bigsqcup_{\hat\bfe^{(k+1)}} \Gr^{\hat Q^{(k+1)}}_{\hat\bfe^{(k+1)}}\big(\hat X^{(k+1)}\big)\cong \bigsqcup_{\tilde\bfe} \Gr^{\tilde Q}_{\tilde\bfe}(\tilde X),\]
  where $\hat\bfe^{(k)},\,\hat\bfe^{(k+1)},\,\tilde\bfe$ run through all dimension vectors compatible with $\bfe$.
\end{corollary}

Define the $F$-polynomial of a representation $X$ by 
\[F_X=\sum_{\bfe\in\NN^{Q_0}}\chi(\Gr_\bfe(X))y^\bfe\in\ZZ[y_i\mid i\in Q_0].\]
\begin{corollary}
  \label{fpoly}
  Let $X$ be a representation which can be lifted to the universal covering quiver.
  \begin{enumerate}
    \item If $\Gr_\bfe(X)$ is smooth and $\Gr^{\hat Q}_\tbfe(\hat X)$ has a cell decomposition, then $\Gr_\bfe(X)$ has a cell decomposition.
    \item We have $F_X=SF_{\tilde X}$ where $SF_{\tilde X}$ is obtained from $F_{\tilde X}$ by applying $S:\ZZ[y_{(i,w)}\mid i\in Q_0,w\in W_Q]\to\ZZ[y_i\mid i\in Q_0]$ given by $S(y_{(i,w)})=y_i$ for all $i\in Q_0$ and $w\in W_Q$.
  \end{enumerate} 
\end{corollary}
An important special case for this is the case of exceptional representations.
In this case the quiver Grassmannians $\Gr_\bfe(X)$ are smooth by \cite[Corollary 4]{cr}.
Moreover, every exceptional representation is a tree module by \cite{rin1} which means that it can be lifted to the universal covering.

\subsection{\texorpdfstring{$\mathbf{GL_n}$}{Gln}-Action on Arrows of \texorpdfstring{$\mathbf{K(n)}$}{K(n)}}

The goal of this section is to prove Theorem~\ref{thm:truncpp} showing that quiver Grassmannians of truncated preprojectives $P_{m+1}^V$ are smooth and only depend on the dimension of the subspace $V\subsetneq \cH_m$.
We begin by observing that $GL_n(\CC)$ naturally acts on the vector space $A_1=\bigoplus_{i=1}^n \CC\alpha_i$ spanned by the arrows of $K(n)$ and hence $GL_n(\CC)$ acts on $\rep K(n)$ via the induced action on the path algebra $A(n)$.
More precisely, given a representation $M=(M_1,M_2,M_{\alpha_i})$ of $K(n)$ and $g=(g_{ij})\in GL_n(\CC)$, the representation $g.M$ is given by $(M_1,M_2,(g.M)_{\alpha_i})$ with $(g.M)_{\alpha_i}=\sum\limits_{j=1}^n g_{ij}M_{\alpha_j}$.
Note that $M$ and $g.M$ are not necessarily isomorphic as representations of $K(n)$.
\begin{lemma}
  \label{le:hom equivariance}
  For any morphism $\theta:M\to N$ between representations $M,N\in\rep K(n)$, the same maps $\theta_1:M_1\to N_1$ and $\theta_2:M_2\to N_2$ give a morphism $\theta^g:g.M\to g.N$ for any $g\in GL_n(\CC)$.
  In particular, the hom-spaces $\Hom(M,N)$ and $\Hom(g.M,g.N)$ are canonically identified for each $g\in G$.
\end{lemma}
\begin{proof}
  Suppose $\theta:M\to N$ is a morphism of representations, i.e.\ $\theta_2\circ M_{\alpha_j}=N_{\alpha_j}\circ\theta_1$ for $1\le j\le n$.
  Then for $g=(g_{ij})\in GL_n(\CC)$ and $1\le i\le n$, we have
  \[\theta_2\circ (g.M)_{\alpha_i}=\sum\limits_{j=1}^n g_{ij}(\theta_2\circ M_{\alpha_j})=\sum\limits_{j=1}^n g_{ij}(N_{\alpha_j}\circ\theta_1)=(g.N)_{\alpha_i}\circ\theta_1\]
  so that $\theta$ also gives a morphism from $g.M$ to $g.N$. 
\end{proof}
\begin{corollary}
  \label{cor:indecomposability}
  For $M\in\rep K(n)$ and $g\in GL_n(\CC)$, the representation $M$ is indecomposable if and only if $g.M$ is indecomposable.
\end{corollary}
\begin{proof}
  The representation $M$ is decomposable if there exists a split epimorphism $\theta:M\onto N$ for some nonzero representation $N$.
  But this occurs exactly when the map $\theta^g:g.M\onto g.N$ is also a split epimorphism.
\end{proof}

While the reflection functors are not $GL_n(\CC)$-equivariant, they do admit the following twisted equivariance.
\begin{lemma}\label{compatibleGSigma}
  For $g\in G$, the reflection functors $\Sigma_i:\rep K(n)\to\rep K(n)$ satisfy $g.\Sigma_i(M)=\Sigma_i(g^{-T}.M)$. 
\end{lemma}
\begin{proof}
  We present all the details for $\Sigma_2$, the proof for $\Sigma_1$ is similar.
  For $M\in\rep K(n)$, we have $\Sigma_2(M)=(M_1,M'_2,M'_{\alpha_i})$, where $M'_2$ fits into the following exact sequence:
  \[\xymatrix@C30pt{0 \ar[r] & M_2 \ar[r]^(.4){\bigoplus\limits_{i=1}^n M_{\alpha_i}} & \bigoplus\limits_{i=1}^n M_1 \ar[r]^\pi & M'_2 \ar[r] & 0}\]
  and $M'_{\alpha_i}=\pi\circ\iota_i$ for $\iota_i:M_1\to\bigoplus_{i=1}^n M_1$ the inclusion of the $i$-th factor.
  Then for $g=(g_{ij})\in GL_n(\CC)$, we have $g.\Sigma_2(M)=(M_1,M'_2,(g.M')_{\alpha_i})$ with 
  \[(g.M')_{\alpha_i}=\sum\limits_{j=1}^n g_{ij}M'_{\alpha_j}=\sum\limits_{j=1}^n g_{ij}(\pi\circ\iota_j)=\pi\circ\sum\limits_{j=1}^n g_{ij}\iota_j=\pi\circ g^T\circ\iota_i.\]
  In particular, we may construct $g.\Sigma_2(M)$ using the following exact sequence:
  \[\xymatrix@C50pt{0 \ar[r] & M_2 \ar[r]^(.45){g^{-T}\circ\bigoplus\limits_{i=1}^n M_{\alpha_i}} & \bigoplus\limits_{i=1}^n M_1 \ar[r]^{\pi\circ g^T} & M'_2 \ar[r] & 0},\]
  i.e. $g.\Sigma_2(M)=\Sigma_2(g^{-T}.M)$.
\end{proof}

Write $\Ind\big(K(n),d\big)$ for the set of isomorphism classes of indecomposable representations of $K(n)$ with dimension vector $d$ for $d\in\NN^{K(n)_0}$. As the $\GL_n(\CC)$-action commutes with the natural base change action, we can define a $\GL_n(\CC)$-action on $\Ind\big(K(n),d\big)$.
\begin{proposition} 
  \label{indecomposables}
  Let $m\geq 1$ and $d(m,r)=\udim P_{m+1}-r\udim P_m$ with $0\leq r\leq n-1$.
  The action of $\GL_n(\CC)$ is transitive on $\Ind\big(K(n),d(m,r)\big)$.
\end{proposition}
\begin{proof}
  For $m=1$, the action of $\GL_n(\CC)$ on $\Gr_d(\CC^n)$ is transitive which shows that it is transitive on $\Ind\big(K(n),d(1,r)\big)$.
  As reflection functors preserve indecomposability and isomorphism classes, Lemma~\ref{compatibleGSigma} gives a commutative diagram as below for each $g\in\GL_n(\CC)$:
  \[\xymatrix{\Ind\big(K(n),d(m,r)\big)\ar[r]^(.45){\Sigma_2}\ar[d]^{g^{-T}}&\Ind\big(K(n),d(m+1,r)\big)\ar[d]^g\\
    \Ind\big(K(n),d(m,r)\big)\ar[r]^(.45){\Sigma_2}&\Ind\big(K(n),d(m+1,r)\big)}.\]
  As the $\GL_n(\CC)$-action is transitive on the left hand side, this implies that it is also transitive on the right hand side.
\end{proof}

\begin{theorem}
  \label{thm:truncpp}
  Fix a dimension vector $\bfe$.
  The quiver Grassmannian $\Gr_\bfe(P_{m+1}^V)$ is smooth for each $V\in \Gr(\cH_m)$.
  Moreover, for $V,W\in \Gr_d(\cH_m)$, we have $\Gr_\bfe(P_{m+1}^V)\cong \Gr_\bfe(P_{m+1}^W)$.
\end{theorem}
\begin{proof}
  Let $g\in\GL_n(\CC)$.
  We first show that $\Gr_\bfe(g.P_{m+1}^V)=\Gr_\bfe(P_{m+1}^V)$.
  Assume that $P_{m+1}^V$ is given by the linear maps $M_{\alpha_i}$.
  Then $g.P_{m+1}^V$ is given by the matrices $(g.M)_{\alpha_i}=\sum\limits_{j=1}^n g_{ij}M_{\alpha_j}$.
  Let $(U_1,U_2)\in\Gr_\bfe(g.P_{m+1}^V)$.
  Then we have  
  \[M_{\alpha_i}(U_2)\subset U_1 \text{ for all } i=1,\ldots,n\Leftrightarrow \left(\sum\limits_{j=1}^n g_{ij}M_{\alpha_j}\right)(U_2)\subset U_1 \text{ for all } i=1,\ldots,n.\]
  Note that we have $g^{-1}.(g.P_{m+1}^V)=P_{m+1}^V$ which shows the non-obvious direction.

  Propositions~\ref{pro:indecomposables} and~\ref{indecomposables} imply that the quiver Grassmannians $\Gr_\bfe(M)$ for $M\in\Ind\big(K(n),\udim P_{m+1}^V\big)$ are all isomorphic for a fixed $\bfe\in\NN^{Q_0}$.
  In other words, we use that all indecomposables with this dimension vector are truncated preprojectives.

  As the indecomposables form a dense open subset of all representations, we found a dense subset whose quiver Grassmannians for a fixed $\bfe$ are isomorphic.
  But now the same proof as for exceptional roots applies in order to show that these quiver Grassmannians need to be smooth, see \cite[Corollary 4]{cr}.
\end{proof}

\subsection{Fibrations of Quiver Grassmannians}
\label{sec:fibrations}

Let $\ses{M}{B}{N}$ be a short exact sequence of representations.
For a fixed dimension vector $\tbfe$, this induces the so-called ``Caldero-Chapoton map'' between quiver Grassmannians
\begin{align*}
  \Psi:\Gr_\tbfe(B)&\to\bigsqcup_{\tbff+\tbfg=\tbfe} \Gr_\tbff(M)\times \Gr_\tbfg(N)\\
  E&\mapsto \big(E\cap M,(E+M)/M\big).
\end{align*}
Following \cite[Section 3]{cc}, any non-empty fiber of $\Psi$ satisfies $\Psi^{-1}(U,W)\cong\mathbb{A}^{\dim\Hom(W,M/U)}$.

For $\cG_{\tbff,\tbfg}:=\Psi^{-1}\big(\Gr_\tbff(M)\times \Gr_\tbfg(N)\big)$, we have 
\begin{equation}
  \label{eq:grassmannian decomposition}
  \Gr_\tbfe(B)=\bigsqcup_{\tbff+\tbfg=\tbfe} \cG_{\tbff,\tbfg}.
\end{equation}
Then $\Psi$ restricts to a map
\[\Psi_{\tbff,\tbfg}:\cG_{\tbff,\tbfg}\to \Gr_\tbff(M)\times \Gr_\tbfg(N).\]
The following results are proven in \cite[Section 3]{cefr}.
For completeness we include a proof of the first.
\begin{lemma}
  \label{le:good partition}
  There exists a total ordering $\preceq$ of the dimension vectors $\tbff$ appearing in the decomposition \eqref{eq:grassmannian decomposition} such that for any fixed $\tbff$ the subset 
  \[\bigsqcup_{ \tbff'\succeq\tbff} \cG_{\tbff',\tbfe-\tbff'}\]
  is closed in $\Gr_\tbfe(B)$.
\end{lemma}
\begin{proof}
  Recall that $\Gr_\tbfe(B)$ is a closed subvariety of $\prod_{i\in Q_0} \Gr_{e_i}(B_i)$.
  For each $i$, the inclusion $M_i\subset B_i$ induces an upper-semicontinuous function $\rho_i:\Gr_{e_i}(B_i)\to\ZZ$ given by $\rho_i(E_i)=\dim(E_i\cap M_i)$.
  In particular, for any fixed $f_i$ the set $\rho_i^{-1}\big(\{f_i,f_i+1,\ldots,e_i\}\big)$ is closed in $\Gr_{e_i}(B_i)$.
  It follows that the lexicographic partial ordering on dimension vectors given by $\tbff\leq\tbff'$ when $f_i\leq f'_i$ for all $i\in Q_0$ gives a closed subset
  \[\bigsqcup_{ \tbff'\geq\tbff} \cG_{\tbff',\tbfe-\tbff'}\subset \Gr_\tbfe(B)\]
  for any fixed $\tbff$.
  Any refinement of this partial order to a total order $\preceq$ will give the claim.
\end{proof}
\begin{theorem}
  \label{vb}
  If $\Im\big(\Psi_{\tbff,\tbfg}\big)$ is locally closed and the fiber dimension of $\Psi_{\tbff,\tbfg}$ is constant over $\Im\big(\Psi_{\tbff,\tbfg}\big)$, then $\Psi_{\tbff,\tbfg}:\cG_{\tbff,\tbfg}\to\Im\big(\Psi_{\tbff,\tbfg}\big)$ is an affine bundle.
  In particular, the existence of a cell decomposition of $\Im\big(\Psi_{\tbff,\tbfg}\big)$ implies a cell decomposition of $\cG_{\tbff,\tbfg}$ in this case.
\end{theorem}

\begin{remark}
  \label{rem:cell decompositions}
  To apply Theorem~\ref{vb} and establish cell decompositions for the $\cG_{\tbff,\tbfg}$, we will find a cell decomposition of $\Gr_\tbff(M)\times\Gr_\tbfg(N)$ such that $\Im\big(\Psi_{\tbff,\tbfg}\big)$ is a union of affine cells, giving an induced cell decomposition of $\Im\big(\Psi_{\tbff,\tbfg}\big)$.

  By Lemma~\ref{le:good partition}, the existence of cell decompositions for each $\cG_{\tbff,\tbfg}$ implies the existence of a cell decomposition for $\Gr_\tbfe(B)$.
  Indeed, we may take the affine cells of all the $\cG_{\tbff,\tbfg}$ as affine cells of $\Gr_\tbfe(B)$ with the natural lexicographic total order induced by taking cells from $\cG_{\tbff',\tbfg'}$ after those of $\cG_{\tbff,\tbfg}$ whenever $\tbff'\succ\tbff$.
\end{remark}

We will apply these results in the setting of the truncated preprojective lifts from Lemma~\ref{le:homdecomposition}.
Fix $m\ge1$.
For fixed subsets $J\subsetneq I\subsetneq\{1,\ldots,n\}$ with $I\setminus J=\{j\}$, Lemma~\ref{le:properties}.(1) provides a short exact sequence 
\begin{equation}
  \label{eq:truncated ses}
  \ses{\tilde P_{m,j}}{\tilde P_{m+1}^J}{\tilde P_{m+1}^I}
\end{equation}
which induces a map between quiver Grassmannians as above for any fixed $\tbfe$:
\begin{equation}
  \label{eq:truncated CC}
  \Psi:\Gr^{\tilde Q}_\tbfe(\tilde P_{m+1}^J)\to\bigsqcup_{\tbff+\tbfg=\tbfe} \Gr^{\tilde Q}_\tbff(\tilde P_{m,j})\times\Gr^{\tilde Q}_\tbfg(\tilde P_{m+1}^I).
\end{equation}
To understand the fibers of this map for $m\ge2$, we will need to make use of another map between quiver Grassmannians coming out of Lemma~\ref{le:special subrepresentations}.
Here we consider the short exact sequence 
\begin{equation}
  \label{eq:tau sequence}
  \ses{\tilde P_m(I,j)}{\tilde P_{m,j}}{\tilde K_m},
\end{equation}
where $\tilde K_m=\tilde K_m(I)=\tau\tilde P_{m+1}^I$ for $m\ge3$ and $\tilde K_2=\tilde K_2(I,j)$ is the representation in \eqref{eq:special case} from Remark~\ref{rem:special case}.
Then, in the same way as above, we obtain the following map for any fixed $\tbff$:
\begin{equation}
  \label{eq:truncated CC2}
  \Phi:\Gr^{\tilde Q}_\tbff(\tilde P_{m,j})\to\bigsqcup_{\tbfs+\tbft=\tbff} \Gr^{\tilde Q}_\tbfs\big(\tilde P_m(I,j)\big)\times\Gr^{\tilde Q}_\tbft(\tilde K_m).
\end{equation}
\begin{proposition}
  \label{quotient}
  For $m\geq 2$ and a non-empty subset $I\subsetneq\{1,\ldots,n\}$.
  The following hold for any $j\in I$:
  \begin{enumerate}
    \item The fiber $\Psi^{-1}(U,\tilde P_{m+1}^I)$ is not empty if and only if $\Ext(\tilde P_{m+1}^I,\tilde P_{m,j}/U)=0$.
    \item The fiber $\Psi^{-1}(U,\tilde P_{m+1}^I)$ is empty if and only if $\Phi(U)=(U,0)$, i.e.\ $U$ is already a subrepresentation of $\tilde P_m(I,j)$.
  \end{enumerate} 
\end{proposition}
\begin{proof}
  Any subrepresentation $U\subset\tilde P_{m,j}$ produces an exact sequence
  \[\xymatrix{\Ext(\tilde P_{m+1}^I,U)\ar[r] &\Ext(\tilde P_{m+1}^I,\tilde P_{m,j})\ar[r] &\Ext(\tilde P_{m+1}^I,\tilde P_{m,j}/U)\ar[r] & 0}.\]
  But note that the fiber $\Psi^{-1}(U,\tilde P_{m+1}^I)$ being non-empty gives rise to a pushout diagram
  \[\xymatrix{0\ar[r]& U\ar[r]\ar[d]\ar@{}[dr]|(.3){\ulcorner} & U'\ar[d]\ar[r]& \tilde P_{m+1}^I\ar[r]\ar@{=}[d]& 0\\
    0\ar[r] &\tilde P_{m,j}\ar[r] &  \tilde P_{m+1}^J\ar[r] & \tilde P_{m+1}^I\ar[r] & 0}\]
  in which the bottom row is not split by Lemma~\ref{le:properties}.(4).
  This implies that the map 
  \[\Ext(\tilde P_{m+1}^I,U)\to\Ext(\tilde P_{m+1}^I,\tilde P_{m,j})\cong\CC\]
  is surjective and thus $\Ext(\tilde P_{m+1}^I,\tilde P_{m,j}/U)=0$.
  This argument can be reversed and thus (1) holds.
  
  Now consider $U\subset\tilde P_{m,j}$ with $U\in\Phi^{-1}(V,W)$.
  This gives rise to the following commutative diagram:  
  \[\xymatrix{
    \Ext(\tilde P_{m+1}^I,V)\ar[r]\ar[d] &\Ext(\tilde P_{m+1}^I,U)\ar[r]\ar[d] &\Ext(\tilde P_{m+1}^I,W)\ar[r]\ar[d] & 0\\
    \Ext\big(\tilde P_{m+1}^I,\tilde P_m(I,j)\big)\ar[r]\ar[d] &  \Ext(\tilde P_{m+1}^I,\tilde P_{m,j})\ar[r]\ar[d] & \Ext(\tilde P_{m+1}^I,\tilde K_m)\ar[r]\ar[d] & 0\\
    \Ext\big(\tilde P_{m+1}^I,\tilde P_m(I,j)/V\big)\ar[r]\ar[d] &\Ext(\tilde P_{m+1}^I,\tilde P_{m,j}/U)\ar[r]\ar[d] &\Ext(\tilde P_{m+1}^I,\tilde K_m/W)\ar[r]\ar[d] & 0 \\
    0&0&0&}\]
  By Corollary~\ref{cor:perpendicular}, we have $\Ext\big(\tilde P_{m+1}^I,\tilde P_m(I,j)\big)=0$ and so $\Ext\big(\tilde P_{m+1}^I,\tilde P_m(I,j)/V\big)=0$ as well.
  This yields the isomorphisms
  \[\Ext(\tilde P_{m+1}^I,\tilde P_{m,j})\cong\Ext(\tilde P_{m+1}^I,\tilde K_m)\quad\text{and}\quad\Ext(\tilde P_{m+1}^I,\tilde P_{m,j}/U)\cong\Ext(\tilde P_{m+1}^I,\tilde K_m/W).\]

  If $W=0$, we get an isomorphism $\Ext(\tilde P_{m+1}^I,\tilde P_{m,j}/U)\cong\Ext(\tilde P_{m+1}^I,\tilde P_{m,j})\cong\CC$ and by part (1) we must have an empty fiber $\Psi^{-1}(U,\tilde P_{m+1}^I)=\varnothing$.

  If $W\neq 0$, $\tilde K_m/W$ is a proper factor of $\tilde K_m$ and we must have $\Ext(\tilde P_{m+1}^I,\tilde K_m/W)=0$.
  Indeed, by Auslander-Reiten theory every non-split morphism $g:\tilde K_m\to\tilde K_m/W$ factors through the middle term $Z$ of the AR-sequence 
  \begin{equation}
    \label{eq:ar2}
    \ses{\tilde K_m}{Z}{\tilde P_{m+1}^I}.
  \end{equation}
  This in particular says the first map of the induced sequence
  \[\Hom(Z,\tilde K_m/W)\to\Hom(\tilde K_m,\tilde K_m/W)\to\Ext(\tilde P_{m+1}^I,\tilde K_m/W)\to\Ext(Z,\tilde K_m/W)\]
  is surjective.
  By Lemma~\ref{le:properties}.(4) and the Auslander-Reiten formulas, we have 
  \[\Ext(\tilde K_m,\tilde K_m)=0\qquad\text{and}\qquad\Ext(\tilde P_{m+1}^I,\tilde K_m)=\CC.\]
  For $m=2$, the identities above follow immediately from the corresponding statements for $\tau\tilde P_3^I$.
  Thus the injective map $\CC=\Hom(\tilde K_m,\tilde K_m)\to\Ext(\tilde P_{m+1}^I,\tilde K_m)$ induced by the Auslander-Reiten sequence \eqref{eq:ar2} is actually bijective and so applying $\Hom(-,\tilde K_m)$ to this sequence yields $\Ext(Z,\tilde K_m)=0$.
  But then $\Ext(Z,\tilde K_m/W)=0$ and so $\Ext(\tilde P_{m+1}^I,\tilde P_{m,j}/U)\cong\Ext(\tilde P_{m+1}^I,\tilde K_m/W)$ must be zero as well.
  By part (1) we see that the fiber $\Psi^{-1}(U,\tilde P_{m+1}^I)$ is non-empty in this case.
\end{proof}

Now we are able to state the following result concerning the fibers of $\Psi$:
\begin{proposition}
  \label{fibers}
  The following hold:
  \begin{enumerate}
    \item For $W\subsetneq\tilde P_{m+1}^I$ and $U\subseteq\tilde P_{m,j}$, we have $\Psi^{-1}(U,W)\cong\AA^{\langle W,\tilde P_{m,j}/U\rangle}$.
    \item If $W=\tilde P_{m+1}^I$, the fiber $\Psi^{-1}(U,W)$ is not empty if and only if $\Phi(U)\neq(U,0)$.
      In this case, we have $\Psi^{-1}(U,W)\cong\AA^{\langle W,\tilde P_{m,j}/U\rangle}$.
  \end{enumerate}
\end{proposition}
\begin{remark}
  Part (1) of Proposition~\ref{fibers} holds equally well when considering the analogues of the Caldero-Chapoton maps $\Psi$ for $K(n)$.
  This follows from Corollary~\ref{cor:base fibers} and Lemma~\ref{le:projective subrepresentations}.
  However, there does not seem to be a reasonable analogue of part (2) when considering Caldero-Chapoton maps for $K(n)$.
\end{remark}
\begin{proof}
  By Lemma~\ref{le:subrep}, any subrepresentation $W\subsetneq\tilde P_{m+1}^I$ is preprojective.
  But the representation $\tilde P_{m,j}/U$ is not preprojective as it is a proper quotient of a preprojective representation unless $U=0$.
  Thus we have $\Ext(W,\tilde P_{m,j}/U)=0$. 
  If $U=0$, we have $\Ext(W,\tilde P_{m,j})=0$ because for dimension reasons every indecomposable direct summand of $W$ is isomorphic to a lift of some $P_l$ with $l\leq m$ and, moreover, $\Ext(P_l,P_m)=0$ for $l\leq m+1$.

  The second statement follows directly from Proposition~\ref{quotient}.
\end{proof}

\begin{corollary}
  For all $\tbff$ and $\tbfg$, the image of $\Psi_{\tbff,\tbfg}$ is open in $\Gr_\tbff(\tilde P_{m,j})\times\Gr_\tbfg(\tilde P_{m+1}^I)$.
\end{corollary}
\begin{proof}
  By Proposition~\ref{fibers}, the map $\Psi_{\tbff,\tbfg}$ is surjective for $\tbfg\ne\udim\tilde P_{m+1}^I$ and there is nothing to show in this case.
  Assume $\tbfg=\udim\tilde P_{m+1}^I$.
  By Proposition~\ref{quotient}, the image of $\Psi_{\tbff,\tbfg}$ consist precisely of those pairs $(U,\tilde P_{m+1}^I)$ for which $\Ext(\tilde P_{m+1}^I,\tilde P_{m,j}/U)=0$.
  But the map $U\mapsto\dim\Ext(\tilde P_{m+1}^I,\tilde P_{m,j}/U)$ is upper semicontinuous so that its minimal value on $\Gr_\tbff(\tilde P_{m,j})$ is its generic value, i.e.\ the image of $\Psi_{\tbff,\tbfg}$ is open.
\end{proof}

\begin{theorem}
  \label{cellscover}
  For $m\geq 1$ and $I\subsetneq\{1,\ldots,n\}$, every quiver Grassmannian $\Gr^{\tilde Q}_\tbfe(\tilde P_{m+1}^I)$ admits a cell decomposition.
\end{theorem}
\begin{proof}
  We work by induction on $m$.
  When $m=1$, the claim is trivial since in this case all quiver Grassmannians are points.
  We establish the result for $m\ge2$ by proving the following more general statement:\bigskip

  \noindent{\bf Claim.}
  For any admissible sequence $\bfI=(I_0,\ldots,I_k)$ with $k<m$ and any subset $J\subset I_0$, the quiver Grassmannian $\Gr^{\tilde Q}_\tbfe(\tilde P_{m+1}^J)$ admits a cell decomposition which is compatible with the unique map $\pi_{m+1}^{J,\bfI}:\tilde P_{m+1}^J\to\tilde P_{m+1}^\bfI$ from Lemma~\ref{le:truncated quotients} in the following sense:
  \begin{itemize}
    \item[($\dagger$)] For any $V\in\Gr^{\tilde Q}_\tbfe(\tilde P_{m+1}^J)$ such that $\pi_{m+1}^{J,\bfI}(V)\neq 0$, we have $\pi_{m+1}^{J,\bfI}(V')\neq 0$ for all $V'\in C_V$, where $C_V$ is the affine cell which contains $V$.
  \end{itemize}
  \vs

  In what follows, we will freely use the notation from Lemma~\ref{le:truncated quotients}.
  We proceed by simultaneous induction on $m$ and reverse induction on $|J|$.
  Fix an admissible sequence $\bfI=(I_0,\ldots,I_k)$ with $k<m$ and $J\subset I_0$.

  To begin, we assume $|J|=n-1$ and thus $J=I_0$.
  This gives $\tilde P_{m+1}^J\cong\tilde P_{m,i_0}^{J_0}$ and so by induction on $m$ the quiver Grassmannian $\Gr^{\tilde Q}_\tbfe(\tilde P_{m+1}^J)=\Gr^{\tilde Q}_\tbfe(\tilde P_{m,i_0}^{J_0})$ admits a cell decomposition so that the compatibility condition ($\dagger$) holds for the unique map $\pi_{m,i_0}^{J_0,\delta\bfI}:\tilde P_{m,i_0}^{J_0}\to\tilde P_{m,i_0}^{\delta\bfI}$.
  But $\tilde P_{m,i_0}^{\delta\bfI}=\tilde P_{m+1}^\bfI$ so that $\pi_{m,i_0}^{J_0,\delta\bfI}$ coincides with the map $\pi_{m+1}^{J,\bfI}$ and thus the condition ($\dagger$) holds for the cell decomposition of $\Gr^{\tilde Q}_\tbfe(\tilde P_{m+1}^J)$.

  Now suppose $|J|<n-1$.
  If $J=I_0$, then we must have $k=0$ and $\tilde P_{m+1}^J=\tilde P_{m+1}^\bfI$ so that the compatibility condition ($\dagger$) with the map $\pi_{m+1}^{J,\bfI}$ is vacuous and any cell decomposition will suffice.
  Thus we may assume $J\subsetneq I_0$.
  
  Choose any subset $I\subset I_0$ with $J\subset I$ and $|I\setminus J|=1$, say $I\setminus J=\{j\}$.
  This gives the short exact sequence \eqref{eq:truncated ses} inducing the maps $\Psi$ and $\Phi$ between quiver Grassmannians from \eqref{eq:truncated CC} and \eqref{eq:truncated CC2}.
  Then by induction on $|I|$, each quiver Grassmannian $\Gr^{\tilde Q}_\tbfg(\tilde P_{m+1}^I)$ has a cell decomposition which is compatible with $\pi_{m+1}^{I,\bfI}$, say 
  \[\Gr^{\tilde Q}_\tbfg(\tilde P_{m+1}^I)=\coprod_{k=1}^r C_k.\]

  When $m=2$, each quiver Grassmannian $\Gr^{\tilde Q}_\tbff(\tilde P_{m,j})$ is just a point.
  For $m\ge3$, each quiver Grassmannian $\Gr^{\tilde Q}_\tbff(\tilde P_{m,j})$ admits a cell decomposition which is compatible with the map $\pi_{m,j}^\bfJ:\tilde P_{m,j}\to\tilde P_{m-1,(j,j)}^I$, where $\bfJ=(\{1,\ldots,\widehat j,\ldots,n\},I)$, say
  \[\Gr^{\tilde Q}_\tbff(\tilde P_{m,j})=\coprod_{\ell=1}^s B_\ell.\] 
 For $m=2$, we write $\Gr^{\tilde Q}_\tbff(\tilde P_{m,j})=B_1$. In view of Remark \ref{rem:cell decompositions}, we need to show that the image of each $\Psi_{\tbff,\tbfg}$ is compatible with these cell decompositions in order to establish a cell decomposition of each $\mathcal G_{\tbff,\tbfg}$ which then gives a cell decomposition of $\Gr_{\tbfe}^{\tilde Q}(\tilde P_{m+1}^J)$.  

  Proposition~\ref{fibers} shows that the fiber of $\Psi_{\tbff,\tbfg}$ over $(U,V)\in B_\ell\times C_k$ is empty exactly when $\tbfg=\udim\tilde P_{m+1}^I$ and one of the following conditions is satisfied
  \begin{itemize}
    \item $m=2$ with $\tbff_{(2,\alpha_j^{-1})}\ne0$ or $\tbff_{(1,e)}\ne0$;
    \item $m\ge3$, with $\pi_{m,j}^\bfJ(U)=0$.
  \end{itemize}
By induction, the compatibility condition ($\dagger$) is true for $\pi_{m,j}^\bfJ$  which shows that either all or none of the fibers over $B_l\times C_k$ are empty. This shows the compatibility with the image. If the fiber is not empty, Theorem \ref{vb} gives that we obtain an affine cell in $\Gr^{\tilde Q}_\tbfe(\tilde P_{m+1}^J)$  of the form $B_\ell\times C_k\times\AA^d$ with~$d=\langle\tbfg,\udim\tilde P_{m,j}-\tbff\rangle$. Altogether this establishes a cell decomposition of $\Gr_{\tbfe}^{\tilde Q}(\tilde P_{m+1}^J)$.

  For $V\in\Gr^{\tilde Q}_\tbfe(\tilde P_{m+1}^J)$ which is contained in such a cell, we have $\pi_{m+1}^{J,I}(V)\in C_k$.
  But then $\pi_{m+1}^{J,\bfI}(V)\ne0$ if and only if $\pi_{m+1}^{I,\bfI}(W)\ne0$ for all $W\in C_k$ and thus $\pi_{m+1}^{J,\bfI}(V')\ne0$ for all $V'$ in this cell $B_\ell\times C_k\times\AA^d$ of $\Gr^{\tilde Q}_\tbfe(\tilde P_{m+1}^J)$. This shows ($\dagger$) for $\pi_{m+1}^{J,\bfI}$.
\end{proof}

\begin{theorem}
  \label{celldec}
  The following hold:
  \begin{enumerate}
    \item Every quiver Grassmannian of any indecomposable preprojective or preinjective representation of $K(n)$ and $\widetilde{K(n)}$ has a cell decomposition.
    \item Let $X\in\rep K(n)$ be an indecomposable representation with dimension vector $(d,e)$ or $(e,d)$, where $(d,e)=\udim P_{m+1}-r\udim P_m$ for $m\geq 1$ and $\leq 0\leq r\leq n-1$.
      Then every quiver Grassmannian $\Gr_\bfe(X)$ has a cell decomposition.
  \end{enumerate}
\end{theorem}
\begin{proof}Both claims follow by considering iterated torus actions taking into account that all quiver Grassmannians under consideration are smooth. For the truncated preprojective representations of $K(n)$ this is Theorem~\ref{thm:truncpp}. For the truncated preprojective representations lifted to the iterated universal abelian covering quivers $\widehat{K(n)}^{(k)}$ for $k\geq 1$, it follows inductively - when applying reflection recursions similar to those stated in (\ref{eq:recursive covers}) -  that these lifts are exceptional representations of $\widehat{K(n)}^{(k)}$, which means that their quiver Grassmannians are again smooth. Note that this is indeed clear for $m= 2$.

  The first part now follows by combining the results of Section~\ref{torusaction} and Theorems~\ref{thm:bb} and~\ref{cellscover}, taking into account that every preprojective representation of $\widetilde{ K(n)}$ is a lift of a preprojective representation (c.f.\ Lemma~\ref{le:lifts of transjectives}).
  The dual version for preinjective representations follows immediately since $\Gr_{\bfe}(P_m)\cong \Gr_{\udim P_m-\bfe}(I_m)$ and $\Gr_{\tbfe}(\tilde P_m)\cong\Gr_{\udim \tilde P_m-\tbfe}(\tilde I_m)$.

  The second part follows in the same way taking into account the initial remark.
\end{proof}

\begin{corollary}
  Let $X\in\rep K(n)$ be a direct sum of exceptional representations.
  Then every quiver Grassmannian $\Gr_{\bfe}(X)$ has a cell decomposition.
  In particular, this is true for all rigid representations of $K(n)$.
\end{corollary}
\begin{proof}
  As every exceptional representation of $K(n)$ is either preprojective or preinjective, we have
  \[X=\bigoplus_{i=1}^rP_{j_i}\oplus\bigoplus_{i=1}^sI_{k_i},\]
  where we assume that $j_i\leq j_{i+1}$ and write
  \[P(r'):=\bigoplus_{i=1}^{r'} P_{j_i},\quad I(s'):=\bigoplus_{i=1}^{s'}I_{k_i}\]
  for $r'\leq r$ and $s'\leq s$. 

  By Theorem~\ref{celldec}, the claim is true for all quiver Grassmannians attached to $P_{j_i}$ or $I_{k_i}$.
  Consider the short exact sequence
  \[\ses{P_{j_{r'+1}}}{P(r'+1)}{P(r')}.\]
  By induction, we can assume that all quiver Grassmannians attached to the two outer terms have a cell decomposition.
  Consider the Caldero-Chapoton map
  \[\Psi_\bfe:\Gr_\bfe\big(P(r'+1)\big)\to\bigsqcup_{\bff+\bfg=\bfe} \Gr_\bff(P_{j_{r'+1}})\times \Gr_\bfg\big(P(r')\big).\]
  The results of \cite[Section 3]{cc} show that $\Psi_{\tbfe}^{-1}(U,W)\cong\AA^{\dim\Hom(W,P_{j_{r'+1}}/U)}$ for all $(U,W)\in \Gr_\bff(P_{j_{r'+1}})\times \Gr_\bfg\big(P(r')\big)$, in particular the fiber is never empty.
  Now every subrepresentation $W$ of $P(r')$ is isomorphic to a direct sum of preprojective representations such that for each direct summand $P_l$ we have $l\leq j_{r'}$.
  Moreover, the quotient $P_{j_{r'+1}}/U$ is not projective if $U\neq 0$ and equal to $P_{j_{r'+1}}$ otherwise.
  Together these yield $\Ext(W,P_{j_{r'+1}}/U)=0$ and thus
  \[\dim\Hom(W,P_{j_{r'+1}}/U)=\langle W, P_{j_{r'+1}}/U\rangle\]
  for all $(U,W)\in \Gr_\bff(P_{j_{r'+1}})\times \Gr_\bfg(P(r'))$.
  Following Theorem~\ref{vb} (see Remark~\ref{rem:cell decompositions}), this already shows that $\Gr_{\bfe}\big(P(r')\big)$ has a cell decomposition for every $1\leq r'\leq r$.
  By duality, the same is true for $\Gr_{\bfe}\big(I(s)\big)$. 

  Finally consider the short exact sequence
  \[\ses{I(s)}{X}{P(r)}.\]
  As every quotient of $I(s)$ is preinjective and as every subrepresentation of $P(r')$ is preprojective, the same argument shows that every quiver Grassmannian attached to $X$ has a cell decomposition.
\end{proof}

As the $F$-polynomial of truncated preprojective representations only depend on the dimension vector, we may denote them by $F_{d(m,r)}$.
The description of the non-empty fibers in Proposition~\ref{fibers} together with Corollary~\ref{fpoly} and Theorem~\ref{celldec} yield the following:
\begin{corollary}
  For $m\geq 1$ and $0\leq r\leq n-2$, we have 
  \[F_{d(m,r)}=F_{d(m,r+1)}F_{\udim P_m}-x^{d(m,r)}F_{d(m-2,r)}.\]
\end{corollary}

\section{Combinatorial Descriptions of Non-Empty Cells}
\label{sec:combinatorics}

\noindent 
In this section, we provide two combinatorial descriptions of the non-empty cells in the quiver Grassmannians of (truncated) preprojective representations of $K(n)$.
The first is quiver theoretic and follows directly from the recursive construction of the cell decomposition from Section~\ref{sec:fibrations}.
The second is the notion of compatible pairs in a maximal Dyck path arising in the computation of rank 2 cluster variables \cite{llz}.
We give a bijection between these which provides a partial geometric explanation for the combinatorial construction of counting polynomials for rank two quiver Grassmannians given in \cite{rupel}.

\subsection{2-Quivers}
\label{sec:2quivers}

The key concept for describing the cell decompositions is the following notion of $2$-quiver which is closely related to certain coefficient quivers of the corresponding representations.
This construction makes use of the support quivers from Examples~\ref{ex:lifted preprojectives} and~\ref{ex:truncated lifts}.
It will turn out that a feature of this construction is that it is blind to the coloring of the different arrows of $\widetilde K(n)$ covering the arrows of $K(n)$.
\begin{definition}
  Let $Q=(Q_0,Q_1)$ be a quiver.
  A subset $\beta\subset Q_0$ is \emph{successor closed in $Q$} if for each $p\in\beta$, the existence of an arrow $\alpha:p\to q$ in $Q_1$ implies $q\in\beta$.

  A \emph{2-arrow} of the quiver $Q$ is an ordered pair $V=\big(\Gamma(1),\Gamma(2)\big)$ of full connected subquivers of $Q$, these will be denoted $V:\Gamma(1)\implies\Gamma(2)$.
  A \emph{2-quiver} is a pair $\cQ=(Q,Q_2)$ consisting of a quiver $Q$ and a collection $Q_2$ of 2-arrows of $Q$.
  Given a 2-quiver $\cQ$, we call a subset $\beta\subset Q_0$ \emph{strong successor closed in $\cQ$} if it is successor closed in $Q$ and for each 2-arrow $V:\Gamma(1)\implies\Gamma(2)$ in $Q_2$ with $\Gamma(1)_0\subset\beta$ we have $\Gamma(2)_0\cap\beta\ne\varnothing$.  
\end{definition}

The following notion of equivalence for $2$-quivers will be useful in the construction of 2-quivers whose strong successor closed subsets label cells in quiver Grassmannians.
Observe that any quiver can be considered as a 2-quiver with no 2-arrows.
\begin{definition}
  \label{def:2equivalence}
  Let $\cQ=(Q,Q_2)$ be a $2$-quiver with a $2$-arrow $V:\Gamma(1)\implies\Gamma(2)$ in $Q_2$ such that one of the following conditions is satisfied:
  \begin{enumerate}
    \item $\Gamma(1)$ has precisely one source $p$; 
    \item $\Gamma(2)$ has precisely one sink $q$;
    \item $\Gamma(1)=\{p\}$ and $\Gamma(2)=\{q\}$.
  \end{enumerate}
  Depending on the condition which is satisfied, we define
  \begin{enumerate}
    \item $\cQ_p$ as the $2$-quiver obtained from $\cQ$ when replacing the $2$-arrow $V$ by a $2$-arrow $V_p:\{p\}\implies\Gamma(2)$;
    \item $\cQ_q$ as the $2$-quiver obtained from $\cQ$ when replacing the $2$-arrow $V$ by a $2$-arrow $V_q:\Gamma(1)\implies\{q\}$;
    \item $\cQ_V$ as the $2$-quiver obtained from $\cQ$ when replacing the $2$-arrow $V$ by a usual arrow $\alpha_V:p\to q$. 
  \end{enumerate}
  This defines a relation on the set of $2$-quivers denoted by $\cQ\to\cQ_p$, $\cQ\to\cQ_q$ and $\cQ\to\cQ_V$ respectively.
  Moreover, it induces an equivalence relation $\sim$ on the set of $2$-quivers when taking the symmetric and transitive closure of this relation. 
\end{definition}

An important consequence of this definition is that the vertex sets of equivalent $2$-quivers coincide, in particular we can formulate the following result.
\begin{lemma}
  \label{lem:2equivalence}
  Let $\cQ=(Q,Q_2)$ and $\cQ'=(Q',Q_2')$ be equivalent $2$-quivers.
  A subset $\beta\subset Q_0$ is strong successor closed in $\cQ$ if and only if it is strong successor closed in $\cQ'$.
\end{lemma}

For the proof of this Lemma the following straightforward observation is essential: in a finite connected quiver which has precisely one source $p$, there exists a path from $p$ to every other vertex of the quiver.
An analogous statement holds if a quiver has precisely one sink. 
\begin{proof}
  By induction, we only need to consider the cases $\cQ'\in\{\cQ_p,\cQ_q,\cQ_V\}$ where one of the conditions of Definition~\ref{def:2equivalence} is satisfied.

  Assume first that $\cQ'\in\{\cQ_p,\cQ_q\}$.
  Then we have $Q=Q'$ from which we immediately see that $\beta\subset Q_0$ is successor closed in $Q$ if and only if $\beta$ is successor closed in $Q'$.
  We only consider the case $\cQ'=\cQ_p$ below, the argument for $\cQ'=\cQ_q$ is dual.

  Let $\beta\subset Q_0$ be strong successor closed in $\cQ$.
  To see that $\beta$ is strong successor closed in $\cQ_p$ it suffices to consider the 2-arrow $V_p:\{p\}\implies\Gamma(2)$.
  Suppose $\{p\}\subset\beta$.
  As $\beta$ is successor closed in $Q$ and $p$ is a source in the connected quiver $\Gamma(1)$, we have $\Gamma(1)_0\subset\beta$ and thus $\Gamma(2)_0\cap\beta\ne\varnothing$, i.e.\ $\beta$ is strong successor closed in $\cQ_p$.
  The reverse implication is immediate since $\{p\}\subset\Gamma(1)_0$.

  Now assume $\cQ'=\cQ_V=\big(Q_V,(Q_V)_2\big)$. 
  Let $\beta\subset Q_0$ be strong successor closed in $\cQ$.
  Since $(Q_V)_2\subset Q_2$, to see that $\beta$ is strong successor closed in $\cQ_V$ we only need to show that $\beta$ is successor closed in $Q_V$.  
  For this it suffices to consider the arrow $\alpha_V:p\to q$ for which that claim is obvious since $p\in\beta$ is equivalent to $\{p\}\subset\beta$ and similarly for $q$.

  Finally, let $\beta\subset(Q_V)_0$ be strong successor closed in $\cQ_V$.
  Since $Q_1\subset(Q_V)_1$, we immediately see that $\beta$ is successor closed in $Q$.
  To see that $\beta$ is strong successor closed in $\cQ$, it suffices to consider the 2-arrow $V:\{p\}\implies\{q\}$ for which the claim is obvious as above.
\end{proof}

\begin{remark}
  Below we will usually apply Lemma~\ref{lem:2equivalence} after performing each of the equivalences from Definition~\ref{def:2equivalence}.
  That is, given a 2-arrow $V:\Gamma(1)\implies\Gamma(2)$ for which $\Gamma(1)$ has a unique source $p$ and $\Gamma(2)$ has a unique sink $q$, we get an equivalent 2-quiver by replacing this 2-arrow with a usual arrow $\alpha_V:p\to q$.
\end{remark}

In the following, we freely use the notation and conventions of Section~\ref{sec:RepK(n)}.
For $m\ge1$, Theorem~\ref{thm:truncpp} shows that up to isomorphism the quiver Grassmannians $\Gr_\bfe(P_{m+1}^V)$ of arbitrary truncated preprojective representations $P_{m+1}^V$ for $V\in\Gr(\cH_m)$ only depend on $\bfe$ and $\udim P_{m+1}^V$.
In particular, fixing $\dim V=r$, we construct a 2-quiver $\cQ_{m+1}^{[r]}$ whose strong successor closed subsets are in one-to-one correspondence with the cells of quiver Grassmannians of~$P_{m+1}^V$.

By Theorem~\ref{celldec}, the cells of the quiver Grassmannians of $P_{m+1}^V$ are in one-to-one correspondence with those attached to any lift $\tilde P_{m+1}^{[r]}$ to $\widetilde{K(n)}$.
Since the choice of $V\in\Gr(\cH_m)$ with $\dim V=r$ is immaterial for understanding the geometry of $\Gr_\bfe(P_{m+1}^V)$, we may fix a particular choice of $V$ and a particular lift to the universal cover.
Indeed, set
\[\tilde P_{m+1}^{[r]}:=
  \begin{cases}
    \tilde P_{m+1}^{\{n,n-1,\ldots,n-r+1\}} & \text{ if $m$ is odd;}\\
    \tilde P_{m+1}^{\{1,2,\ldots,r\}} & \text{ if $m$ is even;}
  \end{cases}\] 
and write $P_{m+1}^{[r]}=G(\tilde P_{m+1}^{[r]})$.
Note that we may allow $r=0$ above and take $\tilde P_{m+1}^{[0]}=\tilde P_{m+1}$, then we write $\cQ_{m+1}$ in place of $\cQ_{m+1}^{[0]}$.
Fixing a choice of lift will allows us to give a concrete description of the 2-quiver~$\cQ_{m+1}^{[r]}$, it will be clear from the construction that making another choice of lift and following an analogous procedure will give a construction of an isomorphic 2-quiver.
In this way, the 2-quiver $\cQ_{m+1}$ should be viewed as a combinatorial shadow of the sequences \eqref{ses1} defining the truncated preprojective representations of $\widetilde{K(n)}$.
In fact, the related sequences \eqref{eq:truncted sequence} will be used together with Lemma~\ref{le:truncated quotients} to recursively construct the 2-quivers $\cQ_{m+1}^{[r]}$.

Each 2-quiver $\cQ_{m+1}^{[r]}$ should be thought of as a combinatorially enhanced version of the coefficient quiver of $\tilde P_{m+1}^{[r]}$ in which certain arrows are upgraded to 2-arrows.
In particular, the vertices and arrows of the quiver $Q_{m+1}^{[r]}$ underlying the 2-quiver $\cQ_{m+1}^{[r]}$ can naturally be associated with vertices and arrows of $\widetilde{K(n)}$.

To begin, we take the $2$-quiver $\cQ_1=\cQ_1^{[0]}$ associated to $\tilde P_1$ to be the quiver $Q_1$ consisting of a single vertex which we associate to the vertex $(1,e)$ of $\widetilde{K(n)}$.
By analogy with the notation of Section~\ref{Lifting}, we define a 2-quiver $\cQ_{1,i}$ for $1\le i\le n$ whose underlying quiver $Q_{1,i}$ has a single vertex which is associated to the vertex $(1,\alpha_i)$ of $\widetilde{K(n)}$.

The 2-quiver $\cQ_2^{[r]}$ associated to $\tilde P_2^{[r]}$ has underlying quiver $Q_2^{[r]}:=Q_1^\sigma\sqcup\coprod\limits_{i=1}^{n-r} Q_{1,i}$, where the single vertex of the quiver $Q_1^\sigma$ is associated to the vertex $(2,e)$ of $\widetilde{K(n)}$, and has 2-arrows (colored red) as in the figure below:
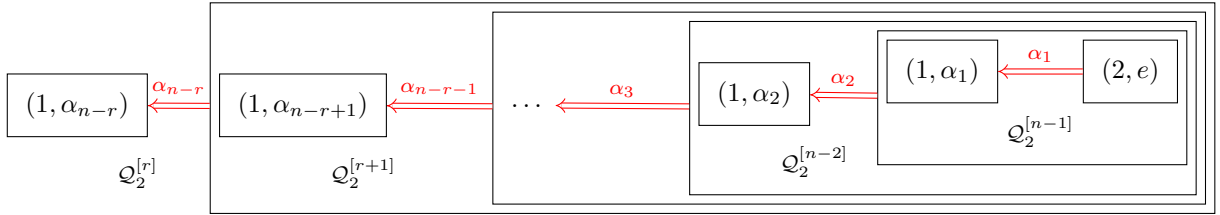
\[\begin{tikzpicture}
  \tikzstyle{box} = [rectangle, minimum width=2cm, minimum height=1cm, text centered, draw=black]
  \draw (0,0) node (Vr) {$(1,\alpha_{n-r})$} +(.8,-.9) node {\footnotesize$ \mathcal Q_2^{[r]}$} +(3,0) node (Vr1) {$(1,\alpha_{n-r+1})$}+(3.8,-.9) node (Lr1) {\footnotesize$ \mathcal Q_2^{[r+1]}$} +(6,0) node (Vx) {$\dots$} +(9,.15) node (V2) {$(1,\alpha_{2})$} +(9.8,-.75) node (L2) {\footnotesize$ \mathcal Q_2^{[n-2]}$}+(11.5,.45) node (V1) {$(1,\alpha_1)$}  +(14,.45) node (V0) {$(2,e)$} +(12.8,-.35) node (L1) {\footnotesize$\mathcal Q_2^{[n-1]}$}; 

  \node (Q0) [draw=black,fit={(V0)}]{};
  \node (P1) [draw=black,fit={(V1)}]{};
  \node (Q1) [draw=black,fit={(P1) (Q0) (L1)}]{};
  \node (P2) [draw=black,fit={(V2)}]{};
  \node (Q2) [draw=black,fit={(Q1) (P2) (L2)}]{};
  \node (Q2b) [draw=black,fit={(Q2) (Vx) }]{};
  \node (Pr1) [draw=black,fit={(Vr1)}]{};
  \node (Q3) [draw=black,fit={(Q2b) (Pr1) (Lr1)}]{};
  \node (Pr) [draw=black,fit={(Vr)}]{};

  \draw[implies-,double equal sign distance,red] (P1)--(Q0) node[pos=.5,above] {\footnotesize$\alpha_1$};
  \draw[-implies,double equal sign distance,red] (Q1)--(P2) node[pos=.5,above] {\footnotesize$\alpha_2$};
  \draw[-implies,double equal sign distance,red] (Q2)--(Vx) node[pos=.5,above] {\footnotesize$\alpha_3$};
  \draw[-implies,double equal sign distance,red] (Q2b)--(Pr1) node[pos=.5,above] {\footnotesize$\alpha_{n-r-1}$};
  \draw[-implies,double equal sign distance,red] (Q3)--(Pr) node[pos=.5,above] {\footnotesize$\alpha_{n-r}$};
\end{tikzpicture}\]
The source and target quivers for each 2-arrow above have been drawn inside a box.
Note that the vertices $(1,\alpha_i)$ are just the $2$-quivers $\cQ_{1,i}$ corresponding to $\tilde P_{1,i}$ and that $\cQ_2^{[t]}$ is a sub-2-quiver of $\cQ_2^{[r]}$ for $t\ge r$.

\begin{remark}
  \label{rem:2-arrows Q2}
  The 2-arrows of $\cQ_2$ should be viewed as a reflection of the isomorphism
  \begin{equation}
    \label{eq:case2}
    \Ext(P_2,P_1)
    \cong\bigoplus_{i=1}^n \Ext(\tilde P_2,\tilde P_{1,i})
    \cong\big\langle (2,e)\xrightarrow{\alpha_i}(1,\alpha_i)\mid i=1,\ldots,n\big\rangle
  \end{equation}
  and the inclusions of $\cQ_2^{[t]}$ in $\cQ_2^{[r]}$ for $t\ge r$ as a reflection of the surjections $\Ext(P_2^{[t]},P_1)\onto\Ext(P_2^{[r]},P_1)$.
  In particular, the isomorphism \eqref{eq:case2} can be used with these surjections to obtain compatible bases for each $\Ext(P_2^{[r]},P_1)$.
\end{remark}
The $2$-quiver $\mathcal Q_2^{[r]}$ given above is clearly equivalent to the support quiver \eqref{coeff} thought of as a 2-quiver with no 2-arrows:
\begin{align}
  \label{coeff}
    \xymatrix@R40pt@C10pt{&(2,e)\ar_{\alpha_1}[ld]\ar^{\alpha_{n-r}}[rd]&\\ (1,\alpha_1)&\cdots&(1,\alpha_{n-r}).}
\end{align}
Thus we may think of $\cQ_2^{[r]}$ as a coefficient quiver of $\tilde P_2^{[r]}$ or of $P_2^V$ for $V\in\Gr(\cH_1)$ with $\dim V=r$.
In order to keep the illustrations and combinatorics simple, we will abuse notation and denote the support quiver \eqref{coeff} by $\mathcal Q_{2}^{[r]}$, working instead with this 2-quiver.
In this way, we may define the translated 2-quivers $\cQ_{2,i}$ (resp.\ $\cQ_{2,i}^{[r]}$) as those obtained from $\cQ_2$ (resp.\ $\cQ_2^{[r]}$) by translating all vertices and ($2$-)arrows by $\alpha_i^{-1}$.

The 2-quiver $\cQ_3^{[r]}$ associated to $\tilde P_3^{[r]}$ has underlying quiver $Q_3^{[r]}:=Q_{2,n}^{[1]}\sqcup\coprod\limits_{i=r+1}^{n-1} Q_{2,i}$.
Note that we are not taking this union as subquivers of $\widetilde{K(n)}$, in particular each quiver $Q_{2,i}$ has a vertex which can be associated to $(1,e)$ in $\widetilde{K(n)}$ but these are not identified in the quiver $Q_3^{[r]}$. 
For $r<s<n$, there is a 2-arrow $V_s:\Gamma_s(1)\implies\Gamma_s(2)$ of $\cQ_3^{[r]}$ given by $\Gamma_s(1)=Q_{2,n}^{[1]}\sqcup\coprod\limits_{i=s+1}^{n-1} Q_{2,i}$ with $\Gamma_s(2)\subset Q_{2,s}$ the subquiver $(2,\alpha_s^{-1})\xrightarrow{\alpha_s} (1,e)$.
By Lemma~\ref{lem:2equivalence}, we obtain an equivalent 2-quiver by replacing each $\Gamma_s(2)$ above with the corresponding sink $(1,e)$ taken as a vertex of $Q_{2,s}$.
By a slight abuse of notation, below we will let $\cQ_3^{[r]}$ denote this equivalent 2-quiver.
Then $\cQ_3^{[r]}$ can be found as a sub-2-quiver of $\cQ_3$ which is constructed recursively by connecting $\cQ_3^{[i]}$ to $\cQ_{2,i}$ for $i=n-1,\ldots, 1$ in the following way: \bigskip
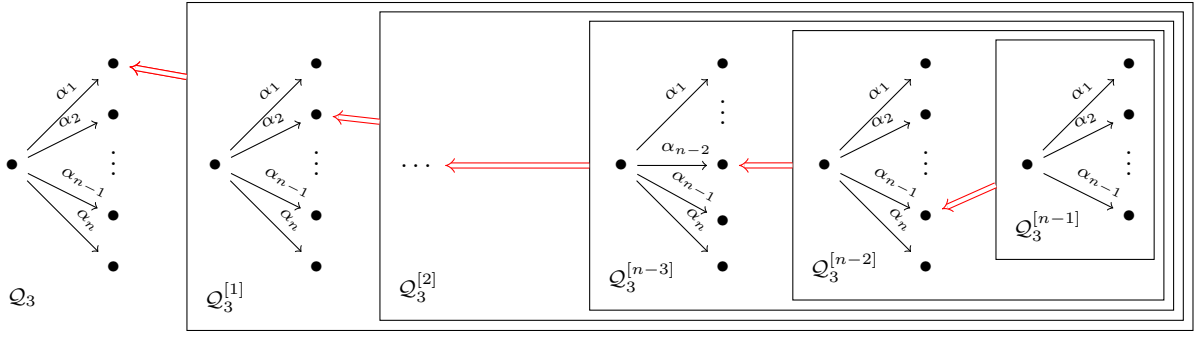
\[
\begin{tikzpicture}[scale=1.35]
  \draw (0,0) node (I1) {$\bullet$} +(1,1) node (J11) {$\bullet$} +(1,.5) node (J12) {$\bullet$} +(1,0.1) node (D1) { $\vdots$}+(1,-.5) node (J1N1) {$\bullet$} +(1,-1) node (J1N) {$\bullet$} +(0.1,-1.3) node {\footnotesize$\mathcal Q_3$}; 
  \draw[->] (I1)--(J11) node[pos=.7,above,sloped]{\scriptsize$\alpha_1$};
  \draw[->] (I1)--(J12) node[pos=.7,above,sloped]{\scriptsize$\alpha_2$};
  \draw[->] (I1)--(J1N1) node[pos=.7,above,sloped]{\scriptsize$\alpha_{n-1}$};
  \draw[->] (I1)--(J1N) node[pos=.7,above,sloped]{\scriptsize$\alpha_{n}$};

  \draw (2,0) node (I2) {$\bullet$} +(1,1) node (J21) {$\bullet$} +(1,.5) node (J22) {$\bullet$}+(1,0.1) node (D2) { $\vdots$} +(1,-.5) node (J2N1) {$\bullet$} +(1,-1) node (J2N) {$\bullet$} +(0.1,-1.3) node (P1) {\footnotesize$\mathcal Q_3^{[1]}$} ; 
  \draw[->] (I2)--(J21) node[pos=.7,above,sloped]{\scriptsize$\alpha_1$};
  \draw[->] (I2)--(J22) node[pos=.7,above,sloped]{\scriptsize$\alpha_2$};
  \draw[->] (I2)--(J2N1) node[pos=.7,above,sloped]{\scriptsize$\alpha_{n-1}$};
  \draw[->] (I2)--(J2N) node[pos=.7,above,sloped]{\scriptsize$\alpha_{n}$};

  \draw (4,0) node (DOTS) {$\dots$}+(0,-1.2) node (P2b) {\footnotesize$\mathcal Q_3^{[2]}$};

  \draw (6,0) node (I3) {$\bullet$} +(1,1) node (J31) {$\bullet$}   +(1,.6) node (D31) { $\vdots$} +(1,0) node (J3N2) {$\bullet$} +(1,-.55) node (J3N1) { $\bullet$} +(1,-1) node (J3N) {$\bullet$} +(0.2,-1.1) node (P3) {\footnotesize$\mathcal Q_3^{[n-3]}$} ; 
  \draw[->] (I3)--(J31) node[pos=.7,above,sloped]{\scriptsize$\alpha_1$};
  \draw[->] (I3)--(J3N2) node[pos=.7,above,sloped]{\scriptsize$\alpha_{n-2}$};
  \draw[->] (I3)--(J3N1) node[pos=.7,above,sloped]{\scriptsize$\alpha_{n-1}$};
  \draw[->] (I3)--(J3N) node[pos=.7,above,sloped]{\scriptsize$\alpha_{n}$};

  \draw (8,0) node (I4) {$\bullet$} +(1,1) node (J41) {$\bullet$} +(1,.5) node (J42) {$\bullet$} +(1,0.1) node (D4) { $\vdots$} +(1,-.5) node (J4N1) {$\bullet$} +(1,-1) node (J4N) {$\bullet$} +(0.2,-1) node (P4) {\footnotesize$\mathcal Q_3^{[n-2]}$}; 
  \draw[->] (I4)--(J41) node[pos=.7,above,sloped]{\scriptsize$\alpha_1$};
  \draw[->] (I4)--(J42) node[pos=.7,above,sloped]{\scriptsize$\alpha_2$};
  \draw[->] (I4)--(J4N1) node[pos=.7,above,sloped]{\scriptsize$\alpha_{n-1}$};
  \draw[->] (I4)--(J4N) node[pos=.7,above,sloped]{\scriptsize$\alpha_{n}$};

  \draw (10,0) node (I5) {$\bullet$} +(1,1) node (J51) {$\bullet$} +(1,.5) node (J52) {$\bullet$} +(1,0.1) node (D5) { $\vdots$} +(1,-.5) node (J5N1) {$\bullet$} +(0.2,-.6) node (P5) {\footnotesize$\mathcal Q_3^{[n-1]}$}   ; 
  \draw[->] (I5)--(J51) node[pos=.7,above,sloped]{\scriptsize $\alpha_1$};
  \draw[->] (I5)--(J52) node[pos=.7,above,sloped]{\scriptsize $\alpha_2$};
  \draw[->] (I5)--(J5N1) node[pos=.7,above,sloped]{\scriptsize $\alpha_{n-1}$};
  \node (Q2) [draw=black,fit={(P5) (I5) (J51) (J52) (J5N1)}]{};
  \draw[-implies,double equal sign distance,red] (Q2)--(J4N1) node[pos=.5,above,sloped]{};
  \node (Q3) [draw=black,fit={(P4) (Q2) (I4) (J41) (J42) (J4N1) (J4N)}]{};
  \draw[-implies,double equal sign distance,red] (Q3)--(J3N2)  node[pos=.5,above,sloped] {};
  \node (Q4) [draw=black,fit={(P3) (Q3) (I3) (J31) (J3N1) (J3N2) (J3N)}]{};
  \draw[-implies,double equal sign distance,red] (Q4)--(DOTS)  node[pos=.5,above] {};
  \node (Q4b) [draw=black,fit={(P2b) (Q4) (DOTS)}]{};

  \node (Q5) [draw=black,fit={(P1) (Q4b) (I2) (J21) (J22) (J2N1) (J2N)}]{};
  \draw[-implies,double equal sign distance,red] (Q5)--(J11)  node[pos=.5,above,sloped,sloped] {};
  \draw[-implies,double equal sign distance,red] (Q4b)--(J22)  node[pos=.5,above,sloped] {};
  \draw[-implies,double equal sign distance,red] (Q5)--(J11)  node[pos=.5,above,sloped] {};
\end{tikzpicture}
\]  
To avoid cluttering the diagram, we did not label the vertices in the illustration.

\begin{remark} 
  Here we justify the definition of the 2-arrows in $\cQ_3^{[r]}$, this discussion will also serve to motivate the choice of 2-arrows for general $\cQ_{m+1}^{[r]}$ and thus we work in that generality.

  For $m\ge2$, we may apply Theorem~\ref{covering} together with Lemma~\ref{le:properties} and the Auslander-Reiten formula to get an isomorphism 
  \begin{align}
    \label{eq:homext} 
    \Ext\big(P_{m+1}^{[r]},P_m\big)
    \cong\bigoplus_{i=r+1}^n \Ext\big(\tilde P_{m+1}^{[r]},\tilde P_{m,i}\big)
    \cong\bigoplus_{i=r+1}^n \Hom\big(\tilde P_{m,i},\tau\tilde P_{m+1}^{[r]}\big).
  \end{align}
  The image of a nonzero map $\tilde P_{m,i}\to\tau\tilde P_{m+1}^{[r]}$ is the representation $\tilde K_m$ from the appropriate sequence \eqref{eq:tau sequence}.
  Such a map is surjective if $m\ge3$ and for $m=2$ has image with support quiver $(2,\alpha_i^{-1})\xrightarrow{\alpha_i} (1,e)$.
  In view of Corollary~\ref{cor:perpendicular}, the sequence \eqref{eq:tau sequence} gives rise to an isomorphism 
  \[\Ext\big(\tilde P_{m+1}^{[r]},\tilde P_{m,i}\big)\cong\Ext\big(\tilde P_{m+1}^{[r]},\tilde K_m\big).\]

  Finally note for $0\le r\le n-2$ that there exists a short exact sequence
  \[\ses{\Hom(P_m,P_m)}{\Ext\big(P_{m+1}^{[r]},P_m\big)}{\Ext\big(P_{m+1}^{[r+1]},P_m\big)}.\]
  Thus a basis of $\Ext\big(P_{m+1}^{[r]},P_m\big)$ can be obtained by taking the last $r$ elements of a basis for $\Ext\big(P_{m+1}^{[n-1]},P_m\big)$.
  The choice of 2-arrows in $\cQ_{m+1}^{[r]}$ should be viewed as a combinatorial shadow of the isomorphisms above.
\end{remark}

We are now ready to define the 2-quivers $\cQ_{m+1}^{[r]}$ for $m\ge3$.
This will be by induction, so assume we have already constructed the 2-quivers $\cQ_m^{[s]}$ for $0\le s\le n-1$ and define the 2-quivers $\cQ_{m,i}^{[s]}:=\alpha_i^{(-1)^{m+1}}.\cQ_m^{[s]}$ for $1\le i\le n$. 
Then we may take the underlying quiver of $\cQ_{m+1}^{[r]}$ to be
\begin{equation}
  \label{eq:recursive 2-quivers}
  Q_{m+1}^{[r]}:=
  \begin{cases}
    Q_{m,n}^{[1]}\sqcup\coprod\limits_{i=r+1}^{n-1} Q_{m,i} & \text{if $m$ is even;}\\
    Q_{m,1}^{[1]}\sqcup\coprod\limits_{i=2}^{n-r} Q_{m,i} & \text{if $m$ is odd.}
  \end{cases}
\end{equation}
For $r<s<n$, there is a 2-arrow $V_s:\Gamma_s(1)\implies\Gamma_s(2)$ of $\cQ_{m+1}^{[r]}$ given by $\Gamma_s(1)=Q_{m+1}^{[s]}\subset Q_{m+1}^{[r]}$ with $\Gamma_s(2)\subset Q_{m,s}$ the subquiver $Q_{m-1}^{[s]}\subset Q_{m-1}=Q_{m-1,(s,s)}\subset Q_{m,s}$.
\begin{remark}
  For $m\geq 3$, the truncated preprojective $\tau\tilde P_{m+1}^{[s]}\cong\tilde P_{m-1}^{[s]}$ can uniquely be found as a quotient of $\tilde P_{m,s}$.
  This is reflected in the structure of the $2$-quivers as we can find $\cQ_{m-1}^{[s]}$ as a subquiver of $\cQ_{m,s}$.
  In the diagrams for 2-quivers given here, this sub-2-quiver can be found at the very right of the $2$-quiver $\mathcal Q_{m,s}$.
\end{remark}

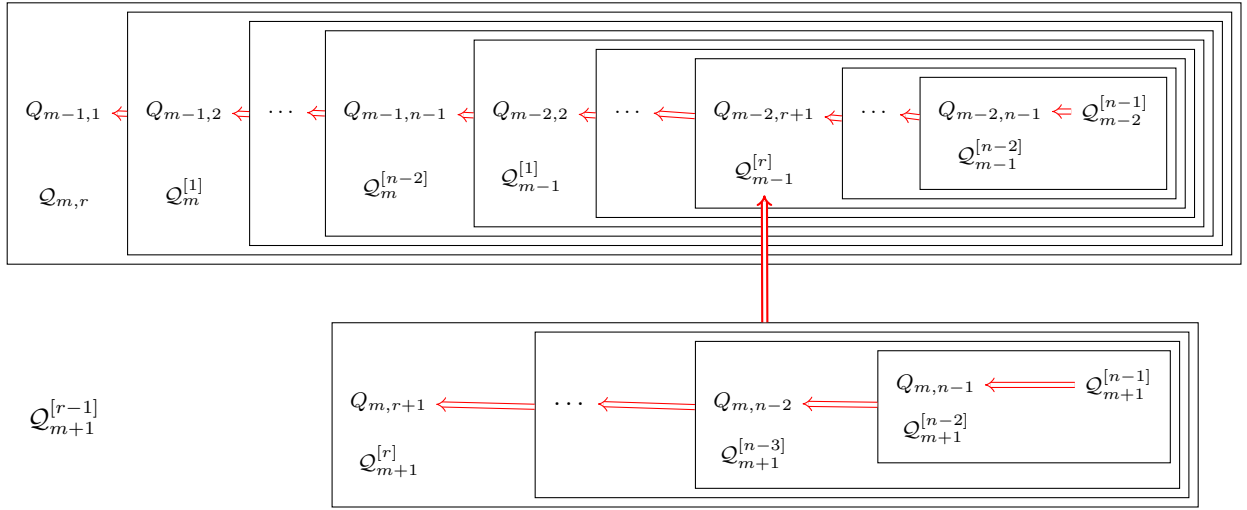
\[
\begin{tikzpicture}[scale=0.97]
\draw (0,4) node (A1) {\footnotesize$Q_{m-1,1}$}
  +(0,-1.2) node (PM) {\footnotesize$\mathcal Q_{m,r}$}
  +(0,-4.2) node {$\cQ_{m+1}^{[r-1]}$}
  +(1.65,0) node (B1) {\footnotesize$Q_{m-1,2}$}
  +(1.65,-1.1) node (P11) {\footnotesize$\mathcal Q_{m}^{[1]}$}
  +(3,0) node (B11) {\footnotesize$\dots$}
  +(4.55,0) node (D1) {\footnotesize$Q_{m-1,n-1}$}
  +(4.55,-1) node (P41) {\footnotesize$\mathcal Q_{m}^{[n-2]}$}
  +(6.4,0) node (E1) {\footnotesize$Q_{m-2,2}$}
  +(6.4,-.9) node (P51) {\footnotesize$\mathcal Q_{m-1}^{[1]}$}
  +(7.75,0) node (A1b) {\footnotesize$\dots$}
  +(9.6,0) node (B1b) {\footnotesize$Q_{m-2,r+1}$}
  +(9.6,-.8) node (P11b) {\footnotesize$\mathcal Q_{m-1}^{[r]}$}
  +(11.125,0) node (B11b) {\footnotesize$\dots$}
  +(12.7,0) node (D1b) {\footnotesize$Q_{m-2,n-1}$}
  +(12.7,-.6) node (P41b) {\footnotesize$\mathcal Q_{m-1}^{[n-2]}$}
  +(14.4,0) node (E1b) {\footnotesize$\mathcal Q^{[n-1]}_{m-2}$};
\node (Q11) [draw=black,fit={(D1b) (E1b) (P41b)}]{};
\node (Q21) [draw=black,fit={(Q11) (B11b)}]{};
\node (Q31)[draw=black,fit={(Q21) (B1b) (P11b)}]{};
\node (Q41)[draw=black,fit={(Q31) (A1b)}]{};
\node (Q51)[draw=black,fit={(Q41) (E1)}]{};
\node (Q61)[draw=black,fit={(Q51) (P41) (D1)}]{};
\node (Q71)[draw=black,fit={(Q61) (B11)}]{};
\node (Q81)[draw=black,fit={(Q71) (P11) (B1)}]{};
\node (Q91)[draw=black,fit={(Q81) (A1)}]{};
\draw[implies-,double equal sign distance,red] (D1b)--(E1b) node[pos=.5,above] {};
\draw[-implies,double equal sign distance,red] (Q11)--(B11b) node[pos=.5,above] {};
\draw[-implies,double equal sign distance,red] (Q21)--(B1b) node[pos=.5,above] {};
\draw[-implies,double equal sign distance,red] (Q31)--(A1b) node[pos=.5,above] {};
\draw[-implies,double equal sign distance,red] (Q41)--(E1) node[pos=.5,above] {};
\draw[-implies,double equal sign distance,red] (Q51)--(D1) node[pos=.5,above] {};
\draw[-implies,double equal sign distance,red] (Q61)--(B11) node[pos=.5,above] {};
\draw[-implies,double equal sign distance,red] (Q71)--(B1) node[pos=.5,above] {};
\draw[-implies,double equal sign distance,red] (Q81)--(A1) node[pos=.5,above] {};

\draw (4.45,0) node (B) {\footnotesize$Q_{m,r+1}$}+(0,-.8) node (P1) {\footnotesize$\mathcal Q_{m+1}^{[r]}$} +(2.5,0) node (B1) {$\dots$} +(5,0) node (C) {\footnotesize$Q_{m,n-2}$} +(5,-.7) node (P3) {\footnotesize$\mathcal Q_{m+1}^{[n-3]}$}+(7.5,0.25) node (D) {\footnotesize$Q_{m,n-1}$}  +(10,0.25) node (E) {\footnotesize$\mathcal Q_{m+1}^{[n-1]}$} +(7.5,-.35) node (P4) {\footnotesize$\mathcal Q_{m+1}^{[n-2]}$}; 
\node (Q1) [draw=black,fit={(D) (E) (P4)}]{};
\node (Q2) [draw=black,fit={(Q1) (C) (P3)}]{};
\node (Q2b)[draw=black,fit={(Q2) (B1) }]{};

\draw[implies-,double equal sign distance,red] (D)--(E) node[pos=.5,above] {};
\draw[-implies,double equal sign distance,red] (Q1)--(C) node[pos=.5,above] {};
\draw[-implies,double equal sign distance,red] (Q2)--(B1) node[pos=.5,above] {};
\draw[-implies,double equal sign distance,red] (Q2b)--(B) node[pos=.5,above] {};

\node (PMK) [draw=black,fit={(Q2b) (B)}]{};
\draw[-implies,double equal sign distance,red,thick] (PMK)--(P11b);
\end{tikzpicture}
\]
As already mentioned two different vertices of $\mathcal Q_{m+1}^{[r]}$ can correspond to the same vertex of $\widetilde{K(n)}$.
Writing dimension vectors $\tbfe\in\NN^{\widetilde{K(n)_0}}$ as $\tbfe=\sum_{q\in \widetilde{K(n)_0}}\tbfe_q\cdot q$, the dimension types $\tbfe(\beta)$ and  $\bfe(\beta)$ of a subset $\beta\subset(\mathcal Q_{m+1}^{[r]})_0$ are defined by
\[\tbfe (\beta)=\sum_{q\in\beta} \tilde q\in\NN^{\widetilde{K(n)_0}}\qquad\text{and}\qquad\bfe(\beta)=G\big(\tbfe(\beta)\big)\in\NN^{K(n)_0},\]
where $\tilde q\in\widetilde{K(n)}_0$ is the vertex which corresponds to $q\in\beta\subset(\mathcal Q_{m+1}^{[r]})_0$.
\begin{theorem}
  \label{thm:2quivercells}
  \mbox{}
  \begin{enumerate}
    \item The affine cells of the cell decomposition of $\Gr_\tbfe\big(\tilde P_{m+1}^{[r]}\big)$ (resp. $\Gr_\bfe\big(P_{m+1}^{[r]}\big)$) induced by Theorem \ref{cellscover} can be labeled by strong successor closed subsets $\beta\subset\cQ_{m+1}^{[r]}$ of dimension type $\tbfe\in\NN^{\widetilde{K(n)_0}}$ (resp. $\bfe\in\NN^{K(n)_0}$) yielding a one-to-one correspondence between cells and strong successor closed subsets.
    \item For $\tbfe\in\NN^{\widetilde{K(n)_0}}$ (resp. $\bfe\in\NN^{K(n)_0}$), the Euler characteristic $\chi(\Gr_\tbfe(\tilde P_{m+1}^{[r]}))$ (resp. $\chi(\Gr_{\bfe }(P_{m+1}^{[r]}))$) is given by the number of strong successor closed subsets of dimension type $\tbfe$ (resp. $\bfe$) of the $2$-quiver of $\cQ_{m+1}^{[r]}$.
  \end{enumerate}
\end{theorem}
\begin{proof}
  The results of Sections~\ref{sec:bb} and \ref{torusaction} imply that the statements in parentheses follow from the respective results for the lifted representations.
  Moreover, the second result follows from the first one.

  We proceed by induction on $m$ and $r$.
  The case of the representation $\tilde P_1$ is trivial.
  We have $(\udim\tilde P_2^{[r]})_q\in\{0,1\}$ for all $q\in\widetilde{K(n)}$, whence the subrepresentations are in one-to-one correspondence with the successor closed subsets of the quiver (\ref{coeff}) which is equivalent to $\mathcal Q_2^{[r]}$.
  Equivalently, we have $\Gr_\tbfe(\tilde P_2^{[r]})\in\{\varnothing,\{\pt\}\}$ so that $\Gr_{\tbfe}(\tilde P_2^{[r]})=\{\pt\}$ if and only if $\tbfe\subset\cQ_2^{[r]}$ is strong successor closed.

  Thus assume that the claim is true for $\tilde P_m$ and $\tilde P_{m+1}^{[r]}$.
  Consider the short exact sequence
  \[\ses{\tilde P_m}{\tilde P_{m+1}^{[r-1]}}{\tilde P_{m+1}^{[r]}}.\]
  We have $(\cQ_{m+1}^{[r-1]})_0=(\cQ_{m+1}^{[r]})_0\coprod{}(\cQ_{m,r})_0$.
  Let $\beta_1\coprod\beta_2\subset (\mathcal Q_{m+1}^{[r-1]})_0\coprod{}(\mathcal Q_{m,r})_0$ be a pair of strong successor closed subsets which gives rise to a pair of non-empty cells by the induction hypothesis.
  By Proposition~\ref{quotient}, the fiber over the pair of cells corresponding to $\beta_1\coprod\beta_2$ is empty if and only if $\beta_2=(\mathcal Q_{m+1}^{[r]})_0$ and $\beta_1\cap(\mathcal Q_{m-1}^{[r]})_0=\varnothing$, where $\mathcal Q_{m-1}^{[r]}$ is considered as a sub-$2$-quiver of $\mathcal Q_{m,r}$.
  But this is precisely the condition on $\beta_1\coprod\beta_2$ to be strong successor closed as $\mathcal Q_{m+1}^{[r]}$ is connected to $\mathcal Q_{m-1}^{[r]}$ by a $2$-arrow. This shows the first claim.
\end{proof}

As an example consider the case $n=3$ and $m=3$. The $2$-quiver of $\tilde P_3$ is given by:\bigskip

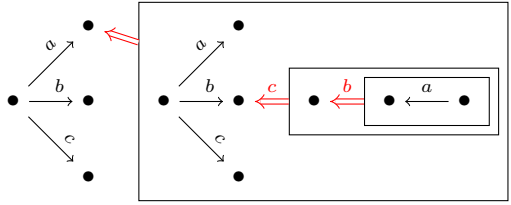
\begin{tikzpicture}[scale=1]
\draw (0,0) node (I1) {$\bullet$} +(1,1) node (J11) {$\bullet$} +(1,0) node (J12) {$\bullet$} +(1,-1) node (J13) {$\bullet$}; 
\draw[->] (I1)--(J11) node[pos=.7,above,sloped]{\scriptsize$a$};
\draw[->] (I1)--(J12) node[pos=.7,above,sloped]{\scriptsize$b$};
\draw[->] (I1)--(J13) node[pos=.7,above,sloped]{\scriptsize$c$};

\draw (2,0) node (I2) {$\bullet$} +(1,1) node (J21) {$\bullet$} +(1,0) node (J22) {$\bullet$} +(1,-1) node (J23) {$\bullet$}; 
\draw[->] (I2)--(J21) node[pos=.7,above,sloped]{\scriptsize$a$};
\draw[->] (I2)--(J22) node[pos=.7,above,sloped]{\scriptsize$b$};
\draw[->] (I2)--(J23) node[pos=.7,above,sloped]{\scriptsize$c$};
\draw (6,0) node (I3) {$\bullet$} (5,0) node (J31) {$\bullet$} (4,0) node (J33) {$\bullet$}; 
\draw[->] (I3)--(J31) node[pos=.5,above,sloped]{\scriptsize$a$};
\node (Q0)[draw=black,fit={(I3) (J31) }]{};
\node (Q1)[draw=black,fit={(Q0) (J33)}]{};
\node (Q2)[draw=black,fit={(Q1) (I2) (J21) (J22) (J23)}]{};

\draw[-implies,double equal sign distance,red] (Q0)--(J33) node[pos=.5,above] {\scriptsize $b$};
\draw[-implies,double equal sign distance,red] (Q1)--(J22) node[pos=.5,above] {\scriptsize $c$};
\draw[-implies,double equal sign distance,red] (Q2)--(J11) node[pos=.5,above] {´};
\end{tikzpicture}
\bigskip

The one of $\tilde P_4$ is given by:\bigskip

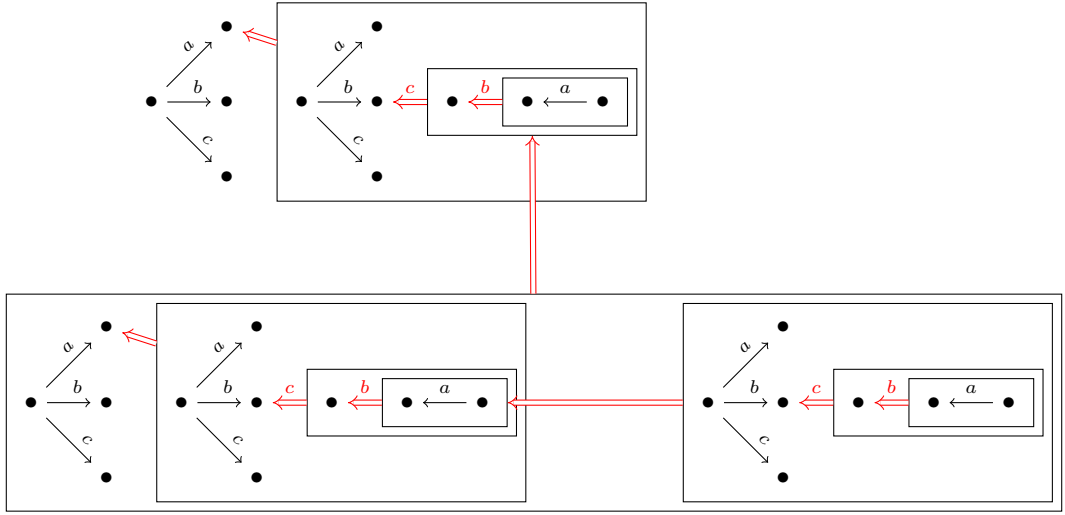
\begin{tikzpicture}[scale=1]
\draw (1.6,0) node (I1) {$\bullet$} +(1,1) node (J11) {$\bullet$} +(1,0) node (J12) {$\bullet$} +(1,-1) node (J13) {$\bullet$}; 
\draw[->] (I1)--(J11) node[pos=.7,above,sloped]{\scriptsize$a$};
\draw[->] (I1)--(J12) node[pos=.7,above,sloped]{\scriptsize$b$};
\draw[->] (I1)--(J13) node[pos=.7,above,sloped]{\scriptsize$c$};

\draw (3.6,0) node (I2) {$\bullet$} +(1,1) node (J21) {$\bullet$} +(1,0) node (J22) {$\bullet$} +(1,-1) node (J23) {$\bullet$}; 
\draw[->] (I2)--(J21) node[pos=.7,above,sloped]{\scriptsize$a$};
\draw[->] (I2)--(J22) node[pos=.7,above,sloped]{\scriptsize$b$};
\draw[->] (I2)--(J23) node[pos=.7,above,sloped]{\scriptsize$c$};
\draw (7.6,0) node (I3) {$\bullet$} +(-1,0) node (J31) {$\bullet$} +(-2,0) node (J33) {$\bullet$}; 
\draw[->] (I3)--(J31) node[pos=.5,above,sloped]{\scriptsize$a$};
\node (Q0)[draw=black,fit={(I3) (J31) }]{};
\node (Q1)[draw=black,fit={(Q0) (J33)}]{};
\node (Q2)[draw=black,fit={(Q1) (I2) (J21) (J22) (J23)}]{};

\draw[-implies,double equal sign distance,red] (Q0)--(J33) node[pos=.5,above] {\scriptsize $b$};
\draw[-implies,double equal sign distance,red] (Q1)--(J22) node[pos=.5,above] {\scriptsize $c$};
\draw[-implies,double equal sign distance,red] (Q2)--(J11) node[pos=.5,above] {};

\draw (0,-4) node (I1b) {$\bullet$} +(1,1) node (J11b) {$\bullet$} +(1,0) node (J12b) {$\bullet$} +(1,-1) node (J13b) {$\bullet$}; 
\draw[->] (I1b)--(J11b) node[pos=.7,above,sloped]{\scriptsize$a$};
\draw[->] (I1b)--(J12b) node[pos=.7,above,sloped]{\scriptsize$b$};
\draw[->] (I1b)--(J13b) node[pos=.7,above,sloped]{\scriptsize$c$};

\draw (2,-4) node (I2b) {$\bullet$} +(1,1) node (J21b) {$\bullet$} +(1,0) node (J22b) {$\bullet$} +(1,-1) node (J23b) {$\bullet$}; 
\draw[->] (I2b)--(J21b) node[pos=.7,above,sloped]{\scriptsize$a$};
\draw[->] (I2b)--(J22b) node[pos=.7,above,sloped]{\scriptsize$b$};
\draw[->] (I2b)--(J23b) node[pos=.7,above,sloped]{\scriptsize$c$};

\draw (6,-4) node (I3b) {$\bullet$} (5,-4) node (J31b) {$\bullet$} (4,-4) node (J33b) {$\bullet$}; 
\draw[->] (I3b)--(J31b) node[pos=.5,above,sloped]{\scriptsize$a$};
\node (Q0b)[draw=black,fit={(I3b) (J31b) }]{};
\node (Q1b)[draw=black,fit={(Q0b) (J33b)}]{};
\node (Q2b)[draw=black,fit={(Q1b) (I2b) (J21b) (J22b) (J23b)}]{};

\draw[-implies,double equal sign distance,red] (Q0b)--(J33b) node[pos=.5,above] {\scriptsize $b$};
\draw[-implies,double equal sign distance,red] (Q1b)--(J22b) node[pos=.5,above] {\scriptsize $c$};
\draw[-implies,double equal sign distance,red] (Q2b)--(J11b) node[pos=.5,above] {};

\draw (9,-4) node (I2c) {$\bullet$} +(1,1) node (J21c) {$\bullet$} +(1,0) node (J22c) {$\bullet$} +(1,-1) node (J23c) {$\bullet$}; 
\draw[->] (I2c)--(J21c) node[pos=.7,above,sloped]{\scriptsize$a$};
\draw[->] (I2c)--(J22c) node[pos=.7,above,sloped]{\scriptsize$b$};
\draw[->] (I2c)--(J23c) node[pos=.7,above,sloped]{\scriptsize$c$};

\draw (13,-4) node (I3c) {$\bullet$} (12,-4) node (J31c) {$\bullet$} (11,-4) node (J33c) {$\bullet$}; 
\draw[->] (I3c)--(J31c) node[pos=.5,above,sloped]{\scriptsize$a$};
\node (Q0c)[draw=black,fit={(I3c) (J31c) }]{};
\node (Q1c)[draw=black,fit={(Q0c) (J33c)}]{};
\node (Q2c)[draw=black,fit={(Q1c) (I2c) (J21c) (J22c) (J23c)}]{};
\node (Q3c)[draw=black,fit={(Q2c) (Q2b) (I1b)}]{};
\draw[-implies,double equal sign distance,red] (Q0c)--(J33c) node[pos=.5,above] {\scriptsize $b$};
\draw[-implies,double equal sign distance,red] (Q1c)--(J22c) node[pos=.5,above] {\scriptsize $c$};
\draw[-implies,double equal sign distance,red] (Q2c)--(Q0b) node[pos=.5,above] {};
\draw[-implies,double equal sign distance,red] (Q3c)--(Q1) node[pos=.6,above] {};
\end{tikzpicture}
\bigskip

\subsection{Compatible Pairs}

For $m\ge1$, let $D_m$ denote the maximal Dyck path in the lattice rectangle with corner vertices $(0,0)$ and $(u_m,u_{m-1})$.  
More precisely, $D_m$ is the lattice path which begins at $(0,0)$, takes East and North steps to end at $(u_m,u_{m-1})$, and never passes above the main diagonal joining $(0,0)$ and $(u_m,u_{m-1})$.  
It is maximal in the sense that any lattice point lying strictly above $D_m$ also lies above the main diagonal.  
The maximal Dyck paths $D_m$, $m\ge1$, exhibit the following recursive structure.
In what follows we assume $n\ge2$.
\begin{theorem}
  \label{th:dyck path recursion}
  \cite[Corollary 2.4]{rupel}
  For $n\ge2$, the maximal Dyck path $D_m$, $m\ge1$, can be constructed recursively as follows:
  \begin{enumerate}
    \item $D_1$ consists of a single horizontal edge;
    \item $D_2$ consists of $n$ consecutive horizontal edges followed by a vertical edge;
    \item $D_m$, $m\ge3$, consists of $n-1$ copies of $D_{m-1}$ followed by a copy of $D_{m-1}$ with its first $D_{m-2}$ removed.
  \end{enumerate}
\end{theorem}

We obtain the following as an immediate consequence.
\begin{corollary}
  \label{cor:short hooks}
  Inside $D_m$, $m\ge2$, there are precisely $u_{m-2}$ vertical edges which are immediately preceded by exactly $n-1$ horizontal edges, all other vertical edges are immediately preceded by exactly $n$ horizontal edges.
\end{corollary}
\begin{proof}
  We work by induction on $m\ge2$.
  The cases $m=2,3$ are immediate from Theorem~\ref{th:dyck path recursion} parts (2) and (3).
  For $m\ge4$, part (3) of Theorem~\ref{th:dyck path recursion} shows by induction that there are $(n-1)u_{m-3}+(u_{m-3}-u_{m-4})=u_{m-2}$ vertical edges which are immediately preceded by exactly $n-1$ horizontal edges.
\end{proof}

For $m\ge1$ and $1\le r\le n-1$, write $D_{m+1}^{[r]}$ for the maximal Dyck path obtained from $D_{m+1}$ by removing the first $r$ copies of $D_m$.
Extending this notation we also set $D_{m+1}^{[0]}:=D_{m+1}$.

For $m\ge1$ and $1\le i\le n-1$, we write $D_{m,i}$ for the $i$-th copy of $D_m$ inside $D_{m+1}$.
Note that for $1\le r\le n-1$, the maximal Dyck paths $D_{m,i}$, $r+1\le i\le n$, naturally identify with subpaths of $D_{m+1}^{[r]}$.
Extending the notation above, for $m\ge2$ and $1\le r\le n-1$, we write $D_{m,i}^{[r]}$ for the Dyck path obtained by removing the first $r$ copies of $D_{m-1}$ from $D_{m,i}$.
\begin{remark}
  For notational convenience, we also set $D_{m,n}^{[1]}:=D_{m+1}^{[n-1]}$ even though there is no maximal Dyck path $D_{m,n}$ identifying with a copy of $D_m$ inside $D_{m+1}$, such notation is justified by Theorem~\ref{th:dyck path recursion}.
  This should be compared with Corollary~\ref{cor:truncated preprojective isomorphism} and Lemma~\ref{le:truncated quotients}.

  This allows to write $D_{m,n}^{[r]}$ for $1\le r\le n-1$ for the terminal subpaths of $D_{m+1}$.
  We also iterate this notation below by identifying $D_{m,r+1}$ with $D_m$ and identifying $D_{m,r+1,n}^{[r+1]}$ with the subpath obtained by removing the first $r+1$ copies of $D_{m-2}$ from a copy of $D_{m-1}$.
\end{remark}

For $m\ge1$, we identify the edges of $D_{m+1}$ with the ordered set $E_{m+1}=\{1,\ldots,u_{m+1}+u_m\}$, where edges of $D_{m+1}$ are taken in the natural order beginning from $(0,0)$.
Let $E_{m+1}=H_{m+1}\sqcup V_{m+1}$, where $H_{m+1}=\{h_1,\ldots,h_{u_{m+1}}\}$ and $V_{m+1}=\{v_1,\ldots,v_{u_m}\}$ denote the horizontal and vertical edges of $D_{m+1}$ respectively.
Following Theorem~\ref{th:dyck path recursion}, we partition the edges as $E_{m+1}=\bigsqcup_{i=1}^n E_{m,i}$, where $E_{m,i}$ denotes the edges of $D_{m,i}$.
The set $E_{m,i}$ is naturally partitioned into its subsets $H_{m,i}$ and $V_{m,i}$ of horizontal and vertical edges.
The edges of $D_{m+1}^{[r]}$ are similarly partitioned as $E_{m+1}^{[r]}=\bigsqcup_{i=r+1}^n E_{m,i}=H_{m+1}^{[r]}\sqcup V_{m+1}^{[r]}$.

Given edges $e,e'\in E_{m+1}$ with $e<e'$, write $ee'$ for the shortest subpath of $D_{m+1}$ containing $e$ and $e'$, in particular $ee$ is the subpath containing the single edge $e$.
\begin{definition}
  \label{def:compatibility}
  For $m\ge1$, a pair of subsets $S_H\subset H_{m+1}$ and $S_V\subset V_{m+1}$ is called \emph{compatible} if: 
  for each pair $(h,v)\in S_H\times S_V$ with $h<v$, there exists an edge $e\in hv$ so that at least one of the following holds
  \begin{equation}
    \label{eq:hgc}
    e\ne v\qquad\text{and}\qquad |he\cap V_{m+1}|=n|he\cap S_H|
  \end{equation}
  or
  \begin{equation}
    \label{eq:vgc}
    e\ne h\qquad\text{and}\qquad |ev\cap H_{m+1}|=n|ev\cap S_V|.
  \end{equation}
  Write $\cC_{m+1}$ for the collection of all pairs $(S_H,S_V)$ which are compatible as above.
\end{definition}
\begin{remark}
  This notion of compatibility extends naturally to the maximal Dyck paths $D_{m+1}^{[r]}$, $1\le r\le n-1$, and trivially to the Dyck path $D_1$.
  Write $\cC_{m+1}^{[r]}$ for the set of all compatible pairs in $D_{m+1}^{[r]}$.
\end{remark}

The recursive structure of the maximal Dyck paths from Theorem~\ref{th:dyck path recursion} gives rise to a recursive characterization of compatible pairs.
\begin{definition}
  \label{def:piecewise compatibility}
  \cite[Definition 3.11]{rupel}
  A pair of subsets $S_H\subset H_{m+1}$ and $S_V\subset V_{m+1}$ is called \emph{piecewise compatible} if, for each $1\le r\le n$, one of the conditions \eqref{eq:hgc} or \eqref{eq:vgc} is satisfied for each pair $(h,v)\in S_H\times S_V$ with $h\in H_{m,i}$ and $v\in V_{m,i}$.
\end{definition}
\begin{remark}
  The notion of piecewise compatibility naturally extends to the maximal Dyck paths $D_{m+1}^{[r]}$, $1\le r\le n-1$.
  Given a compatible pair $(S_H,S_V)$ in $D_{m+1}^{[r]}$, we write $S_H^{[r+1]}=S_H\cap H_{m+1}^{[r+1]}\subset H_{m+1}^{[r]}$ and $S_V^{[r+1]}=S_V\cap V_{m+1}^{[r+1]}\subset V_{m+1}^{[r]}$.
  In particular, the pair $(S_H^{[r+1]},S_V^{[r+1]})$ is compatible in $D_{m+1}^{[r+1]}$.
  We also write $S_{H,i}=S_H\cap H_{m,i}$ and $S_{V,i}=S_V\cap V_{m,i}$ for $r+1\le i\le n-1$.
\end{remark}

To describe precisely when a piecewise compatible pair $(S_H,S_V)$ is compatible we need more notation.
For a horizontal edge $h\in H_{m+1}$ and a subset $S_H\subset H_{m+1}$, write $D(h;S_H)=he$ for the shortest subpath of $D_{m+1}$ for which $|he\cap V_{m+1}|=n|he\cap S_H|$, if no such subpath exists we set $D(h;S_H)=hv_{u_m}$.
The subpath $D(h;S_H)$ is called the \emph{local shadow path} of $h$ with respect to $S_H$.
Similarly, for a vertical edge $v\in V_{m+1}$ and a subset $S_V\subset V_{m+1}$, the \emph{local shadow path} of $v$ with respect to $S_V$ is $D(v;S_V)=ev$ for the shortest subpath of $D_{m+1}$ for which $|ev\cap H_{m+1}|=n|ev\cap S_V|$ and we take $D(v;S_V)=h_1v$ if there does not exist such an edge $e$.
\begin{definition}
  \cite[Definition 3.17]{rupel}
  A horizontal edge $h_i\in H_{m+1}$, $m\ge2$, is called \emph{blocking} for a subset $S_H\subset H_{m+1}$ if $D(h_i;S_H)=h_iv_{u_m}$ and $h_i$ is furthest to the right with this property, i.e.\ the index $i$ is maximal. 

  Suppose $S_H\subset H_{m+1}$ admits a blocking edge $h_i\in H_{m+1}$.
  Then $S_H$ is \emph{left-justified at $h_i$} if there exists $k\ge i$ so that $S_H=\{h_i,h_{i+1},\ldots,h_k\}$.
  The subset $S_H$ is \emph{strongly left-justified at $h_i$} if $S_H$ is left-justified at $h_i$ and $|h_iv_{u_m}\cap V_{m+1}|=n|h_iv_{u_m}\cap S_H|$.

  A subset $S_V\subset V_{m+1}$ is \emph{right-justified with respect to $h_i$} if there exists a vertical edge $v_s\in h_iv_{u_m}$ so that $S_V\cap h_iv_{u_m}=\{v_s,v_{s-1},\ldots,v_{u_m}\}$.
  The subset $S_V$ is \emph{strongly right-justified with respect to $h_i$} if $S_V$ is right-justified with respect to $h_i$ and $D(v_{u_m};S_V)=h_iv_{u_m}$ with $|h_iv_{u_m}\cap H_{m+1}|=n|h_iv_{u_m}\cap S_V|$.
\end{definition}

\begin{theorem}
  \label{th:blocking edge conditions}
  \cite[Theorem 3.20 and Corollary 3.22]{rupel}
  For $m\ge2$, suppose $S_H\subset H_{m+1}$ and $S_V\subset V_{m+1}$ are piecewise compatible. 
  Then the following hold:
  \begin{enumerate}
    \item If $S_H$ does not admit a blocking edge, then $(S_H,S_V)\in\cC_{m+1}$.
    \item Suppose $S_H$ admits a blocking edge $h_i\in H_{m+1}$ and $(S_H,S_V)$ is not compatible.
      Then $S_H$ is left-justified at $h_i$ and $S_V$ is strongly right-justified with respect to $h_i$.
      In addition, the following hold:
      \begin{enumerate}
        \item If $m=2$, then $S_H\cap h_iv_{u_m}=\{h_i\}$.
        \item If $m\ge3$, then $S_H$ is strongly left-justified at $h_i$.
        \item If $m\ge4$, then either $i=1$ or $h_i$ is immediately preceded by a vertical edge in $D_{m+1}$.
      \end{enumerate}
  \end{enumerate}
\end{theorem}

\begin{corollary}
  \label{cor:piecewise not compatible}
  For $m\ge3$ and $0\le r\le n-1$, consider $S_H\subset H_{m+1}^{[r]}$ and $S_V\subset V_{m+1}^{[r]}$ so that $(S_H,S_V)$ is piecewise compatible.
  Assume $\big(S_H^{[r+1]},S_V^{[r+1]}\big)\in\cC_{m+1}^{[r+1]}$.
  Then $(S_H,S_V)$ is not compatible if and only if $H_{m,r+1,n}^{[r+1]}\subset S_{H,r+1}$ and $V_{m+1}^{[r+1]}\subset S_V$.
\end{corollary}
\begin{proof}
  We begin with the reverse implication.
  First note that there are $(n-r)u_m-u_{m-1}$ horizontal edges and $(n-r)u_{m-1}-u_{m-2}$ vertical edges in $D_{m+1}^{[r]}$. 
  It follows that $H_{m,r+1,n}^{[r+1]}\sqcup H_{m+1}^{[r+1]}$ contains $n(n-r)u_{m-1}-nu_{m-2}$ horizontal edges and $V_{m-1}^{[r+1]}\sqcup V_{m+1}^{[r+1]}$ contains $n(n-r)u_{m-2}-nu_{m-3}$ vertical edges (note that $D_{m,r+1,n}^{[r+1]}$ naturally identifies with the Dyck path $D_{m-1}^{[r+1]}$).

  Assuming $H_{m,r+1,n}^{[r+1]}\subset S_{H,r+1}$ and $S_V^{[r+1]}=V_{m+1}^{[r+1]}$, we have $S_V\cap V_{m,r+1,n}^{[r+1]}=\varnothing$ and $S_H\cap H_{m+1}^{[r+1]}=\varnothing$ by piecewise compatibility.
  Let $h\in H_{m+1}^{[r]}$ be the horizontal edge corresponding to the first horizontal edge of $H_{m,r+1,n}^{[r+1]}$.
  Then, since there are $(n-r)u_{m-2}-u_{m-3}$ horizontal edges in $H_{m,r+1,n}^{[r+1]}$, the local shadow path $D(h;S_H)$ contains $n\big((n-r)u_{m-2}-u_{m-3}\big)$ vertical edges and is thus equal to $hv_{u_m}$.
  Similarly, the local shadow path $D(v_{u_m};S_V)$ is also equal to $hv_{u_m}$.
  In particular, neither of the compatibility conditions of Definition~\ref{def:compatibility} are satisfied for the path $hv_{u_m}$ and so $(S_H,S_V)$ is not compatible.

  For the forward implication, we work by induction on $m\ge3$.
  Consider a pair $(S_H,S_V)$ for $D_4^{[r]}$ as above which is not compatible.
  Following Theorem~\ref{th:blocking edge conditions}, write $h\in H_4^{[r]}$ for the blocking edge of $S_H$.
  Then the number of vertical edges in the local shadow path $D(h;S_H)=hv_{u_3}$ must be divisible by $n$. 
  Since $\big(S_H^{[r+1]},S_V^{[r+1]}\big)$ is compatible, we must have $h\in H_{3,r+1}$.
  But observe that $|V_4^{[r]}|=(n-r)n-1$ and so the divisibility condition above implies $h\in H_{3,r+1}^{[n-1]}$.
  But $S_V$ is strongly right-justified with respect to $h$ and thus the number of horizontal edges in $D(v_{u_3};S_V)=hv_{u_3}$ is divisible by $n$.
  Identifying $H_{3,r+1}^{[n-1]}$ with $H_2^{[1]}$, this divisibility condition only occurs when $h$ is the first horizontal edge in $H_2^{[r+1]}\subset H_2^{[1]}$.
  Then by piecewise compatibility, the vertical edge of $H_2^{[1]}$ cannot be an element of $S_V$ and we must have $V_4^{[r+1]}\subset S_V$.
  By piecewise compatibility again, this implies $H_4^{[r+1]}\cap S_H=\varnothing$ and so $D(h;S_H)=hv_{u_3}$ implies $H_2^{[r+1]}\subset S_H$.

  To continue, let $(S_H,S_V)$ be a pair for $D_{m+1}$, $m\ge4$, which is not compatible.
  Write $\varphi:H_m\to V_{m+1}$ for the bijection given by $\varphi(h_i)=v_i$ for $1\le i\le u_m$.
  For any subset $T\subset H_m$, set $\varphi^*(T)=V_{m+1}\setminus\varphi(T)$.
  Clearly, the map $\varphi^*$ gives a bijection between subsets of $H_m$ and subsets of $V_{m+1}$.
  In Section 3.2 of \cite{rupel}, a new pair of subsets $\big((\varphi^*)^{-1}S_V,\Omega^{-1}S_H\big)$ for $D_m$ is given, we refer the reader to \emph{loc.\ cit} for notation.
  By \cite[Proposition 3.10]{rupel}, the pair $\big((\varphi^*)^{-1}S_V,\Omega^{-1}S_H\big)$ is not compatible, but is piecewise compatible by \cite[Proposition 3.16]{rupel}.
  Thus by induction, we must have $H_{m-1,r+1,n}^{[r+1]}\subset(\varphi^*)^{-1}S_V$ and $V_m^{[r+1]}\subset\Omega^{-1}S_H$.
  It follows from piecewise compatibility that $H_m^{[r+1]}\cap(\varphi^*)^{-1}S_V=\varnothing$.
  But then by the definition of $\varphi^*$ we have $V_{m,r+1,n}^{[r+1]}\cap S_V=\varnothing$ and $S_V^{[r+1]}=V_{m+1}^{[r+1]}$ so that $D(v_{u_m};S_V)=hv_{u_m}$ with $h$ as in the first case above.
  Since $(S_H,S_V)$ is not compatible, Theorem~\ref{th:blocking edge conditions} states that $h$ must be the blocking edge for $S_H$ and we must have $D(h;S_H)=hv_{u_m}$.
  But this can only occur if $H_{m,r+1,n}^{[r+1]}\subset S_H$ since $H_{m+1}^{[r+1]}\cap S_H=\varnothing$ by piecewise compatibility.
\end{proof}

The following result is an immediate consequence of the combinatorial construction of rank 2 cluster variables \cite{llz} and the categorification of these variables using representations of $K(n)$ \cite{cc,ck}.
\begin{theorem}
  \label{th:combinatorial euler}
  \cite{llz}
  For each $m\ge1$ and $\bfe\in\ZZ_{\ge0}^2$, we have 
  \[\chi\big(\Gr_\bfe(P_m)\big)=\Big|\big\{(S_H,S_V)\in\cC_m:|S_H|=u_m-e_1,|S_V|=e_2\big\}\Big|.\]
\end{theorem}

Our goal is to give a geometric explanation for this by showing that the compatible pairs provide a natural labeling for the cells of $\Gr_\bfe(P_{m+1})$ found in Theorem~\ref{celldec}.
In fact, we will see more: that the cells of quiver Grassmannians $\Gr_\bfe(P_{m+1}^V)$ for truncated preprojectives $P_{m+1}^V$ are also naturally labeled by compatible pairs.
We accomplish this by providing a bijection between the compatible pairs as in Theorem~\ref{th:combinatorial euler} and the successor closed sets of vertices in the 2-quivers $\cQ_{m+1}^{[r]}$ used in Theorem~\ref{thm:2quivercells} to describe the non-empty cells.
\begin{theorem}
  \label{th:compatible cells}
  For $m\ge1$ and $V\in\Gr(\cH_m)$ or $V=0$, each quiver Grassmannian $\Gr_\bfe(P_{m+1}^V)$ admits a cell decomposition with affine cells labeled by compatible pairs in the maximal Dyck path $D_{m+1}^{[r]}$, where $r=\dim V$.
\end{theorem}
\begin{proof}
  The recursive construction of the 2-quivers $\cQ_{m+1}^{[r]}$ provides a natural ordering of the vertices in the underlying quiver $Q_{m+1}^{[r]}$.
  Indeed, when considering the recursive construction of the 2-quiver $Q_{m+1}^{[r]}$ from \eqref{eq:recursive 2-quivers}, we order the component quivers $Q_{m,i}$ and $Q_{m,*}^{[1]}$ naturally according to their indices so that $Q_{m,*}^{[1]}$ comes last.
  This provides a bijection of these vertices with the edges of $D_{m+1}^{[r]}$ whereby vertices covering the vertex 1 (resp.\ vertex 2) of $K(n)$ correspond to horizontal edges (resp.\ vertical edges) of $D_{m+1}^{[r]}$.

  Given a strong successor closed subset $\beta\subset(Q_{m+1}^{[r]})_0$, we define a pair of subsets $S_H(\beta)\subset H_{m+1}^{[r]}$ and $S_V(\beta)\subset V_{m+1}^{[r]}$ as follows: a vertical edge $v\in V_{m+1}^{[r]}$ is in $S_V(\beta)$ exactly when the corresponding vertex of $Q_{m+1}^{[r]}$ is in $\beta$ while a horizontal edge $h\in H_{m+1}^{[r]}$ is in $S_H(\beta)$ exactly when the corresponding vertex of $Q_{m+1}^{[r]}$ is \emph{not} in $\beta$.
  Then Corollary~\ref{cor:piecewise not compatible} shows that under this bijection a subset $\beta\in(Q_{m+1}^{[r]})_0$ is strong successor closed in~$\cQ_{m+1}^{[r]}$ if and only if the corresponding pair of subsets $\big(S_H(\beta),S_V(\beta)\big)$ is compatible.
  Applying Theorem~\ref{thm:2quivercells} completes the proof.
\end{proof}

The results of \cite{rupel} provide a stronger statement than Theorem~\ref{th:combinatorial euler}.
Indeed, the compatible pairs are shown to compute the counting polynomials of these quiver Grassmannians $\Gr_\bfe(P_{m+1})$ over a finite field (these coincide with their Poincar\'e polynomials in this case).  
We conjecture that the torus action on $\Gr_\bfe(P_{m+1})$ can be chosen to provide a geometric explanation of this result.
\begin{conjecture}
  \label{conj:cell dimensions}
  For $m\ge1$ and $V\in\Gr(H_m)$ or $V=0$, there exists a torus action on $\Gr_\bfe(P_{m+1}^V)$ such that the dimension of the cell labeled by a compatible pair $(S_H,S_V)$ in the maximal Dyck path $D_{m+1}^{[r]}$, where $r=\dim V$, is given by $\overline{\gamma}_{S_H,S_V}=\sum\limits_{e<e'\in E_{m+1}^{[r]}}\overline{\gamma}_\omega(e,e')$ for
  \begin{equation*}
    \overline{\gamma}_\omega(e,e')=
    \begin{cases}
      -n & \text{if $e\in S_H$ and $e'\in S_V$;}\\ 
      1 & \text{if $e\in S_H$ and $e'\in H_{m+1}^{[r]}\setminus S_H$;}\\ 
      1 & \text{if $e\in V_{m+1}^{[r]}\setminus S_V$ and $e'\in S_V$;}\\ 
      0 & \text{otherwise.}\\ 
    \end{cases}
  \end{equation*}
\end{conjecture}


\end{document}